\documentclass[11pt,reqno]{amsart}
\usepackage{amssymb}
\usepackage{amsmath, amssymb, amsthm}
\usepackage{mathrsfs}
\usepackage{soul,color}

\newtheorem{theorem}{Theorem}[section]
\newtheorem{definition}{Definition}[section]
\newtheorem{lemma}{Lemma}[section]
\newtheorem{remark}{Remark}[section]
\newtheorem{proposition}{Proposition}[section]

\numberwithin{equation}{section}

\newcommand{\supp}{{\mathrm {supp}}}

\begin{document}
\title[Global well-posedness of the compressible  Navier-Stokes Equations]{Global well-posedness of  regular solutions   to the three-dimensional  isentropic  compressible  Navier-Stokes Equations  with degenerate viscosities and vacuum}

\author{Zhouping Xin}
\address[Z.P. Xin]{The Institute of Mathematical Sciences, The Chinese University of Hong Kong, Shatin, N.T., Hong Kong.}
\email{\tt zpxin@ims.cuhk.edu.hk}

\author{Shengguo Zhu}
\address[S. G.  Zhu]{Mathematical Institute, University of Oxford,  Oxford OX2 6GG, UK; School of Mathematical Sciences, Monash University, Clayton, 3800, Australia; The Institute of Mathematical Sciences, The Chinese University of Hong Kong, Shatin, N.T., Hong Kong.}
\email{\tt zhushengguo@alumni.sjtu.edu.cn}

\begin{abstract}
In this paper,  the Cauchy problem for  the three-dimensional (3-D) isentropic  compressible  Navier-Stokes equations with degenerate viscosities is considered. By  introducing  some new variables and making use of the   `` quasi-symmetric hyperbolic"--``degenerate elliptic" coupled structure to control the behavior of the fluid velocity,  we prove the global-in-time well-posedness of regular solutions with vacuum for a class of smooth initial data that are of small density but possibly large velocities. Here the initial  mass density is required to  decay to zero in the far field, and the spectrum of the Jacobi matrix of the initial velocity are all positive.  The result here applies to a class of degenerate density-dependent viscosity coefficients, is  independent of the  BD-entropy, and seems to be    the first on the global existence of smooth solutions which have large velocities and contain vacuum state for such degenerate system in three space dimensions.

\end{abstract}

\date{Jun. 6th, 2018}

\maketitle

\section{Introduction}\ \\

The time evolution of the mass  density $\rho\geq 0$ and the velocity $u=\left(u^{(1)},u^{(2)},u^{(3)}\right)^\top$ of a general viscous isentropic
compressible   fluid occupying a spatial domain $\Omega\subset \mathbb{R}^3$ is governed by the following isentropic  compressible
Navier-Stokes equations (\textbf{ICNS}):
\begin{equation}
\label{eq:1.1}
\begin{cases}
\rho_t+\text{div}(\rho u)=0,\\[4pt]
(\rho u)_t+\text{div}(\rho u\otimes u)
  +\nabla
   P =\text{div} \mathbb{T}.
\end{cases}
\end{equation}
Here, $x=(x_1,x_2,x_3)\in \Omega$, $t\geq 0$ are the space and time variables, respectively. In considering the polytropic gases, the constitutive relation, which is also called the equations of state, is given by
\begin{equation}
\label{eq:1.2}
P=A\rho^{\gamma}, \quad \gamma> 1,
\end{equation}
where $A>0$ is  an entropy  constant and  $\gamma$ is the adiabatic exponent. $\mathbb{T}$ denotes the viscous stress tensor with the  form
\begin{equation}
\label{eq:1.3}
\begin{split}
&\mathbb{T}=\mu(\rho)\Big(\nabla u+(\nabla u)^\top\Big)+\lambda(\rho) \text{div}u\,\mathbb{I}_3,\\
\end{split}
\end{equation}
 where $\mathbb{I}_3$ is the $3\times 3$ identity matrix,

\begin{equation}
\label{fandan}
\mu(\rho)=\alpha  \rho^\delta,\quad \lambda(\rho)=\beta  \rho^\delta,
\end{equation}
for some  constant $\delta\geq 0$,
 $\mu(\rho)$ is the shear viscosity coefficient, $\lambda(\rho)+\frac{2}{3}\mu(\rho)$ is the bulk viscosity coefficient,  $\alpha$ and $\beta$ are both constants satisfying
 \begin{equation}\label{10000}\alpha>0,\quad \text{and} \quad   2\alpha+3\beta\geq 0.
 \end{equation}


Let $\Omega=\mathbb{R}^3$. We look for smooth solutions,
$(\rho(t,x),u(t,x))$ to the Cauchy problem for (\ref{eq:1.1})-(\ref{10000}) with the  initial data and far field behavior:
\begin{equation} \label{initial}
(\rho,u)|_{t=0}=(\rho_0(x)\geq 0,\ u_0(x)) \quad \text{for} \quad  x\in \mathbb{R}^3,
\end{equation}
\begin{equation} \label{far}
\rho(t,x)\rightarrow 0,\quad  \text{as}\quad  |x|\rightarrow \infty \quad \ \  \ \text{for} \quad \   t\geq 0.
\end{equation}

In the theory of gas dynamics,  the compressible Navier-Stokes equations can be derived from the Boltzmann equations through the Chapman-Enskog expansion, cf. Chapman-Cowling \cite{chap} and Li-Qin \cite{tlt}. Under some proper physical assumptions,  the viscosity coefficients and the heat conductivity coefficient $\kappa$ are not constants but functions of the absolute temperature $\theta$ such as:
\begin{equation}
\label{eq:1.5g}
\begin{split}
\mu(\theta)=&a_1 \theta^{\frac{1}{2}}F(\theta),\quad  \lambda(\theta)=a_2 \theta^{\frac{1}{2}}F(\theta), \quad \kappa(\theta)=a_3 \theta^{\frac{1}{2}}F(\theta)
\end{split}
\end{equation}
for some   constants $a_i$ $(i=1,2,3)$.  Actually in  \cite{chap} for the cut-off inverse power force model, if the intermolecular potential varies as $r^{-a}$,
 where $ r$ is intermolecular distance,  then in (\ref{eq:1.5g}): $F(\theta)=\theta^{b}$ with $ b=\frac{2}{a} \in [0,+\infty)$.
  In particular, for Maxwellian molecules,  $a = 4$ and $b=\frac{1}{2}$;
while for elastic spheres,  $a=\infty $ and $b=0$. As a typical model whose $F$ is not a power function of $\theta$, the Sutherland's model is well known where
\begin{equation}\label{sutherland}
F(\theta)=\frac{\theta}{\theta+s_0},\quad (s_0>0: \ \text{Sutherland's constant}).
\end{equation}
According to Liu-Xin-Yang \cite{taiping}, if  we restrict the gas flow to be isentropic, such dependence is inherited through the laws of Boyle and Gay-Lussac:
$$
P=R\rho \theta=A\rho^\gamma,  \quad \text{for \ \ constant} \quad R>0,
$$
i.e., $\theta=AR^{-1}\rho^{\gamma-1}$,  and one finds that the viscosity coefficients are functions of the density. In this paper, we will focus on the  cut-off inverse power force model whose viscosities have the  forms shown in (\ref{fandan}).  The corresponding conclusion for the  isentropic  flow of the Sutherland's model will be shown in our forthcoming paper Xin-Zhu \cite{xz2}.

Throughout this paper, we adopt the following simplified notations, most of them are for the standard homogeneous and inhomogeneous Sobolev spaces:
\begin{equation*}\begin{split}
 & \|f\|_s=\|f\|_{H^s(\mathbb{R}^3)},\quad |f|_p=\|f\|_{L^p(\mathbb{R}^3)},\quad \|f\|_{m,p}=\|f\|_{W^{m,p}(\mathbb{R}^3)},\\[4pt]
& |f|_{C^k}=\|f\|_{C^k(\mathbb{R}^3)},\quad \|f\|_{L^pL^q_t}=\| f\|_{L^p([0,t]; L^q(\mathbb{R}^3))},\quad  \|f\|_{X_1 \cap X_2}=\|f\|_{X_1}+\|f\|_{X_2},\\[4pt]
& D^{k,r}=\{f\in L^1_{loc}(\mathbb{R}^3): |f|_{D^{k,r}}=|\nabla^kf|_{r}<+\infty\},\quad D^k=D^{k,2},  \quad  \int_{\mathbb{R}^3}  f \text{d}x  =\int f,\\[4pt]
& \Xi=\{f\in L^1_{loc}(\mathbb{R}^3): |\nabla f|_\infty+\|\nabla^2 f\|_2<+\infty\},\quad \|f\|_{\Xi}=|\nabla f|_\infty+\|\nabla^2 f\|_2, \\[4pt]
&  \|f(t,x)\|_{T,\Xi}=\|\nabla f(t,x)\|_{L^\infty([0,T]\times\mathbb{R}^3)}+\|\nabla^2 f(t,x)\|_{L^\infty([0,T]; H^2(\mathbb{R}^3))}.
\end{split}
\end{equation*}
 A detailed study of homogeneous Sobolev spaces  can be found in \cite{gandi}.

There is a lot of literature on the  well-posedness  of solutions to the problem \eqref{eq:1.1}--\eqref{initial} in multi-dimensional space. For 3-D constant viscous flow ($\delta=0$ in (\ref{fandan})) with $\inf_x {\rho_0(x)}>0$, it is well-known that the local existence of classical solutions has been obtained by a standard Banach fixed point argument by Nash \cite{nash}, which has been extended to be a global one by Matsumura-Nishida \cite{mat} for initial data close to a nonvacuum equilibrium in some Sobolev space $H^s$ ($s>\frac{5}{2}$).  Hoff \cite{HD} studied the global weak solutions with strictly positive initial density and temperature for discontinuous initial data.
However, these approaches do not work when $\inf_x {\rho_0(x)}=0$, which occurs when some physical requirements are imposed, such as finite total initial mass  and energy in the whole space. The main breakthrough for the well-posedness of  solutions  with generic data and vacuum  is due to Lions \cite{lions},  where he established the global existence of weak solutions with finite energy to the isentropic compressible flow provided that $\gamma >\frac{9}{5}$ (see also Feireisl-Novotn\'{y}-Petzeltov\'{a} \cite{fu1} for the case $\gamma>\frac{3}{2}$). However, the
uniqueness problem of these weak solutions  is widely open due to their fairly low regularities. Recently,  Xin-Yan \cite{zwy} proved that any classical solutions of viscous non-isentropic compressible fluids without heat conduction will blow up in finite time, if the initial data has an isolated mass group.

For density-dependent viscosities ($\delta>0$ in (\ref{fandan})) which degenerate at vacuum,  system \eqref{eq:1.1} has received extensive attentions in recent years. In this case, the strong degeneracy of the momentum equations in  \eqref{eq:1.1}  near vacuum creats serious difficulties for well-posedness of both  strong and weak solutions. A   mathematical entropy function  was proposed by Bresch-Desjardins \cite{bd3} for
 $\lambda(\rho)$ and $\mu(\rho)$ satisfying the relation
\begin{equation}\label{bds}\lambda(\rho)=2(\mu'(\rho)\rho-\mu(\rho)),
\end{equation}
which offers an estimate
$
\mu'(\rho)\nabla \sqrt{\rho}\in L^\infty ([0,T];L^2(\mathbb{R}^d))
$
provided that 
$ \mu'(\rho_0)\nabla \sqrt{\rho_0}\in L^2(\mathbb{R}^d)$ for any $d\geq 1$.
This observation plays an important role in the development of the global existence of weak solutions with vacuum for system (\ref{eq:1.1}) and some related models, see Bresh-Dejardins \cite{bd6}, Li-Xin \cite{lx}, Mellet-Vassuer \cite{vassu} and some other interesting results, c.f. \cite{hailiang,  tyc2}. However, the regularities and uniquness of such weak solutions remain open especially in multi-dimensional cases.

In this paper, we study the global   well-posedness  of regular solutions to the system $(\ref{eq:1.1})$ with density-dependent viscosities given in  (\ref{fandan}) and initial data such that  $\inf_x {\rho_0(x)}=0$. Then the analysis of the   degeneracies  in momentum equations $(\ref{eq:1.1})_2$ requires some special attentions. The major concerns are:
\begin{enumerate}
\item The degeneracy of  time evolution in momentum equations $(\ref{eq:1.1})_2$. Note that   the leading coefficient of  $u_{t}$  in momentum equations vanishes at vacuum, and this leads to infinitely many ways to define velocity (if it exits) when
vacuum appears.  Mathematically, this degeneracy leads to a difficulty that  it is hard to find a reasonable way to extend the definition of velocity into vacuum region. For constant viscosities, a remedy was suggested  by Cho-Choe-Kim  (for example \cite{CK3}), where they imposed  initially a {\it compatibility condition}
\begin{equation*}
-\text{div} \mathbb{T}_{0}+\nabla P(\rho_0)=\sqrt{\rho_{0}} g,\quad \text{for\  \ some} \ \ g\in L^2(\mathbb{R}^3),
\end{equation*}
which,  is roughly  equivalent to the $L^2$-integrability of $\sqrt{\rho}u_t (t = 0)$, and  plays a key role in deducing that $(\sqrt{\rho}u_t, \nabla u_t)\in L^\infty([0,T_*]; L^2(\mathbb{R}^3))$  for a  short  time $T_{*}>0$. Then  they established  successfully   the local well-posedness of smooth solutions with non-negative  density  in $\mathbb{R}^3$, which,  recently,  has been shown to be   a global one with small energy but large oscillations by Huang-Li-Xin \cite{HX1} in  $\mathbb{R}^3$ and Li-Xin \cite{lx2d} in $\mathbb{R}^2$.

\item The strong degeneracy of the elliptic operator $\text{div}\mathbb{T}$ caused by vacuum for $\delta>0$.  In  \cite{CK3, HX1} for $\delta=0$, the uniform ellipticity of the Lam\'e operator $L$ defined by
\begin{equation*}Lu =-\alpha\triangle u-(\alpha+\beta)\nabla \mathtt{div}u
\end{equation*}
plays an essential role in the high order regularity estimates on $u$. One can use standard elliptic theory to estimate $|u|_{D^{k+2}}$  by the $D^k$-norm of all other terms in momentum equations.  However, for $\delta>0$,  viscosity coefficients vanish as density function connects to vacuum continuously.  This degeneracy makes it difficult to adapt the approach in the constant viscosity case in \cite{CK3,HX1} to the current case.

\item The strong nonlinearity for the variable coefficients of  the viscous term due to  $\delta>0$.
 It should be pointed out here that unlike the case of  constant viscosities, despite the weak regularizing effect on solutions, the elliptic part  $\text{div} \mathbb{T}$ will also cause some troubles in the high order   regularity estimates.  For example, to establish some uniform a priori estimates independent of the lower bound of density   in $H^3$ space, the key is to handle the extra nonlinear terms such as
$$\text{div}\big( \nabla^k\rho^\delta \mathbb{S}(u)\big) \ \ \text{for} \ \  \mathbb{S}(u)=\alpha(\nabla u+(\nabla u)^\top)+\beta\mathtt{div}u\mathbb{I}_3,$$
where $k=0,1, 2, 3$. Therefore,  much attentions need to be  paid in order to control these strong nonlinearities, especially for establishing the global well-posedenss of classical solutions.

\end{enumerate}

Recently, there have some interesting works  to overcome these  difficulties mentioned above.
In Li-Pan-Zhu \cite{sz3} for the case $\delta=1$, the degeneracies of the time evolution and the viscosity can be transferred to the possible singularity of the special source term  $\Big(\frac{\nabla \rho}{\rho}\Big) \cdot  \mathbb{S}(u)$.  Based on this observations, via establishing a uniform a priori estimates in $L^6\cap D^1\cap D^2$ space for the  quantity $\frac{\nabla \rho}{\rho}$, the existence of the unique  local classical solution in   2-D space (see also Zhu \cite{sz33} for 3-D case)   to  (\ref{eq:1.1}) has been obtained under the assumptions
$$
\rho_0(x)\rightarrow 0,\quad  \text{as}\quad  |x|\rightarrow \infty,
$$
which also applies to the 2-D shallow water equations.

However, this result only allows vacuum at the far field, and the corresponding problem with vacuum appearing in some open sets, or even at a single point is still unsolved.
Later, via introducing a proper class of solutions and establishing the uniform weighted  a priori estimates for the higher order term $\rho^{\frac{\delta-1}{2}}\nabla^4u$,
the same authors \cite{sz333,sz34} gave the existence of 3-D local classical solutions to system (\ref{eq:1.1}) for the cases $$1<\delta\leq \min\{3, (\gamma+1)/2\}.$$ Some other interesting results can also be seen in Ding-Zhu \cite{ding} and  Li-Pan-Zhu \cite{hyp}.

\subsection{Symmetric formulation} In order to deal with the issues mentioned above, we first need to analyze the  structure of the momentum equations $(\ref{eq:1.1})_2$ carefully, which can be decomposed into hyperbolic, elliptic and source  parts  as follows:
\begin{equation}
\label{eq:1.1mo}
\begin{split}
\displaystyle
\underbrace{ \rho \big(u_t+ u\cdot  \nabla u\big)+\nabla P }_{\text{Hyperbolic}}&=\underbrace{- \rho^\delta Lu}_{\text{Elliptic} }+\underbrace{ \nabla \rho^\delta \cdot  \mathbb{S}(u)}_{\text{Source}}.
\end{split}
\end{equation}

For smooth solutions $(\rho,u)$ away from vacuum,  these equations  could be written into
\begin{equation}\label{c-2}
\begin{split}
&\underbrace{ u_t+u\cdot\nabla u +\frac{A\gamma}{\gamma-1}\nabla \rho^{\gamma-1}-\frac{\delta}{\delta-1} \nabla \rho^{\delta-1} \cdot \mathbb{S}(u)}_{\text{Principal \ order}}=\underbrace{ -\rho^{\delta-1}Lu}_{\text{Higher order}}.
\end{split}
\end{equation}
Note that if  $\rho$ is smooth enough, one could  pass to the limit as   $\rho\rightarrow 0$ on both sides of (\ref{c-2}) and  formally have
\begin{equation}\label{zhenkong1}
 u_t+ u\cdot \nabla u=0\quad  \text{as } \quad  \rho(t,x)=0.
\end{equation}
So (\ref{c-2})-(\ref{zhenkong1}) indicate that  the velocity $u$ could be governed by a nonlinear degenerate  parabolic system  when vacuum appears in some open sets or at the far field.

Recently, in Geng-Li-Zhu \cite{zhu}, via introducing two new quantities:
$$
\varphi=\rho^{\frac{\delta-1}{2}},\quad \text{and} \quad \phi=\sqrt{\frac{4A\gamma}{(\gamma-1)^2}}\rho^{\frac{\gamma-1}{2}},
$$
 system \eqref{eq:1.1} can be rewritten into a system that consists of a transport equation for $\varphi$, and
a  ``quasi-symmetric hyperbolic"--``degenerate elliptic" coupling system  for $U=(\phi, u)^\top$:
\begin{equation}\label{li46}
\begin{cases}
\displaystyle
\underbrace{\varphi_t+u\cdot\nabla\varphi+\frac{\delta-1}{2}\varphi\text{div} u=0}_{\text{Transport equations}},\\[8pt]
\displaystyle
\underbrace{U_t+\sum_{j=1}^3A_j(U) \partial_j U}_{\text{Symmetric hyperbolic}}=\underbrace{- \varphi^{2}\mathbb{{L}}(u)}_{\text{Degenerate elliptic}}+\underbrace{ \mathbb{{H}}(\varphi)  \cdot \mathbb{{Q}}(u)}_{\text{First order source}},
 \end{cases}
\end{equation}
where  $\partial_j U=\partial{x_j}U$,
\begin{equation} \label{sseq:5.2qq}
\begin{split}
\displaystyle
A_j(U)=&\left(\begin{array}{cc}
u^{(j)}&\frac{\gamma-1}{2}\phi e_j\\[8pt]
\frac{\gamma-1}{2}\phi e_j^\top &u^{(j)}\mathbb{I}_3
\end{array}
\right),\quad j=1,2,3,\quad
\mathbb{{L}}(u)=\left(\begin{array}{c}0\\
Lu\\
\end{array}\right),\\[6pt]
 \mathbb{{H}}(\varphi)=&\left(\begin{array}{c}0\\
\nabla \varphi^{2}\\
\end{array}\right), \quad
 \mathbb{{Q}}(u)=\left(\begin{array}{cc}
0 & 0\\
0 &Q(u)
\end{array}\right), \quad
 Q(u)=\frac{\delta}{\delta-1}\mathbb{S}(u),
\end{split}
\end{equation}
and  $e_j=(\delta_{1j},\delta_{2j},\delta_{3j})$ $(j=1,2,3)$ is the Kronecker symbol satisfying $\delta_{ij}=1$, when $ i=j$ and $\delta_{ij}=0$, otherwise.
Based on this  reformulation,    a local well-posedness for arbitrary $\delta>1$ in some non-homogenous Sobolev spaces has been obtained  for the compressible degenerate viscous flow in  \cite{zhu}, which could be shown as follows:
\begin{theorem}\cite{zhu}\label{th1} Let $\gamma>1$ and $\delta>1$.  If  initial data $( \rho_0, u_0)$ satisfy
\begin{equation}\label{th78}
\rho_0\geq 0,\quad \Big(\rho^{\frac{\gamma-1}{2}}_0, \rho^{\frac{\delta-1}{2}}_0, u_0\Big)\in H^3,
\end{equation}
then there exists a time $T_*>0$ independent of the viscosity coefficients and a unique regular solution $(\rho, u)$ in $[0,T_*]\times \mathbb{R}^3$ to the  Cauchy problem  (\ref{eq:1.1})-(\ref{far}), where the regular solution $(\rho,u)$  satisfies this problem in the sense of distribution and:
\begin{equation*}\begin{split}
&(\textrm{A})\quad  \rho\geq 0, \  \rho^{\frac{\delta-1}{2}}\in C([0,T_*]; H^3), \ \  \rho^{\frac{\gamma-1}{2}}\in C([0,T_*]; H^3); \\
& (\textrm{B})\quad u\in C([0,T_*]; H^{s'})\cap L^\infty([0,T_*]; H^{3}),\quad  \rho^{\frac{\delta-1}{2}}\nabla^4 u \in L^2([0,T_*]; L^2),\\
&(\textrm{C})\quad u_t+u\cdot\nabla u =0\quad  \text{as } \quad  \rho(t,x)=0,
\end{split}
\end{equation*}
where $s'\in[2,3)$ is an arbitrary constant.
\end{theorem}

\subsection{Singularity formation} We consider first whether the local regular solution in Theorem \ref{th1} can be extended globally in time. In contrast to the classical theory for the constant viscosity case, we show the following somewhat surprising phenomenon that such an extension is impossible if the velocity field decays to zero as $t\rightarrow +\infty$ and the initial total momentum is non-zero. More pricisely,  let
\begin{align*}
&\mathbb{P}(t)=\int \rho u \quad \textrm{(total momentum)}.
\end{align*}
\begin{theorem}\label{th2}
 Let $1< \min\{\gamma,  \delta\} \leq 2$.  Assume   $|\mathbb{P}(0)|>0$.
Then there is no global  regular  solution $(\rho,u)$ obtained in Theorem  \ref{th1}
satisfying the following decay
\begin{equation}\label{eq:2.15}
\limsup_{t\rightarrow +\infty} |u(t,x)|_{\infty}=0.
\end{equation}

\end {theorem}

 Second, in the presence of vacuum region, there is another possibility of finite singularity defined as follows.
\begin{definition}\label{bugers}
The non-empty open set $V \subset \mathbb{R}^3$ is  called a hyperbolic singularity set of $(\rho_0,u_0)(x)$, if $V$ satisfies
\begin{equation} \label{eq:12131ss}
\begin{cases}
\displaystyle
\rho_0(x)=0, \ \forall \ x\in V;\\[8pt]
\displaystyle
 Sp(\nabla u_0) \cap \mathbb{R}^-\neq\  \emptyset,\   \forall  \ x \in V,
\end{cases}
\end{equation}
where $Sp(\nabla u_0(x))$ denotes the spectrum of the matrix $\nabla u_0(x)$.
\end{definition}
By exploring  the hyperbolic structure in (\ref{eq:1.1}) in vacuum region, one  can  confirm exactly that hyperbolic singularity set does generate singularities from local regular solutions in finite time.
\begin{theorem}\label{th3}\
Let $\gamma>1$ and $\delta>1$.  If the initial data $(\rho_0,u_0)(x)$ have a non-empty hyperbolic singularity set $V$, then
the  regular solution $(\rho, u)(t,x)$ on $\mathbb{R}^3\times[0, T_m]$ obtained in Theorem \ref{th1} with maximal existence time $T_m$ blows up in finite time, i.e.,
$
T_m<+\infty.
$
\end{theorem}

\subsection{Global-in-time well-posedness of smooth solutions} We now come to the major task  to construct  global smooth solutions for the system (\ref{eq:1.1}).  Due to  Theorems \ref{th2}-\ref{th3},  one needs to identify a class of initial data and a proper energy space to   avoid the  two singularity mechanisms shown above.

Let $ \widehat{u}=(\widehat{u}^{(1)},\widehat{u}^{(2)},\widehat{u}^{(3)})^\top$ be the solution in $ [0,T]\times \mathbb{R}^3$  of the following Cauchy problem:
\begin{equation} \label{eq:approximation}
\displaystyle
 \widehat{u}_t+ \widehat{u}\cdot \nabla  \widehat{u}=0, \quad
 \widehat{u}(t=0,x)=u_0(x)
\end{equation}
for $ x\in \mathbb{R}^3$. We give the definition of regular solutions considered in this paper:
\begin{definition}\label{d1}
 Let $T> 0$ be a finite constant. A solution $(\rho,u)$ to the  Cauchy problem  (\ref{eq:1.1})-(\ref{far}) is called a regular solution in $ [0,T]\times \mathbb{R}^3$ if $(\rho,u)$ satisfies this problem in the sense of distribution and:
\begin{equation*}\begin{split}
&(\textrm{A})\quad  \rho\geq 0, \  \Big(\rho^{\frac{\delta-1}{2}},\rho^{\frac{\gamma-1}{2}}\Big)\in C([0,T]; H^{s'}_{loc})\cap L^\infty([0,T]; H^3);\\
& (\textrm{B})\quad u- \widehat{u}\in C([0,T]; H^{s'}_{loc})\cap L^\infty([0,T]; H^{3}),\quad  \rho^{\frac{\delta-1}{2}}\nabla^4 u \in L^2([0,T]; L^2);\\
&(\textrm{C})\quad u_t+u\cdot\nabla u =0\quad  \text{as } \quad  \rho(t,x)=0,
\end{split}
\end{equation*}
where $s'\in[2,3)$ is an arbitrary constant.
\end{definition}

The main result on the global well-posedness of  regular solutions   to the three-dimensional  isentropic  compressible  Navier-Stokes equations  with degenerate viscosities and vacuum could be stated as follows.

\begin{theorem}\label{thglobal} Let parameters  $(\gamma,\delta, \alpha,\beta)$ satisfy
\begin{equation}\label{canshu}
\gamma>1,\quad \delta>1, \quad \alpha>0, \quad 2\alpha+3\beta\geq 0,
\end{equation}
and any one of the following  conditions $(P_1)$-$(P_3)$:
\begin{itemize}
\item[$(\rm P_1)$] $2\alpha+3\beta=0$;\quad $(\rm P_2)$ $\delta\geq 2\gamma-1$; \quad $(\rm P_3)$ $\delta=\gamma$.
\end{itemize}
If the   initial data $( \rho_0, u_0)$ satisfy
\begin{itemize}
\item[$(\rm A_1)$] $\rho_0\geq 0$\  and  \ $\Big\|\rho^{\frac{\gamma-1}{2}}_0\Big\|_3+ \Big\|\rho^{\frac{\delta-1}{2}}_0\Big\|_3\leq D_0(\gamma, \delta, \alpha,\beta,  A, \kappa, \|u_0\|_{\Xi})$,
\item[$(\rm A_2)$] $u_0\in \Xi$ and  there exists a constant  $\kappa>0$ such that ,
$$
\text{Dist}\big(\text{Sp}( \nabla u_0(x)), \mathbb{R}_{-} \big)\geq \kappa \quad \text{for\  all}\quad  x\in \mathbb{R}^3,
$$
\end{itemize}
where $D_0>0$ is some constant depending on $(\alpha,\beta, \delta,  A,\gamma, \kappa, \|u_0\|_{\Xi})$,
then for any $T>0$,  there exists a   unique  regular solution $(\rho, u)$ in $[0,T]\times \mathbb{R}^3$  to the  Cauchy problem  (\ref{eq:1.1})-(\ref{far}). Particularly,  when condition $(P_2)$ holds, the smallness assumption on $\rho^{\frac{\delta-1}{2}}_0$ could be removed.

Moreover, if $1<\min(\gamma,\delta)\leq 3$,  $(\rho, u)$  is a classical solution to  the   Cauchy problem  (\ref{eq:1.1})-(\ref{far}) in $[0,T]\times \mathbb{R}^3$.

\end{theorem}
\begin{remark}\label{r1} The conditions $(A_1)$-$(A_2)$ identify a class of admissible initial data that
provides unique solvability to the  Cauchy problem  (\ref{eq:1.1})-(\ref{far}).
Such initial data contain  the following examples:
$$
\rho_0(x)=\frac{\epsilon_1}{(1+|x|)^{2\sigma_1}}, \quad \epsilon_1 g^{2\sigma_2}(x), \quad \epsilon_1 \exp\{-x^2\}, \quad \frac{\epsilon_1 |x|}{(1+|x|)^{2\sigma_3}},
$$
where $\epsilon_1>0$ is a sufficiently small constant,  $0\leq g(x)\in C^3_c(\mathbb{R}^3)$,   and
$$\sigma_1>\frac{3}{2}\max\Big\{\frac{1}{\delta-1}, \frac{1}{\gamma-1}\Big\}, \  \sigma_2>3\max\Big\{\frac{1}{\delta-1}, \frac{1}{\gamma-1}\Big\},\  \sigma_3>\frac{3}{2}\max\Big\{\frac{1}{\delta-1}, \frac{1}{\gamma-1}\Big\}+\frac{1}{2};
$$
$$
u_0=Ax+\textbf{b}+\epsilon_2 f(x),
$$
where $A$ is a $3\times 3$ constant matrix whose  eigenvalues are all greater than $2\kappa$,  $\epsilon_2>0$ is a sufficiently small constant,  $f\in \Xi$ and $\textbf{b}\in \mathbb{R}^3$ is a constant vector.
\end{remark}

\begin{remark}\label{P0}
It is worth pointing out that, for any $\gamma>1$, even $2\alpha+3\beta \neq 0$, that is to say $(P_1)$ fails, we can still deal with the corresponding well-posedness problem for a relatively wide class of $(\alpha,\beta,\delta)$ under the following condition:
 \begin{itemize}
\item[$(\rm P_0)$] $0<M_1=\frac{2\alpha+3\beta}{2\alpha+\beta}<\frac{3}{2}-\frac{1}{\delta}$ and \\[4pt]
$
M_2=-3\delta+1+\frac{1}{2}\left(\frac{(\delta-1)^2}{4(2\alpha+\beta)}+\frac{4\delta^2(2\alpha+\beta)}{(\delta-1)^2}M^2_1+2M_1\delta\right)<-1
$,
\end{itemize}
which can be seen in \S 5.

Next we indicate that the set of parameters   $(\alpha,\beta,\delta)$   satisfying  $(P_0)$ must be non-empty. First, for fixed $\delta>1$,  let
$$
\alpha= a_1\eta, \quad \beta=a_2 \eta,
$$
where $\eta>0$, $a_1>0$ and $2a_1+3a_2\geq 0$ are all constants.
Thus
$M_1=\frac{2a_1+3a_2}{2a_1+a_2}
$.
Then one can  adjust the values of $a_1$ and $a_2$ to ensure that $M_1<\frac{3}{2}-\frac{1}{\delta}$ holds.

Second, let $
x=(\delta-1)^2/(4(2a_1+a_2)\eta)
$,
and  consider the following function:
$$
F(x)=x+\frac{1}{x}\delta^2M^2_1+2M_1\delta-6\delta+4,
$$
whose minimum value is
$$F(\delta M_1)=4M_1\delta-6\delta+4<0$$ due to $M_1<\frac{3}{2}-\frac{1}{\delta}$. Thus, one needs only  to choose $\eta$  to ensure that $x$ belongs to a small neighbor of $\delta M_1$ and $F<0$, which is equivalent to $M_2<-1$.
\end{remark}

\begin{remark}\label{gllt}
Without generality, we can assume that $u_0(0)=0$, which can be achieved by  the following Galilean transformation:
\begin{equation*}
t'=t,\quad x'=x+u_0(0)t,\quad 
   \rho'(t',x')=\rho(t,x),\quad 
    \widehat{u}'(t',x')= \widehat{u}(t,x)-u_0(0).
\end{equation*}
\end{remark}

The above well-posedness theory is still available for some other models such as the ones shown in the following two theorems.
\begin{theorem}\label{thglobal2}
Let the viscous stress tensor $\mathbb{T}$ in  (\ref{eq:1.1}) be  given by
$$
\mathbb{T}=\rho^\delta(2\alpha  \nabla u+\beta  \text{div}u\mathbb{I}_3).
$$
Then the same well-posedness theory  as in Theorem \ref{thglobal} holds in this case.

\end{theorem}

\begin{theorem}\label{thglobal3}
Let (\ref{canshu}) hold with $\beta=0$. If the viscous term $\text{div}\mathbb{T}$ is given  by
$$
\alpha \rho^\delta\triangle u,
$$
then under  the  initial conditions shown in Theorem \ref{thglobal},  there exists a   unique regular solution $(\rho, u)$ in $[0,T]\times \mathbb{R}^3$  to  the Cauchy problem  (\ref{eq:1.1})-(\ref{eq:1.2}) with (\ref{initial})-(\ref{far}).

Moreover, if $1<\min(\gamma,\delta)\leq 3$, the  solution $(\rho, u)$  solves the   Cauchy problem  (\ref{eq:1.1})-(\ref{eq:1.2}) with (\ref{initial})-(\ref{far})  in $[0,T]\times \mathbb{R}^3$ classically.
\end{theorem}

The rest of the paper is organized as follows: In \S $2$,  we show some decay estimates for the classical solutions of the multi-dimensional Burgers equations and the global well-posedenss of an ordinary differential equation, which will be used later. In \S $3$,  we prove the finite time singularity formation  in  Theorems \ref{th2}-\ref{th3}.   \S $4$-\S$8$ are devoted to establishing the global well-posedness of regular solutions stated in Theorem \ref{thglobal}.  We start with the reformulation of the original problem (\ref{eq:1.1})-(\ref{far})  as (\ref{li47-1}) in terms of the new variables, and establish the local-in-time well-posedness of smooth solutions to (\ref{li47-1}) in the case that the initial density is compactly supported in \S 4.  Note that one cannot apply the result in  \cite{zhu} directly here,  the initial velocity  here does not have a uniform upper bound in the whole space and  stays in a homogenous sobolev space.  It is also worth pointing out  that recently, Li-Wang-Xin \cite{ins} prove that  classical solutions with finite energy to the Cauchy problem of the compressible Navier-Stokes systems with constant viscosities do not exist in general    inhomogeneous Sobolev space for any short time, which indicates in particular that the homogeneous Sobolev space is crucial as studying the well-posedness (even locally in time) for the Cauchy problem of the compressible Navier-Stokes systems in the presence of vacuum. This local smooth solution to (\ref{li47-1}) is shown to be global in time in 
\S 5-\S 6 by deriving  a uniform (in time) a priori estimates independent of the size of the initial density's support through energy methods based on suitable choice of time weights. Here,  we have employed some  arguments due to  Grassin \cite{MG} and  Serre \cite{danni}   to deal with the nonlinear convection term $u\cdot \nabla u$. Next, the assumption that the initial density has compact support is removed in \S 7. Finally, in \S 8, the proof of Theorem \ref{thglobal} is completed by making use of the results obtained in \S 7.  Furthermore, we  give an appendix to list some lemmas that are  used in our proof, and outline   some proofs of properties shown in \S 2 and some inequalities  used frequently in this paper.

\section{Preliminary}

This section will be devoted to show some decay estimates for the classical solutions of the multi-dimensional Burgers equations and the global well-posedenss of an ordinary differential equation, which will be used frequently in our proof.

First,
let $\widehat{u}$ be the solution to the   problem (\ref{eq:approximation}) in    d-dimensional space. Then,  along the particle path $X(t;x_0)$ defined as
\begin{equation}\label{eq:1.5}
\frac{d}{d t}X(t;x_0)= \widehat{u}(t, X(t;x_0)), \quad x(0;  x_0)=x_0,
\end{equation}
$\widehat{u}$ is a constant in t: $\widehat{u}(t,X(t;x_0))=u_0(x_0)$ and
$$
\nabla  \widehat{u}(t,X(t;x_0))=\big(\mathbb{I}_d+t\nabla u_0(x_0)\big)^{-1}\nabla u_0(x_0).
$$

Based on this observation, one can have the following  decay estimates of $ \widehat{u}$, which play important roles in  establishing the global existence of the smooth solution to the  problem considered in this paper, and their proof  could be found in \cite{MG} or  the  appendix.
\begin{proposition}\cite{MG}\label{p1} Let $m>1+\frac{d}{2}$. Assume that
$$\nabla u_0\in L^\infty( \mathbb{R}^d),\quad \nabla^2 u_0\in H^{m-1}( \mathbb{R}^d),$$
 and  there exists a constant  $\kappa>0$ such that for all $x\in \mathbb{R}^d$,
$$
\text{Dist}\big(\text{Sp}( \nabla u_0(x)), \mathbb{R}_{-} \big)\geq \kappa,
$$
then there exists a unique global  classical solution $ \widehat{u}$ to the  problem \eqref{eq:approximation}, which satisfies
\vspace{0.1cm}
\begin{itemize}
\item[(1)]$\nabla \widehat{u}(t, x)= \frac{1}{1+t}\mathbb{I}_d+\frac{1}{(1+t)^2}K(t,x),\quad \text{for \ all} \quad x\in \mathbb{R}^d,\quad t\geq 0$;\\[1pt]
\item[(2)]$
\|\nabla^l \widehat{u}(\cdot, t)\|_{L^2(\mathbb{R}^d)}\leq C_{0,l} (1+ t)^{\frac{d}{2}-(l+1)},\quad \text{for} \quad  2\leq l\leq m+1$;\\[1pt]
\item[(3)]$
\|\nabla^2 \widehat{u}(\cdot, t)\|_{L^\infty(\mathbb{R}^d)}\leq C_0(1+ t)^{-3}\|\nabla^2 u_0\|_{L^\infty(\mathbb{R}^d)}$,
\end{itemize}
where the matrix $K(t,x)=K_{ij}:  \mathbb{R}^+\times \mathbb{R}^d  \rightarrow M_d(\mathbb{R}^d)$  satisfies
$$\|K\|_{L^{\infty}(\mathbb{R}^+\times\mathbb{R}^d)}\leq C_0\big(1+\kappa^{-d}\|\nabla u_0\|^{d-1}_{L^\infty(\mathbb{R}^d)}\big).$$ Here, $C_{0}$ is a constant depending only on $m$, $d$, $\kappa$ and $u_0$, and $C_{0,l}$ are all constants depending on $C_0$ and $l$.
Moreover, if  $u_0(0)=0$,  then it holds  that
\begin{equation}\label{linear relation}
| \widehat{u}(t,x)|\leq |\nabla \widehat{u}|_\infty |x|,\quad \text{for\ \ any } \quad t\geq 0.
\end{equation}
\end{proposition}

Secondly, we  give a global well-posedness  to the Cauchy problem of an ordinary differential  equation:
\begin{proposition}\label{ode1}
For the constants $b$, $C_i$ and $D_i$ $(i=1,2)$ satisfying
\begin{equation}\label{zhibiao1}
a>1, \quad D_1-(a-1)b<-1,\quad D_2<-1, \quad C_i\geq 0,  \quad \text{for} \quad i=1,2,
\end{equation}
there exists a constant $\Lambda$ such that there exists a global smooth solution to the  following  Cauchy problem
\begin{equation}\label{Beq:2.16A1}
\begin{cases}
\begin{split}
&\frac{\text{d} Z}{\text{d} t}(t)+\frac{b }{1+t}Z(t)=C_1(1+t)^{D_1}Z^a(t)+C_2(1+t)^{D_2}Z,\\[6pt]
&Z(x, 0)=Z_{0}<\Lambda.
\end{split}
\end{cases}
\end{equation}
\end{proposition}
Its proof can be found in the  appendix.

For simplicity,  we use the following notations.  For  matrices $A_1, A_2, A_3, B=(b_{ij})=(\textbf{b}_1,\textbf{b}_2,\textbf{b}_3)$, a vector $N=(n_1,n_2,n_3)^\top$,   set $A=(A_1,A_2, A_3)$,
\begin{equation}\begin{cases}\label{E:1.34}\displaystyle
\text{div}A=\sum_{j=1}^3 \partial_jA_j,\quad W\cdot B W=\sum_{i,j=1}^3 b_{ij}w_i w_j,\\[8pt]
\displaystyle
 |B|_2^2=B: B=\sum_{i,j=1}^3 b^2_{ij},\quad
N \cdot B=n_1 \textbf{b}_1+n_2  \textbf{b}_2+n_3\textbf{b}_3.
\end{cases}
\end{equation}

\section{Singularity formation}
In order to prove  Theorem \ref{th2},   we define:
\begin{align*}
m(t)=&\int \rho \quad \textrm{(total mass)},\quad
E_k(t)=\frac{1}{2}\int \rho|u|^{2} \quad \textrm{ (total kinetic energy)}.
\end{align*}

Then the regular solution 
$ (\rho,u)(t,x)$ in $[0,T]\times \mathbb{R}^3$ defined in Theorem \ref{th1} has finite mass $m(t)$,  momentum $\mathbb{P}(t)$ and  kinetic energy $E_k(t)$.
Indeed, due to $1<\gamma \leq 2$,  
$$
m(t)=\int \rho  \leq C\int \phi^{\frac{2}{\gamma-1}}  \leq C|\phi|^2_2<+\infty,
$$
which, together with the regularity shown in Theorem \ref{th1}, implies that
\begin{equation}\label{finite}
\begin{split}
\mathbb{P}(t)=&\int \rho u  \leq  |\rho|_2|u|_2  <+\infty,\\
E_k(t)=&\int \frac{1}{2}\rho|u|^2  \leq C|\rho|_\infty|u|^2_2<+\infty.
\end{split}
\end{equation}
The case for $1<\delta \leq 2$ can be verified similarly.

Next it is shown that the total mass and momentum are conserved.
\begin{lemma}
\label{lemmak} Let $1< \min\{\gamma,  \delta\} \leq 2$,  and  $(\rho,u)$ be the regular solution obtained in Theorem \ref{th1} with	
$|\mathbb{P}(0)|>0$, then
$$\mathbb{P}(t)=\mathbb{P}(0), \quad  m(t)= m(0), \quad \text{for} \quad t\in [0,T]. $$
\end{lemma}
\begin{proof}
The momentum equations imply that
\begin{equation}\label{deng1}
\mathbb{P}_t=-\int \text{div}(\rho u \otimes u)-\int \nabla P+\int \text{div}\mathbb{T}=0,
\end{equation}
where one has used the fact that
$$
\rho u^{(i)}u^{(j)},\quad \rho^\gamma \quad \text{and} \quad \rho^\delta \nabla u \in W^{1,1}(\mathbb{R}^3),\quad \text{for} \quad i,\ j=1,\ 2,\ 3,
$$
due to the regularities of the  solutions.

Similarly, one can also get the conservation of the total mass.
\end{proof}

Now we are ready to prove Theorem \ref{th2}.
It follows from  Lemma \ref{lemmak}  that 
\begin{equation}\label{rty}
\begin{split}
|\mathbb{P}(0)|\leq\int  \rho(t,x) |u|(t,x)
  \leq \sqrt{2}m^{\frac{1}{2}}(t)E^{\frac{1}{2}}_k(t)=\sqrt{2}m^{\frac{1}{2}}(0)E^{\frac{1}{2}}_k(t),
\end{split}
\end{equation}
which yields that there exists a unique positive lower bound for $E_k(t)$,
\begin{equation}\label{rty1}
E_k(t)\geq \frac{|\mathbb{P}(0)|^2}{2m(0)}> 0  \quad \text{for} \quad t\in [0,T].
\end{equation}
Thus one gets that 
$$
C_0\leq E_k(t)\leq \frac{1}{2} m(0)|u(t)|^2_\infty \quad \text{for} \quad t\in [0,T].
$$
Obviously, that there exists a positive constant $C_u$ such that
$$
|u(t)|_\infty\geq C_u  \quad \text{for} \quad t\in [0,T].
$$
Then Theorem \ref{th2} follows.

Finally, we prove Theorem \ref{th3}.
It follows from the definition of regular solutions given in Theorem \ref{th2}  that in the vacuum domain,  the velocity satisfies
$u_t+u\cdot\nabla u =0$,
which, along with the formula
$$\nabla u(t,X(t;x_0))=\big(\mathbb{I}_d+t\nabla u_0(x_0)\big)^{-1}\nabla u_0(x_0)$$
and  (\ref{eq:12131ss}), yields the desired conclusion.

\section{Local-in-time well-posedness with compactly supported density}
The rest of this paper is devoted to proving Theorem \ref{thglobal}. In this section, we first reformulate the original Cauchy problem (\ref{eq:1.1})-(\ref{far})  as (\ref{li47-1})  below in terms of some variables, and then  establish the local well-posedness of the  smooth solutions  to (\ref{li47-1})  in the case that the initial density has  compact support.  

Let $ \widehat{u}$ be the unique  classical solution to  (\ref{eq:approximation}) obtained in Proposition \ref{p1}. In terms of  the new  variables
$$
(\varphi, W=(\phi, w=u- \widehat{u}))=\left(\rho^{\frac{\delta-1}{2}},\sqrt{\frac{4A\gamma}{(\gamma-1)^2}}\rho^{\frac{\gamma-1}{2}}, u- \widehat{u}\right)
$$
with $w=(w^{(1)},w^{(2)},w^{(3)})^\top$,
the Cauchy problem  (\ref{eq:1.1})-(\ref{far}) can be reformulated  into
\begin{equation}\label{li47-1}
\begin{cases}
\displaystyle
\varphi_t+(w+ \widehat{u} )\cdot\nabla\varphi+\frac{\delta-1}{2}\varphi\text{div} (w+ \widehat{u} )=0,\\[10pt]
\displaystyle
W_t+\sum_{j=1}^3A^*_j(W,  \widehat{u}) \partial_j W+\varphi^2\mathbb{{L}}(w)=\mathbb{{H}}(\varphi)  \cdot \mathbb{{Q}}(w+ \widehat{u})+G(W, \varphi,  \widehat{u}),\\[10pt]
(\varphi,W)|_{t=0}=(\varphi_0,W_0)=(\varphi_0,\phi_0,0),\quad x\in \mathbb{R}^3,\\[10pt]
(\varphi,W)=(\varphi, \phi, w)\rightarrow (0,0,0) \quad\quad   \text{as}\quad \quad  |x|\rightarrow \infty \quad \text{for} \quad  t\geq 0,
 \end{cases}
\end{equation}
where
\begin{equation} \label{li47-2}
\begin{split}
\displaystyle
A^*_j(W,  \widehat{u}) =&\left(\begin{array}{cc}
w^{(j)}+ \widehat{u}^{(j)}&\frac{\gamma-1}{2}\phi e_j\\[10pt]
\frac{\gamma-1}{2}\phi e_j^\top &(w^{(j)}+ \widehat{u}^{(j)})\mathbb{I}_3
\end{array}
\right),\quad j=1,2,3,\\[10pt]
G(W, \varphi,  \widehat{u})=&-B(\nabla  \widehat{u},W)-D(\varphi^2,\nabla^2 \widehat{u}),\\[10pt]
 B(\nabla \widehat{u},W)=&\left(\begin{array}{c}
\frac{\gamma-1}{2}\phi \text{div} \widehat{u} \\[10pt]
(w\cdot\nabla)  \widehat{u}
\end{array}
\right),\quad
 D(\varphi^2,\nabla^2  \widehat{u})=\left(\begin{array}{c}0\\[10pt]
 \varphi^2 L  \widehat{u}\\
\end{array}\right),
\end{split}
\end{equation}
and $\mathbb{L}$, $\mathbb{H}$ and $\mathbb{Q}$ are given in (\ref{sseq:5.2qq}).

The main result in this section can be stated as follows:
\begin{theorem}\label{newjie} Let (\ref{canshu}) hold.
 If  initial data $( \varphi_0, \phi_0,u_0)$ satisfy
\begin{itemize}
\item[$(\rm A_1)$] $ \varphi_0\geq 0$,  $\phi_0\geq 0$ and $\big(\varphi_0, \phi_0\big)\in H^3$;
\item[$(\rm A_2)$] $u_0\in \Xi$ and  there exists a constant  $\kappa>0$ such that for all $x\in \mathbb{R}^3$:
$$
\text{Dist}\big(\text{Sp}( \nabla u_0(x)), \mathbb{R}_{-} \big)\geq \kappa;
$$
\item[$(\rm A_3)$] $\varphi_0$ and $\phi_0$ are both compactly supported: $\supp_x \varphi_0=\supp_x \phi_0 \subset B_{R} $;
\end{itemize}
where $B_{R}$ is  the ball centered at the origin with radius $R>0$, then there exist a time $T_*=T_*(\alpha,\beta,  A,\gamma,\delta, \varphi_0,W_0)>0$ independent of $R$ and a unique classical solution $(\varphi, \phi,w)$ in $[0,T_*]\times \mathbb{R}^3$ to   (\ref{li47-1}) satisfying
\begin{equation*}\begin{split}
 (\varphi,\phi) \in C([0,T_*]; H^3), \  w\in C([0,T_*]; H^{s'}) \cap L^\infty([0,T_*]; H^3),\  \varphi\nabla^4w \in L^2([0,T_*]; L^2),
\end{split}
\end{equation*}
for any constant $s'\in[2,3)$.
Moreover, for $t \in [0,T_*]$,
\begin{equation}\label{zhijijieguoA}
\begin{split}
\varphi(t,z(t;\xi_0))=\phi(t,z(t;\xi_0))=0, \quad \text{and} \quad  w(t,z(t;\xi_0))=0 \quad \text{for} \quad \xi_0 \in \mathbb{R}^3/ \supp \varphi_0,
\end{split}
\end{equation}
where  the curve $z(t;\xi_0)$ is given via
\begin{equation}\label{eq:quxianA}
\frac{d}{d t}z(t;\xi_0)=(w+ \widehat{u})(t, z(t;\xi_0)), \quad z(0;  \xi_0)=\xi_0.
\end{equation}

\end{theorem}
The next three subsections will be devoted to prove  Theorem \ref{newjie}.

\subsection{Uniform a priori estimates for the linear problem}

In order to show the local well-posedenss for (\ref{li47-1}), we will consider the following  linearized approximate problem:
\begin{equation}\label{li4}
\begin{cases}
\displaystyle
\varphi_t+(v+ \widehat{u}^N )\cdot\nabla\varphi+\frac{\delta-1}{2}h\text{div} (v+ \widehat{u} )=0,\\[10pt]
\displaystyle
W_t+\sum_{j=1}^3A^{*}_j(V,  \widehat{u}^N) \partial_j W+(\varphi^2+\eta^2)\mathbb{{L}}(w)=\mathbb{{H}}(\varphi)  \cdot \mathbb{{Q}}(v+ \widehat{u})+G(W, \varphi,  \widehat{u}),\\[10pt]
\displaystyle
(\varphi,W)|_{t=0}=(\varphi_0,W_0)=(\varphi_0,\phi_0,0),\quad x\in \mathbb{R}^3,\\[10pt]
(\varphi,W)=(\varphi, \phi, w)\rightarrow (0,0,0) \quad\quad   \text{as}\quad \quad  |x|\rightarrow \infty \quad \text{for} \quad  t\geq 0,
 \end{cases}
\end{equation}
where
\begin{equation} \label{li47-2B}
\begin{split}
\displaystyle
A^{*}_j(V,  \widehat{u}^N) =&\left(\begin{array}{cc}
v^{(j)}+ \big(\widehat{u}^N\big)^{(j)}&\frac{\gamma-1}{2}\phi e_j\\[10pt]
\frac{\gamma-1}{2}\phi e_j^\top &\big(v^{(j)}+\big(\widehat{u}^N\big)^{(j)}\big)\mathbb{I}_3
\end{array}
\right),\quad j=1,2,3,
\end{split}
\end{equation}
with the vector $\widehat{u}^N=\widehat{u}F(|x|/N)$. $F(x)\in C^\infty_c(\mathbb{R}^3)$ is a  truncation function  satisfying
 \begin{equation}\label{eq:2.6-77A}
0\leq F(x) \leq 1, \quad \text{and} \quad F(x)=
 \begin{cases}
1 \;\qquad  \text{if} \ \ |x|\leq 1,\\[8pt]
0   \ \ \ \ \ \ \    \text{if} \ \   |x|\geq 2,
 \end{cases}
 \end{equation}
and   $N\geq 1$ is a sufficiently large constant.   $\eta>0$ is a constant,  and   $ V=\left(\psi, v\right)^\top$.
$(h, \psi)$ are both known functions and $v=(v^{(1)},v^{(2)},v^{(3)})^\top\in \mathbb{R}^3$ is a known vector satisfying:
\begin{equation}\label{vg}
\begin{split}
&(h, \psi, v)(0,x)=(\varphi_0,\phi_0, 0),\quad  h \in C([0,T]; H^3), \quad  \psi \in C([0,T]; H^3), \\
&  v\in C([0,T] ; H^{s'})\cap L^\infty([0,T]; H^3),\quad  h \nabla^4 v\in L^2([0,T] ; L^2),
\end{split}
\end{equation}
for any constant $s'\in[2,3)$.   
We also assume that for $t \in [0,T_*]$,
\begin{equation}\label{zhiji}
\begin{split}
\psi(t,X(t;\xi_0))=h(t,X(t;\xi_0))=0 \ \  \text{and} \quad v(t,X(t;\xi_0))=0 \ \  \text{for} \ \  \xi_0 \in \mathbb{R}^3/ \supp \varphi_0.
\end{split}
\end{equation}

Now  the following global well-posedness in $[0,T]\times \mathbb{R}^3$  of a classical solution $(\varphi^{N\eta}, W^{N\eta})=(\varphi^{N\eta},\phi^{N\eta}, w^{N\eta})$ to  (\ref{li4})  can be obtained by the standard theory \cite{CK3,fu3} at least when $0<\eta<+\infty$ and $1\leq N <+\infty$.

 \begin{lemma}\label{lem1}
Let $\eta>0$ and \text{(A1)}-\text{(A3)}  in Theorem \ref{newjie} hold.
Then there exists  a unique classical solution $(\varphi^{N\eta},W^{N\eta})$ in $[0,T]\times \mathbb{R}^3$ to  (\ref{li4})  satisfying
\begin{equation}\label{reggh}\begin{split}
&( \varphi^{N\eta},\phi^{N\eta}) \in C([0,T]; H^3), \  \    w^{N\eta}\in C([0,T]; H^{3}) \cap L^2([0,T]; H^4).
\end{split}
\end{equation}
\end{lemma}

Next we give some a priori estimates for  solutions $(\varphi^{N\eta},\phi^{N\eta}, w^{N\eta})$  in $H^3$  in the following Lemmas \ref{f2}-\ref{4}, which are  independent of  $(R, N, \eta)$. For simplicity, we denote $(\varphi^{N\eta},\phi^{N\eta}, w^{N\eta})$ as $(\varphi,\phi, w)$, and $W^{N\eta}$ as $W$ in the rest of Subsection 4.1. For this purpose, we fix  a  positive constant $c_0$ large enough such  that

\begin{equation}\label{houmian}\begin{split}
2+\|\varphi_0\|_{3}+\|\phi_0\|_{3}+
\|u_0\|_{\Xi}\leq c_0,
\end{split}
\end{equation}
and
\begin{equation}\label{jizhu1}
\begin{split}
\displaystyle
\sup_{0\leq t \leq T^*}\big(\| h(t)\|^2_{3} +\| \psi(t)\|^2_{3}+\| v(t)\|^2_{2})+\text{ess}\sup_{0\leq t \leq T^*}|v(t)|^2_{D^3}+\int_0^{T^*}
|h \nabla^4 v|_2^2\text{d}t  & \leq c^2_1,
\end{split}
\end{equation}
for some constant 
$$c_1\geq c_0>1$$
 and  time $T^*\in (0,T)$, which will be determined later  (see (\ref{dingyi})) and depend only on $c_0$ and the fixed constants  $(\alpha, \beta, \gamma, A, \delta, T)$.

In the rest of this section,   $C\geq 1$ will  denote  a generic positive constant depending only on fixed constants $(\alpha, \beta, \gamma, A, \delta, T)$,  but independent of $(R, N, \eta)$, which may be different from line to line. It follows from the proof of   Proposition \ref{p1}  that
\begin{equation}\label{gujishuyun}\begin{split}
\|\widehat{u}\|_{T,\Xi}\leq C c^4_0.
\end{split}
\end{equation}
Based on this fact, one  can  establish  the following estimates for $\varphi$.

\begin{lemma}\label{f2} Let $(\varphi,W)$ be the unique classical solution to (\ref{li4}) in $[0,T] \times \mathbb{R}^3$. Then
\begin{equation*}\begin{split}
1+
\|\varphi(t)\|^2_3\leq & Cc^2_0\quad  \text{for} \quad  0\leq t \leq T_1=\min (T^*, c^{-5}_1).
\end{split}
\end{equation*}
\end{lemma}

\begin{proof}
Applying $\nabla^k$  $(0\leq k\leq 3)$ to $(\ref{li4})_1$,  multiplying both sides by $\nabla^k {\varphi}$, and integrating over $\mathbb{R}^3$, one gets
\begin{equation}\label{guji2}
\frac{1}{2}\frac{d}{dt}|\nabla^k {\varphi}|^2_2\leq C|\text{div}(v+ \widehat{u}^N)|_\infty |\nabla^k {\varphi}|^2_2+C\Lambda^k_1 |\nabla^k {\varphi}|_2+C\Lambda^k_2 |\nabla^k {\varphi}|_2,
\end{equation}
where
\begin{equation*}
\begin{split}
\Lambda^k_1=&|\nabla^k ((v+ \widehat{u}^N)\cdot \nabla {\varphi})-(v+ \widehat{u}^N)\cdot  \nabla^{k+1} {\varphi}|_2,\quad  \Lambda^k_2=|\nabla^k (h \text{div}(v+ \widehat{u}))|_2.
\end{split}
\end{equation*}
It follows from  Lemma \ref{lem2as} and  H\"older's inequality that
\begin{equation}\label{com1}
\begin{split}
|\Lambda^k_1|_2\leq&  C\big(\| \widehat{u}\|_{T, \Xi}+\|v\|_3\big) \|\varphi(t)\|_3,\\[4pt]
|\Lambda^k_2|_2\leq&  C\big(\| \widehat{u}\|_{T, \Xi}+\|v\|_3\big) \|h(t)\|_3+C|h \nabla^4 v|_2,
\end{split}
\end{equation}
where one has used  Proposition \ref{p1} and the fact that for  $0\leq t \leq T^*$,
\begin{equation}\label{jieduan}
\begin{split}
|\widehat{u}(t,x)|\leq  2N|\nabla \widehat{u}|_\infty\quad \text{and} \quad |\nabla F(|x|/N)| \leq CN^{-1} \quad \text{for} \quad N\leq |x|\leq 2N.
\end{split}
\end{equation}

Then, it follows from (\ref{guji2})-(\ref{com1}), Gronwall's inequality and (\ref{jizhu1}) that 
\begin{equation*}\begin{split}
\|{\varphi}(t)\|_3
\leq& \Big(\|\varphi_0\|_{3}+c^5_1t+c_1t^{\frac{1}{2}}\Big)\exp (Cc^4_1t)\leq Cc_0 \quad \text{for} \quad    0\leq t\leq T_1=\min\{T^*,c_1^{-5}\}.
\end{split}
\end{equation*}
\end{proof}

 \begin{lemma}\label{4}Let $(\varphi, W)$ be the unique classical solution to (\ref{li4}) in $[0,T] \times \mathbb{R}^3$. Then
\begin{equation*}
\begin{split}
\|W(t)\|^2_3+\int_0^t |\varphi \nabla^4 u|^2_2\text{d}s \leq& Cc^2_0\quad \text{for} \quad 0\leq t\leq T_2=\min\{T_1,c^{-10}_1\}.
\end{split}
\end{equation*}
 \end{lemma}
\begin{proof}
 Applying  $\nabla^k$ to $(\ref{li4})_2$,  multiplying both sides   by $\nabla^k W$, and integrating  over $\mathbb{R}^3$ by parts,  we have

\begin{equation}\label{zhu101}
\begin{split}
&\frac{1}{2} \frac{d}{dt}\int | \nabla^k W |^2+\alpha  | \sqrt{\varphi^2+\eta^2} \nabla^{k+1} w |^2_2
 +(\alpha+\beta)| \sqrt{\varphi^2+\eta^2} \text{div} \nabla^k w |^2_2\\[6pt]
\displaystyle
=&\int(\nabla^kW)^\top \text{div}A^{*}_j(V,  \widehat{u}^N)\nabla^kW-\int \nabla \varphi^2\cdot \nabla^k\mathbb{S}(w)\cdot \nabla^k w\\
&-\sum_{j=1}^3\int \Big(\nabla^k(A^{*}_j(V,  \widehat{u}^N)\partial_j W\big)- A^{*}_j(V,  \widehat{u}^N) \partial_j \nabla^kW\Big) \cdot \nabla^k W\\
&+ \int \Big(-\nabla^k ((\varphi^2+\eta^2) Lw)+(\varphi^2+\eta^2) L\nabla^k w+\nabla \varphi^2 \cdot Q(\nabla^k v)\Big) \cdot \nabla^k w\\
&+\int  \Big( \nabla^k(\nabla \varphi^2  \cdot Q(v+ \widehat{u}))-\nabla \varphi^2  \cdot  Q(\nabla^k(v+ \widehat{u}))\Big)\cdot \nabla^k w\\
&+ \int  \Big(\nabla \varphi^2  \cdot Q(\nabla^k  \widehat{u})\Big) \cdot \nabla^k w+\int \nabla^kG(W, \varphi,  \widehat{u})\cdot \nabla^k W
\equiv \sum_{i=1}^{8} I_i.
\end{split}
\end{equation}
Now consider the terms on the right-hand side of (\ref{zhu101}). It follows from Lemmas \ref{lem2as} and  \ref{f2}, Proposition \ref{p1}, H\"older's and  Young's  inequalities, (\ref{gujishuyun}) and (\ref{jieduan})  that 
\begin{equation}\label{term1}
\begin{split}
I_1=&\int(\nabla^k W)^\top\text{div}A^{*}_j(V,  \widehat{u}^N)\nabla^k W
\leq   C(|\nabla V|_\infty+|\nabla  \widehat{u}|_\infty) |\nabla^k W|^2_2
\leq Cc^4_1|\nabla^k W|^2_2,\\
I_2=&- \int  \Big(\nabla \varphi^2  \cdot \mathbb{S}(\nabla^k w)\Big) \cdot \nabla^k w\\
\leq&  C|\nabla \varphi|_\infty|\varphi  \nabla^{k+1} w|_2|\nabla^k w |_2
\leq  \frac{\alpha}{20}|\varphi  \nabla^{k+1} w|^2_2+Cc^2_0|\nabla^k w |^2_2,\\
I_{3}=&-\sum_{j=1}^3\int\Big(\nabla^k (A^{*}_j(V,  \widehat{u}^N)\partial_j W)-A^{*}_j(V,  \widehat{u}^N)\partial_j\nabla^k W\Big)\nabla^k W\\
  \leq&  C(|\nabla V|_\infty+|\nabla \widehat{u}|_\infty)|\nabla W|^2_2\leq Cc^4_1|\nabla W|^2_2 \quad \text{when} \quad k=1,\\
I_3\leq & C\big((|\nabla V|_\infty+|\nabla \widehat{u}|_\infty)\|\nabla W\|_1+(|\nabla^2V|_3+|\nabla^2\widehat{u}|_3)|\nabla W|_6\big)|\nabla^2 W|_2\\
\leq &Cc^4_1\|\nabla W\|^2_1 \quad \text{when} \ k=2,\\
I_3\leq  & C\big((|\nabla V|_\infty+\|\nabla \widehat{u}\|_{W^{1,\infty}})\|\nabla W\|_2+(|\nabla^3 V|_2+|\nabla^3 \widehat{u}|_2)|\nabla W|_\infty\big)|\nabla^3 W|_2\\
&+C (|\nabla^2V|_3+|\nabla^2 \widehat{u} |_3)|\nabla^2 W|_6 |\nabla^3 W|_2
\leq  Cc^4_1\|\nabla W\|^2_2 \quad \text{when} \ k=3,\\
I_{4}=&-\int \Big( \nabla^k ((\varphi^2+\eta^2) Lw)-(\varphi^2+\eta^2) L\nabla^k w \Big)\cdot \nabla^k w \\
 \leq &  C |\varphi\nabla \varphi|_{\infty}|\nabla^2 w|_2|\nabla w|_2\leq C c^2_0\|\nabla w\|^2_1  \quad  \text{when} \quad  k=1,\\
I_4\leq & C\big(|\varphi\nabla \varphi|_{\infty}|\nabla^3 w|_2+\big(|\nabla \varphi \cdot \nabla \varphi|_{3}+|\varphi \nabla^2 \varphi|_3\big)|\nabla^2 w|_6\big)|\nabla^2 w|_2\\
 \leq & C c^2_0\|\nabla^2 w\|^2_1 \quad   \text{when} \quad  k=2,\\ 
I_4\leq & C \big(|\nabla^3 \varphi|_2|\varphi \nabla^3 w|_6|\nabla^2 w|_3+|\nabla \varphi|_\infty|\nabla^2 \varphi|_3|\nabla^2 w|_6|\nabla^3 w|_2\big)\\
&+C\big(|\nabla^2 \varphi|_3|\varphi \nabla^3 w|_6+|\nabla \varphi|^2_\infty|\nabla^3 w|_2+
|\nabla \varphi|_{\infty}|\varphi\nabla^4 w|_2\big)|\nabla^3 w|_2\\
 \leq &\frac{ \alpha}{20}|\varphi\nabla^4 w|^2_2+Cc^2_0\|\nabla^2 w\|^2_1+Cc^2_0 \quad  \text{when} \ k=3,\\
 \end{split}
\end{equation}
and
\begin{equation}\label{term1^*}
\begin{split}
I_5=&\int  \Big(\nabla \varphi^2  \cdot Q(\nabla^k v)\Big) \cdot \nabla^k w\\
\leq&  C| \varphi|_\infty|\nabla \varphi|_\infty| \nabla^{k+1} v|_2|\nabla^k w |_2
\leq Cc^2_0c_1|\nabla^k w |_2, \quad  \text{when} \ k\leq 2,\\
I_5=&\int  \Big(\nabla \varphi^2  \cdot Q(\nabla^k v)\Big) \cdot \nabla^k w\\
\leq & C\big(|\nabla \varphi|^2_\infty|\nabla^3 w |_2+|\nabla^2 \varphi|_3|\varphi\nabla^3 w |_6+|\nabla \varphi|_\infty|\varphi\nabla^4 w |_2\big)| \nabla^{3} v|_2\\
\leq & Cc^2_0c_1|\nabla^3 w|_2+Cc^2_0c^2_1+\frac{\alpha}{20}|\varphi\nabla^4 w|^2_2 \quad  \text{when} \ k=3,\\
I_{6}=&\int  \Big(\nabla^k(\nabla \varphi^2  \cdot Q(v+ \widehat{u}))-\nabla \varphi^2  \cdot  Q(\nabla^k(v+ \widehat{u}))\Big)\cdot \nabla^k w\\
\leq& C\big(\| \widehat{u}\|_{T, \Xi}+\|v\|_3\big)  \|\varphi\|^2_3\|w\|_2\leq Cc^6_1 \| w\|_2 \quad  \text{when}\quad  k\leq 2,\\
I_{6}\leq & C\big(\| \widehat{u}\|_{T, \Xi}+\|v\|_3\big)  \|\varphi\|_3\big(\|\varphi\|_3|\nabla^3 w|_2+|\varphi \nabla^3 w|_6\big)+I^*_6\\
\leq&  \frac{ \alpha}{20}|\varphi \nabla^4  w |^2_2+C\big(c^6_1 \| w\|_3+c^{10}_1\big)+I^*_6\quad  \text{when}\quad k= 3,\\
I_7=& \int  \Big(\nabla \varphi^2  \cdot Q(\nabla^k  \widehat{u})\Big) \cdot \nabla^k w
\leq  C|\nabla \varphi|_\infty |\varphi|_\infty | \nabla^{k+1}   \widehat{u}|_2|\nabla^k w |_2
\leq  Cc^6_1 |\nabla^k w |_2,\\
I_8
=& -\frac{\gamma-1}{2}\int \nabla^k (\phi \text{div}  \widehat{u}) \nabla^k \phi -\int  \nabla^k (w \cdot \nabla  \widehat{u}) \nabla^k w-\int  \nabla^k (\varphi^2 L  \widehat{u}) \nabla^k w \\
\leq&  C\| \widehat{u}\|_{T, \Xi}\big(\|W\|^2_3+\|W\|_3\|\varphi\|^2_3\big)+I^*_8 \delta_{3,k},
\end{split}
\end{equation}
where  integrations by parts have been used for the terms $I_5$-$I_6$ and $I_8$ when $k=3$.
%
%
And the terms $I^*_6$ and $I^*_8 \delta_{3,k}$   can be estimated similarly by integration by parts as
\begin{equation*}
\begin{split}
 I^*_6=&\int \varphi \nabla^4 \varphi   \cdot Q(v+ \widehat{u})) \cdot \nabla^3 w
\leq \frac{ \alpha}{20}|\varphi \nabla^4  w |^2_2+ Cc^6_1 \| w\|_3+Cc^{10}_1,\\
I^*_8 \delta_{3,k}=&\int \varphi^2 L\nabla^3  \widehat{u} \cdot \nabla^3 w
\leq   \frac{ \alpha}{20}|\varphi \nabla^4  w |^2_2+Cc^6_0 \| w\|_3+Cc^{10}_0.
\end{split}
\end{equation*}

Due to (\ref{zhu101})-(\ref{term1^*}), one gets that
\begin{equation}\label{zhu101qw}
\begin{split}
&\frac{1}{2} \frac{d}{dt}\int |\nabla^k W|^2+\frac{1}{2}\alpha  | \sqrt{\varphi^2+\eta^2}\nabla^4 w |^2_2
\leq Cc^6_1\|W\|^2_{3}+Cc^{10}_1,
\end{split}
\end{equation}
which, along with Gronwall's inequality, implies that for $0\leq t \leq T_2=\min\{T_1,c^{-10}_1\}$,

\begin{equation}\label{zhu10222}
\begin{split}
\| W(t)\|^2_{3}+\int_0^t (\varphi^2+\eta^2)| \nabla^4 w|^2_2\text{d}s  \leq  \big(\| W_0\|^2_{3}+Cc^{10}_1 t\big)\exp (Cc^{6}_1t)\leq Cc^{2}_0.
\end{split}
\end{equation}

\end{proof}

Combining the estimates obtained in Lemmas \ref{f2}-\ref{4} shows that
\begin{equation}\label{jkkll}
\begin{split}
1+
\|\varphi(t)\|^2_3+\|W(t)\|^2_3+\sum_{k=0}^3\int_0^{t} (\varphi^2+\eta^2)
|\nabla^{k+1} w|_2^2\text{d}s \leq& Cc^2_0,
\end{split}
\end{equation}
for $0 \leq t \leq \min\{T^*, c^{-10}_1\}$.
Therefore, defining  the constant $c_1$ and time $T^*$ by

\begin{equation}\label{dingyi}
\begin{split}
&c_1=C^{\frac{1}{2}}c_0,  \quad T^*=\min\{T, c^{-10}_1\},
\end{split}
\end{equation}
we then deduce that  for $0\leq t \leq T^*$,
\begin{equation}\label{jkk}
\begin{split}
\| \varphi(t)\|^2_{3} +\| \phi(t)\|^2_{3}+\| w(t)\|^2_{3}+\sum_{k=0}^3\int_0^{t} (\varphi^2+\eta^2)
|\nabla^{k+1} w|_2^2\text{d}s & \leq c^2_1.
\end{split}
\end{equation}
In other words, given fixed $c_0$ and $T$, there exist positive
constant $c_1$ and time $T^*$, depending solely on $c_0$, $T$ and the
generic constant $C$  such that if
\eqref{jizhu1} holds for $h$ and  $V$, then   \eqref{jkk} holds for the
solution to \eqref{li4} in $[0, T^*]\times \mathbb{R}^3$. Here, it should be noted that the definitions of $(c_1, T^*)$ are all independent of the parameters $(R, N, \eta)$.

\subsection{Passing to the  limits as  $\eta\rightarrow 0$ and $N\rightarrow +\infty$} Due to  the a priori estimate (\ref{jkk}), one can solve the following problem:
\begin{equation}\label{li4*}
\begin{cases}
\displaystyle
\varphi_t+(v+\widehat{u})\cdot\nabla\varphi+\frac{\delta-1}{2}h\text{div} (v+\widehat{u})=0,\\[10pt]
\displaystyle
W_t+\sum_{j=1}^3A^*_j(V,\widehat{u}) \partial_j W+ \varphi^2 \mathbb{{L}}(w)= \mathbb{{H}}(\varphi)  \cdot \mathbb{{Q}}(v+\widehat{u})+G(W, \varphi, \widehat{u}),\\[10pt]
\displaystyle
(\varphi,W)|_{t=0}=(\varphi_0,W_0)=(\varphi_0,\phi_0,0),\quad x\in \mathbb{R}^3,\\[10pt]
(\varphi,W)=(\varphi, \phi, w)\rightarrow (0,0,0) \quad\quad   \text{as}\quad \quad  |x|\rightarrow \infty \quad \text{for} \quad  t\geq 0.
 \end{cases}
\end{equation}
\begin{lemma}\label{lem1q}
 Assume that  the initial data $(\varphi_0, \phi_0,u_0)$ satisfy conditions \text{(A1)}-\text{(A3)} shown in Theorem \ref{newjie}.
Then there exist a time $T^*=T^*(\alpha,\beta,  A,\gamma,\delta, \varphi_0,W_0)>0$ independent of $R$ and a  unique classical solution $(\varphi, W)$ in $[0,T^*]\times \mathbb{R}^3$ to   (\ref{li4*}) such that
\begin{equation*}\begin{split}
& (\varphi,\phi) \in C([0,T^*]; H^3),\quad     w\in C([0,T^*]; H^{s'})\cap L^\infty([0,T^*]; H^3), \quad    \varphi \nabla^4 w \in L^2([0,T^*] ; L^2),
\end{split}
\end{equation*}
for any constant $s'\in[2,3)$. Moreover,  $(\varphi, W)$  satisfies  (\ref{jkk}), and for $t \in [0,T^*]$,
\begin{equation}\label{zhijijieguo}
\varphi(t,y(t;\xi_0))=\phi(t,y(t;\xi_0))=0 \quad \text{and} \quad
w(t,y(t;\xi_0))=0 \quad \text{for} \quad   \xi_0 \in \mathbb{R}^3/ \supp \varphi_0.
\end{equation}
\end{lemma}
\begin{proof}
\textbf{Step 1:} Passing to the limit as $\eta \rightarrow 0$.
 Let $N\geq 1$ be a  fixed constant. Due to Lemma \ref{lem1}, for every $\eta>0$,  there exists a unique classical solution $(\varphi^{N\eta}, W^{N\eta})$  to the linear Cauchy  problem (\ref{li4}) satisfying  (\ref{jkk})  in $[0,T^*]\times \mathbb{R}^3$,  where  these estimates and the life span $T^*>0$ are all independent of $(N,R,\eta)$.

Due to the estimate (\ref{jkk}) and the  equations in (\ref{li4}),
it holds that 
\begin{equation}\label{uniformshijianA}
\begin{split}
\|\varphi^{N\eta}_t\|_2+\|\phi^{N\eta}_t\|_2+\|w^{N\eta}_t\|_1+\int_0^t \| w^{N\eta}_t\|^2_2\text{d}s\leq C_0(N),\quad \text{for}\quad 0\leq t \leq T^*,
\end{split}
\end{equation}
where the constant $C_0(N)$ depends only on the generic constant $C$,  $(\varphi_0,\phi_0,u_0)$  and $N$,  and is independent of $(R,\eta)$.

By virtue of the  uniform estimate (\ref{jkk}) and (\ref{uniformshijianA}) independent of  $(R,\eta)$, there exists a subsequence of solutions (still denoted by) $(\varphi^{N\eta}, W^{N\eta})$, which    converges to a limit $(\varphi^N,W^N)=(\varphi^N, \phi^N, w^N)$  in   weak  or  weak* sense as $\eta\rightarrow 0$:
\begin{equation}\label{ruojixiana}
\begin{split}
(\varphi^{N\eta}, W^{N\eta})\rightharpoonup  (\varphi^N,W^N) \quad &\text{weakly* \ in } \ L^\infty([0,T^*];H^3(\mathbb{R}^3)),\\
(\varphi^{N\eta}_t, \phi^{N\eta}_t)\rightharpoonup (\varphi^N_t,  \phi^N_t) \quad &\text{weakly* \ in } \ L^\infty([0,T^*];H^2(\mathbb{R}^3)),\\
w^{N\eta}_t \rightharpoonup w^N_t  \quad & \text{weakly* \ in } \ L^\infty([0,T^*];H^1(\mathbb{R}^3)), \\
 w^{N\eta}_t \rightharpoonup w^N_t  \quad &\text{weakly \ \ in } \ L^2([0,T^*];H^2(\mathbb{R}^3)).
\end{split}
\end{equation}
The lower semi-continuity of weak convergence implies  that $(\varphi^N, W^N)$ also satisfies the corresponding  estimates (\ref{jkk}) and (\ref{uniformshijianA}) except that of $\varphi^N \nabla^4 w^N$.

Again by the uniform estimate (\ref{jkk}) and (\ref{uniformshijianA}) independent of  $\eta$,  and  the compactness in Lemma \ref{aubin} (see \cite{jm}), it holds that for any $R_0> 0$,  there exists a subsequence of solutions (still denoted by) $(\varphi^{N\eta}, W^{N\eta})$, which  converges  to  the same  limit $(\varphi^N,W^N)$ as above in the following  strong sense:
\begin{equation}\label{ert}\begin{split}
&(\varphi^{N\eta}, W^{N\eta}) \rightarrow (\varphi^N, W^N) \quad \text{ in } \ C([0,T^*];H^2(B_{R_0})),\quad \text{as}\quad  \eta\rightarrow 0.
\end{split}
\end{equation}

Let $ \beta_i=1,2,3$ for $i=1,2,3,4$.
Because $\|\varphi^{N\eta} \nabla^4 w^{N\eta}\|_{L^2L^2_{T^*}}$ is uniformly bounded with respect to $N$ and $\eta$,  there exists a vector $g=(g_1,g_2,g_3)\in L^2([0,T^*]; L^2(\mathbb{R}^3))$ such that, 
\begin{equation}\label{bridge}
\begin{split}
\int_{\mathbb{R}^3}\int_0^{T^*} \varphi^{N\eta} \partial_{\beta_1\beta_2\beta_3\beta_4}  w^{N\eta} f \text{d}x\text{d}t \rightarrow \int_{\mathbb{R}^3}\int_0^{T^*}  g f \text{d}x\text{d}t \quad \text{as}\quad  \eta\rightarrow 0
\end{split}
\end{equation}
 for any $f\in L^2([0,T^*]; L^2(\mathbb{R}^3))$.

For any $f^* \in C^\infty_c([0,T^*]\times \mathbb{R}^3)$, it follows from  integration by parts, (\ref{ruojixiana}) and (\ref{ert}) that
\begin{equation}\label{bridge1}
\begin{split}
&\lim_{\eta\rightarrow0}\int_{\mathbb{R}^3}\int_0^{T^*} \varphi^{N\eta} \partial_{\beta_1\beta_2\beta_3\beta_4}  w^{N\eta} f^* \text{d}x\text{d}t\\
=&\lim_{\eta\rightarrow0}\Big(-\int_{\mathbb{R}^3}\int_0^{T^*} \Big(\partial_{\beta_1}\varphi^{N\eta} \partial_{\beta_2\beta_3\beta_4}  w^{N\eta} f^* + \varphi^{N\eta}\partial_{\beta_2\beta_3\beta_4}  w^{N\eta} \partial_{\beta_1} f^*\Big) \text{d}x\text{d}t\Big)\\
=&-\int_{\mathbb{R}^3}\int_0^{T^*} \Big(\partial_{\beta_1}\varphi^{N} \partial_{\beta_2\beta_3\beta_4}  w^{N} f^* + \varphi^{N} \partial_{\beta_2\beta_3\beta_4}  w^{N} \partial_{\beta_1} f^* \Big)\text{d}x\text{d}t\\
=&-\int_{\mathbb{R}^3}\int_0^{T^*}  \partial_{\beta_2\beta_3\beta_4}  w^{N} \partial_{\beta_1} (\varphi^{N}f^*) \text{d}x\text{d}t=\int_{\mathbb{R}^3}\int_0^{T^*} \varphi^{N} \partial_{\beta_1\beta_2\beta_3\beta_4} w^{N} f^* \text{d}x\text{d}t.
\end{split}
\end{equation}
Because $C^\infty_c([0,T^*]\times \mathbb{R}^3)$ is dense in $L^2([0,T^*]; L^2(\mathbb{R}^3))$, one can get
\begin{equation}\label{ruojixian1A}
\begin{split}
\varphi^{N\eta} \nabla^4 w^{N\eta} \rightharpoonup  \varphi^N\nabla^4 w^N \quad &\text{weakly \ in } \ L^2( [0,T^*]  \times \mathbb{R}^3 ),\quad \text{as}\quad  \eta\rightarrow 0.
\end{split}
\end{equation}
Thus, $(\varphi^N, W^N)$  satisfies   (\ref{jkk}).

Now, it is easy to show  that $(\varphi^N, W^N) $ is a weak solution in the sense of distribution  to the  following problem:

\begin{equation}\label{li4**}
\begin{cases}
\displaystyle
\varphi^N_t+(v+\widehat{u}^N)\cdot\nabla\varphi^N+\frac{\delta-1}{2}h\text{div} (v+\widehat{u})=0,\\[8pt]
\displaystyle
W^N_t+\sum_{j=1}^3A^{*}_j(V,  \widehat{u}^N) \partial_j W^N+ (\varphi^N)^2 \mathbb{{L}}(w^N)= \mathbb{{H}}(\varphi^N)  \cdot \mathbb{{Q}}(v+\widehat{u})+G(W^N, \varphi^N, \widehat{u}),\\[8pt]
\displaystyle
(\varphi^N,W^N)|_{t=0}=(\varphi_0,W_0)=(\varphi_0,\phi_0,0),\quad x\in \mathbb{R}^3,\\[8pt]
(\varphi^N,W^N)=(\varphi^N, \phi^N, w^N)\rightarrow (0,0,0) \quad\quad   \text{as}\quad \quad  |x|\rightarrow \infty \quad \text{for} \quad  t\geq 0,
 \end{cases}
\end{equation}
and has the following  regularities
\begin{equation}\label{zheng}
\begin{split}
& (\varphi^N, \phi^N,  w^N)\in L^\infty([0,T^*]; H^3),\quad   (\varphi^N_t,  \phi^N_t) \in L^\infty([0,T^*]; H^2),   \\
& \varphi^N \nabla^4 w^N \in L^2([0,T^*] ; L^2),\quad   \ w^N_t\in L^\infty([0,T^*]; H^1)\cap L^2([0,T^*]; H^2).
\end{split}
\end{equation}

\textbf{Step 2:} Passing to the limit as $N \rightarrow +\infty$. One can prove the existence, uniqueness and time continuity  of the classical solutions to   (\ref{li4*}) as follows.

\textbf{Step 2.1:} Existence.  According to Step 1,  for every $N\geq 1$, there exists a solution $(\varphi^N, W^N) $ to    (\ref{li4**}) with the  a priori estimate (\ref{jkk}) independent of $(R,N)$.

Due to (\ref{jkk}) and (\ref{li4**}),
 for any $R_0>0$, it holds that 
\begin{equation}\label{uniformshijianAS}
\begin{split}
\|\varphi^{N}_t\|_{H^2(B_{R_0})}+\|\phi^{N}_t\|_{H^2(B_{R_0})}+\|w^{N}_t\|_{H^1(B_{R_0})}+\int_0^t \|\nabla^2 w^{N}_t\|^2_{L^2(B_{R_0})}\text{d}s\leq C_0(R_0),
\end{split}
\end{equation}
where the constant $C_0(R_0)$  is independent of $(R, N)$.

It follows from  (\ref{jkk}) and (\ref{uniformshijianAS}) that  there exists a subsequence of solutions (still denoted by) $(\varphi^{N}, W^{N})$, which    converges to a  limit $(\varphi,W)$   in   weak  or  weak* sense as $\eta\rightarrow 0$:
\begin{equation}\label{ruojixianaB}
\begin{split}
(\varphi^N, W^{N})\rightharpoonup  (\varphi,W) \quad &\text{weakly* \ in } \ L^\infty([0,T^*];H^3(\mathbb{R}^3)),\\
(\varphi^{N}_t, \phi^{N}_t)\rightharpoonup (\varphi_t,  \phi_t) \quad &\text{weakly* \ in } \ L^\infty([0,T^*];H^2(B_{R_0})),\\
w^{N}_t \rightharpoonup w_t  \quad & \text{weakly* \ in } \ L^\infty([0,T^*];H^1(B_{R_0})), \\
 w^{N}_t \rightharpoonup w_t  \quad &\text{weakly \ \ in } \ L^2([0,T^*];H^2(B_{R_0})),
\end{split}
\end{equation}
for any $R_0> 0$.
The lower semi-continuity of weak convergence implies  that $(\varphi, W)$ also satisfies the   estimate (\ref{jkk}) except that for  $\varphi \nabla^4 w$.

Again by virtue of the uniform estimates (\ref{jkk}) and (\ref{uniformshijianAS}) independent of  $N$ and  the compactness in Lemma \ref{aubin} (see \cite{jm}), for any $R_0> 0$,  there exists a subsequence of solutions (still denoted by) $(\varphi^{N}, W^{N})$, which  converges  to  the same  limit $(\varphi,W)=(\varphi, \phi, w)$ in the following  strong sense:
\begin{equation}\label{ertB}\begin{split}
&(\varphi^{N}, W^{N}) \rightarrow (\varphi, W) \quad \text{ in } \ C([0,T^*];H^2(B_{R_0})),\quad \text{as}\quad  N\rightarrow +\infty.
\end{split}
\end{equation}
Then,   (\ref{ruojixianaB})-(\ref{ertB}) and a similar proof as for (\ref{ruojixian1A}) show  that
\begin{equation}\label{ruojixianB}
\begin{split}
\varphi^{N} \nabla^4 w^{N} \rightharpoonup  \varphi\nabla^4 w \quad &\text{weakly \ in } \ L^2( [0,T^*]  \times \mathbb{R}^3 ),\quad \text{as}\quad  N\rightarrow +\infty.
\end{split}
\end{equation}
Thus, $(\varphi, W)$  satisfies   (\ref{jkk}).

It is easy to show  that $(\varphi, W) $ is a weak solution in the sense of distribution  to  (\ref{li4*}) with the properties that 
\begin{equation}\label{zhengV}
\begin{split}
& (\varphi,\phi, w) \in L^\infty([0,T^*]; H^3),\quad ( \varphi_t, \phi_t ) \in L^\infty([0,T^*]; H^2_{loc}),   \\
&  \varphi \nabla^4 w \in L^2([0,T^*] ; L^2),\quad w_t\in L^\infty([0,T^*]; H^1_{loc})\cap L^2([0,T^*]; H^2_{loc}).
\end{split}
\end{equation}

\textbf{Step 2.2:} Time-continuity.   (\ref{li4*}) and (\ref{zhiji}) imply that 
\begin{equation*}
\begin{split}
&\frac{d}{dt}\varphi(t,X(t;\xi_0))=0,\quad \frac{d}{dt}\phi(t,X(t;\xi_0))=0,\quad \frac{d}{dt}w(t,X(t;\xi_0))=-w\cdot \nabla \widehat{u},
\end{split}
\end{equation*}
for $t \in [0,T^*]$ and $ \xi_0 \in \mathbb{R}^3/ \supp \varphi_0$, which yields that
\begin{equation*}\begin{split}
\varphi(t,X(t;\xi_0))=\phi(t,X(t;\xi_0))=0, \quad w(t,X(t;\xi_0))=0,
\end{split}
\end{equation*}
for $t \in [0,T^*]$ and $ \xi_0 \in \mathbb{R}^3/ \supp \varphi_0$. Thus for any positive time $t\in [0,T^*]$, the solution $(\varphi, \phi, w)$ is  compactly supported with respect to the space variable $x$.

Let $ \xi_0 \in \partial \supp \varphi_0$ and  $|\xi_0|\leq R$.  It follows form   Proposition \ref{p1}  that 
\begin{equation}\label{zhongyaoguancha}
\begin{split}
|X(t;\xi_0)|\leq R+\|\nabla \widehat{u}\|_{L^\infty L^\infty_{T^*}}\int_0^t |X(\tau;\xi_0)|\text{d}\tau.
\end{split}
\end{equation}
Then the  Gronwall's inequality implies that  for $0\leq t\leq T^*$
\begin{equation}\label{zhengVC}
|X(t;\xi_0)| \leq R\exp \Big(\|\nabla \widehat{u}\|_{L^\infty L^\infty_{T^*}}T^*\Big)\leq C_0(R,T^*),
\end{equation}
 where $C_0(R,T^*)$ depends  on the generic constant $C$, $T^*$,  $(\varphi_0,\phi_0,u_0)$  and $R$.
Then for any $0<R<+\infty$,  it follows from  (\ref{li4*})  that 
\begin{equation}\label{zhengVD}
\begin{split}
& (\varphi_t,\phi_t) \in L^\infty([0,T^*]; H^2),\quad   \ w_t\in L^\infty([0,T^*]; H^1)\cap L^2([0,T^*]; H^2).
\end{split}
\end{equation}

 (\ref{zhengV}) and  (\ref{zhengVD}) imply that  for any $s'\in [0,3)$,
\begin{equation}\label{zhengW}
\begin{split}
&(\varphi,\phi) \in C([0,T^*]; H^{s'}\cap \text{weak-}H^3),\  w\in C([0,T^*]; H^{s'}\cap \text{weak-}H^3).
\end{split}
\end{equation}

Using the same arguments as  in the proof of Lemma \ref{f2}-\ref{4}, one has
\begin{equation}\label{liu03}
\displaystyle \lim_{t\rightarrow 0}\sup \|\varphi(t)\|_3 \leq \|\varphi_0\|_3,\quad \displaystyle \lim_{t\rightarrow 0}\sup \|\phi(t)\|_3 \leq \|\phi_0\|_3,
\end{equation}
which, together with  Lemma \ref{zheng5} and  (\ref{zhengW}), implies  that  $(\varphi,\phi)$ is   right continuous at $t=0$ in $H^3$ space. The time reversibility of equations in $(\ref{li4*})$ for $(\phi,\varphi)$
yields that
\begin{equation}\label{xian}
(\varphi,\phi) \in C([0,T^*]; H^3).
\end{equation}

\textbf{Step 2.3:} Uniqueness. Let $(\varphi_i,W_i)=(\varphi_i,\phi_i,w_i)$, $i=1,2$,  be two solutions obtained in  Step 2.1, and
\begin{equation*}\begin{split}
&\overline{\varphi}=\varphi_1-\varphi_2, \  \overline{\phi}=\phi_1-\phi_2, \\
 & \overline{w}=w_1-w_2,\  \overline{W}=W_1-W_2.
\end{split}
\end{equation*}
Then  ${\varphi}_1=\varphi_2$ follows from
$\overline{\varphi}_t+v\cdot \nabla \overline{\varphi}=0
$, and  $W_1=W_2$ follows from
\begin{equation}\label{li4*cha}
\displaystyle
\overline{W}_t+\sum_{j=1}^3A^*_j(V,\widehat{u}) \partial_j \overline{W}+ \varphi^2_1 \mathbb{{L}}(\overline{w})=G(\overline{W}, \varphi_1, \widehat{u}).
\end{equation}

\end{proof}

\subsection{Proof of Theorem \ref{newjie}}
The proof  is based on the classical iteration scheme and the existence results for the linearized problem obtained  in  \S 4.2. As in \S 4.2, we define constants $c_{0}$ and  $c_1$, and assume that
\begin{equation*}\begin{split}
&1+\|\varphi_0\|_{3}+ +\|\phi_0\|_{3}+\|u_0\|_{\Xi}\leq c_0.
\end{split}
\end{equation*}
Let $(\varphi^0, W^0=(\phi^0,w^0))$ with the regularities
\begin{equation*}
\begin{split}
&(\varphi^0,\phi^0) \in C([0,T^*];H^{3}),\  w^0\in C([0,T^*];H^{s'})\cap  L^\infty([0,T^*];H^3), \ \varphi^0 \nabla^4 w^0 \in L^2([0,T^*];L^2)
\end{split}
\end{equation*}
for any $ s' \in [2,3)$  be the solution to the  following problem:
\begin{equation}\label{zheng6}
\begin{cases}
X_t+\widehat{u} \cdot\nabla X=0 \quad \text{in} \quad (0,+\infty)\times \mathbb{R}^3,\\[10pt]
Y_t+\widehat{u}  \cdot \nabla Y=0 \quad \text{in} \quad (0,+\infty)\times \mathbb{R}^3,\\[10pt]
\Phi_t- X^2\triangle \Phi=0 \quad \ \text{in} \quad  (0,+\infty)\times \mathbb{R}^3 ,\\[10pt]
(X,Y,\Phi)|_{t=0}=(\varphi_0,\phi_0,0) \quad \text{in} \quad \mathbb{R}^3,\\[10pt]
(X,Y,\Phi)\rightarrow (0,0,0) \quad \text{as } \quad |x|\rightarrow +\infty,\quad t> 0.
\end{cases}
\end{equation}
Choose a  time $T^{**}\in (0,T^*]$ small enough such that for $0\leq t \leq T^{**}$,
\begin{equation}\label{jizhu}
\begin{split}
|| \varphi^0(t)||^2_{3} +|| \phi^0(t)||^2_{3}+||w^0(t)||^2_{3}+\int_0^{t} |\varphi^0\nabla^4w^0|\text{d}s\leq c_1^2.
\end{split}
\end{equation}

Now the existence, uniqueness and time continuity can be proved as follows.

\textbf{Step 1:} Existence. Let $(h,\psi, v)=(\varphi^0,\phi^0,w^0)$ and  $(\varphi^1, W^1)$ be a classical solution to (\ref{li4*}). Then we construct approximate solutions $(\varphi^{k+1}, W^{k+1})=(\varphi^{k+1}, \phi^{k+1},w^{k+1})$
 inductively as follows: Assume that  $(\varphi^{k}, W^{k})$ has been defined for $k\geq 1$, and  let $(\varphi^{k+1},W^{k+1})$  be the unique solution to (\ref{li4*}) with $(h,V)$ replaced by $ (\varphi^k, W^k)$ as:
\begin{equation}\label{li6}
\begin{cases}
\displaystyle
 \varphi^{k+1}_t+(w^{k}+ \widehat{u})\cdot \nabla \varphi^{k+1}+\frac{\delta-1}{2}\varphi^{k}\text{div} (w^{k}+ \widehat{u})=0,\\[10pt]
\displaystyle
 W^{k+1}_t+\sum_{j=1}^3A^*_j({W^k}, \widehat{u})\partial_j W^{k+1}+(\varphi^{k+1})^{2}\mathbb{L}(w^{k+1})\\[10pt]
=\mathbb{H}(\varphi^{k+1}) \cdot \mathbb{Q}(w^{k}+ \widehat{u})+G(W^{k+1},\varphi^{k+1}, \widehat{u}),\\[10pt]
  (\varphi^{k+1}, W^{k+1})|_{t=0}=(\varphi_0, W_0),\quad x\in \mathbb{R}^3,\\[10pt]
(\varphi^{k+1}, W^{k+1})\rightarrow (0,0) \quad \text{as } \quad |x|\rightarrow +\infty,\quad t> 0.
 \end{cases}
\end{equation}
Here  the sequence $(\varphi^k, W^{k})$ satisfies the  uniform a priori estimate (\ref{jkk})
for $0 \leq t \leq T^{**}$, and is also compactly supported as long as it exists.

Now we  prove the  convergence of the whole sequence  $(\varphi^k, W^k)$  to a limit $(\varphi, W)$  in some strong sense.
Set
\begin{equation*}\begin{split}
&\overline{\varphi}^{k+1}=\varphi^{k+1}-\varphi^{k},\
\overline{W}^{k+1}=(\overline{\phi}^{k+1},\overline{w}^{k+1} )^\top,\\
 & \overline{\phi}^{k+1}=\phi^{k+1}-\phi^k,\ \overline{w}^{k+1} =w^{k+1}-w^k,
\end{split}
\end{equation*}
then it follows  from (\ref{li6}) that
 \begin{equation}
\label{eq:1.2w}
\begin{cases}
\displaystyle
\ \  \overline{\varphi}^{k+1}_t+(w^{k}+ \widehat{u})\cdot \nabla\overline{\varphi}^{k+1} +\overline{w}^k\cdot\nabla\varphi ^{k}\\[6pt]
\displaystyle
=-\frac{\delta-1}{2}(\overline{\varphi}^{k} \text{div}(w^{k-1}+ \widehat{u}) +\varphi ^{k}\text{div}\overline{w}^k),\\[6pt]
\displaystyle
\ \ \overline{W}^{k+1}_t+ \sum\limits_{j=1}^3A^*_j(W^k, \widehat{u})\partial_{j}\overline{W}^{k+1}+(\varphi^{k+1})^{2}\mathbb{L}(\overline{w}^{k+1})\\[6pt]
\displaystyle
=-\sum\limits_{j=1}^3A_j(\overline{W}^k)\partial_{j}{W}^{k}- \overline{\varphi}^{k+1}(\varphi^{k+1}+\varphi^k)\mathbb{L}(w^k)\\[6pt]
\displaystyle
\ \  +\big(\mathbb{H}({\varphi}^{k+1})-\mathbb{H}({\varphi}^{k})\big)\cdot \mathbb{Q}(w^{k}+ \widehat{u})+\mathbb{H} (\varphi^{k})\cdot \mathbb{Q}(\overline{w}^{k})\\[6pt]
\displaystyle
\ \  -B(\nabla  \widehat{u},\overline{W}^{k+1})-D(\overline{\varphi}^{k+1}(\varphi^{k+1}+\varphi^k),\nabla^2  \widehat{u}).
\end{cases}
\end{equation}
First, multiplying $(\ref{eq:1.2w})_1$ by $2\overline{\varphi}^{k+1}$ and integrating  over $\mathbb{R}^3$ yield that 
\begin{equation}\label{varcha}
\begin{split}
\frac{d}{dt}|\overline{\varphi}^{k+1}|^2_2
\leq& C\big(|\nabla w^k|_\infty+ |\nabla\widehat{u}|_\infty\big)|\overline{\varphi}^{k+1}|^2_2+C |\overline{\varphi}^{k+1}|_2|\overline{w}^k|_2|\nabla \varphi^k|_\infty\\
&+C |\overline{\varphi}^{k+1}|_2\big(\big(|\nabla w^{k-1}|_\infty+ |\nabla\widehat{u}|_\infty\big)|\overline{\varphi}^{k}|_2+|\varphi^k \text{div}\overline{w}^k|_2 \big)\\
\leq& C\nu^{-1}|\overline{\varphi}^{k+1}(t)|^2_2+\nu\big( |\overline{w}^k|^2_2+|\overline{\varphi}^k|^2_2+|\varphi^k \text{div}\overline{w}^k|^2_2\big),
\end{split}
\end{equation}
where $0<\nu \leq \frac{1}{10}$ is a constant to be determined.

Second, multiplying $(\ref{eq:1.2w})_2$ by $2\overline{W}^{k+1}$ and integrating  over $\mathbb{R}^3$, one has
\begin{equation}\label{zheng8}
\begin{split}
&\frac{d}{dt}\int |\overline{W}^{k+1}|^2+2\alpha|\varphi^{k+1}\nabla\overline{w}^{k+1} |^2_2+2(\alpha+\beta)|\varphi^{k+1}\text{div}\overline{w}^{k+1} |^2_2\\
= & \int(\overline{W}^{k+1})^\top\text{div}A^*(W^k, \widehat{u})\overline{W}^{k+1}-2\int  \sum_{j=1}^3(\overline{W}^{k+1})^\top A_j(\overline{W}^{k}) \partial_{j}W^{k}\\
&-2 \frac{\delta-1}{\delta}\int\nabla( \varphi^{k+1})^2\cdot Q(\overline{w}^{k+1} )\cdot  \overline{w}^{k+1} -2\int  \overline{\varphi}^{k+1}(\varphi^{k+1}+\varphi^k)  Lw^k\cdot  \overline{w}^{k+1}\\
&+2\int \nabla (\overline{\varphi}^{k+1}(\varphi^{k+1}+\varphi^k))\cdot Q(w^k+ \widehat{u})\cdot  \overline{w}^{k+1}+2\int \nabla (\varphi^{k})^2\cdot Q(\overline{w}^{k}) \cdot  \overline{w}^{k+1} \\
&-2\int B(\nabla  \widehat{u},\overline{W}^{k+1})\cdot \overline{W}^{k+1}-2\int \overline{\varphi}^{k+1}(\varphi^{k+1}+\varphi^k)L( \widehat{u})\cdot \overline{w}^{k+1}  :=\sum_{i=1}^{8} J_i.
\end{split}
\end{equation}

The terms $J_1$-$J_8$ above can be estimated as follows
\begin{equation}\label{ya1}
\begin{split}
J_1=&\int(\overline{W}^{k+1})^\top\text{div}A^*({W}^k, \widehat{u})\overline{W}^{k+1}
\leq C|\nabla (w^k+ \widehat{u})|_\infty| \overline{W}^{k+1}|^2_2\leq C|\overline{W}^{k+1}|^2_2,\\
J_2=&-2\int  \sum_{j=1}^3A_j(\overline{W}^{k})\partial_{j}W^{k}\cdot \overline{W}^{k+1}\\
\leq & C|\nabla W^k|_\infty | \overline{W}^{k}|_2 | \overline{W}^{k+1}|_2
\leq C\nu^{-1}|\overline{W}^{k+1}|^2_2+\nu| \overline{W}^{k}|^2_2,\\
\end{split}
\end{equation}
\begin{equation}\label{ya1A}
\begin{split}
J_3=&-2 \frac{\delta-1}{\delta}   \int \nabla(\varphi^{k+1})^2\cdot Q( \overline{w}^{k+1} )\cdot \overline{w}^{k+1} \\
\leq & C|\nabla \varphi^{k+1}|_\infty |\varphi^{k+1}\nabla \overline{w}^{k+1} |_2|\overline{w}^{k+1} |_2
\leq  C |\overline{W}^{k+1}|^2_2+\frac{\alpha}{20} |\varphi^{k+1}\nabla \overline{w}^{k+1} |^2_2,\\
J_4=&-2\int  \overline{\varphi}^{k+1}(\varphi^{k+1}+\varphi^k)Lw^k \cdot  \overline{w}^{k+1} \\
\leq & C| \overline{\varphi}^{k+1}|_2| \varphi^{k+1}\overline{w}^{k+1} |_6|Lw^k|_3+C|\overline{\varphi}^{k+1}|_2|\varphi^kLw^k|_\infty| \overline{w}^{k+1} |_2\\
\leq & C|\overline{\varphi}^{k+1}|^2_2+\frac{\alpha}{20}  | \varphi^{k+1}\nabla \overline{w}^{k+1} |^2_2
+C(1+ |\varphi^k\nabla^4 w^k|^2_2)|\overline{W}^{k+1} |^2_2,\\
J_5=&2\int \nabla (\overline{\varphi}^{k+1}(\varphi^{k+1}+\varphi^k))\cdot Q(w^k+ \widehat{u}) \cdot  \overline{w}^{k+1} \\
\leq& C  (|\nabla^2 w^k|_6+|\nabla^2 \widehat{u}|_6) |\varphi^{k+1}\overline{w}^{k+1} |_3 |\overline{\varphi}^{k+1}|_2\\
&+C(|\varphi^{k}\nabla^2 w^k|_\infty+|\varphi^{k}\nabla^2  \widehat{u}|_\infty) |\overline{w}^{k+1} |_2|\overline{\varphi}^{k+1}|_2\\
&+ C(|\nabla w^k|_\infty+|\nabla \widehat{u}|_\infty)| \overline{\varphi}^{k+1}|_2 |\varphi^{k+1}\nabla\overline{w}^{k+1} |_2+J^*_5,\\
J_6
=& 4\int  \varphi^k \nabla \varphi^k \cdot Q(\overline{w}^{k}) \cdot  \overline{w}^{k+1} \\
\leq & C|\nabla \varphi^k|_\infty|\varphi^{k} \nabla\overline{w}^{k} |_2|\overline{w}^{k+1}|_2
\leq \nu |\varphi^{k} \nabla\overline{w}^{k} |^2_2+C\nu^{-1} |\overline{w}^{k+1}|^2_2,\\
J_7=&-2\int B(\nabla  \widehat{u},\overline{W}^{k+1})\cdot \overline{W}^{k+1}
\leq C |\nabla  \widehat{u}|_\infty|\overline{W}^{k+1}|^2_2,\\
J_8=&-2\int \overline{\varphi}^{k+1}(\varphi^{k+1}+\varphi^k)L\widehat{u}\cdot \overline{w}^{k+1}
\leq  C|\overline{\varphi}^{k+1}|_2|\varphi^{k+1}+\varphi^k|_\infty|L \widehat{u}|_\infty|\overline{w}^{k+1}|_2,
\end{split}
\end{equation}
where   one has  used the inequality (\ref{qianru1}) for the term $|\varphi^kLw^k|_\infty$, and
\begin{equation}\label{ya7q}
\begin{split}
J^*_5
=&-2\int  \sum_{i,j} \overline{\varphi}^{k+1}(\varphi^k-\varphi^{k+1}+\varphi^{k+1}) Q_{ij}(w^k+ \widehat{u}) \partial_i (\overline{w}^{k+1})^{(j)}\\
\leq &C(|\nabla w^k|_\infty+|\nabla \widehat{u}|_\infty)\Big( |\nabla\overline{w}^{k+1} |_\infty|\overline{\varphi}^{k+1}|^2_2+|\varphi^{k+1} \nabla\overline{w}^{k+1} |_2|\overline{\varphi}^{k+1}|_2\Big).
\end{split}
\end{equation}

Then it follows from (\ref{zheng8})-(\ref{ya7q}) and Young's inequality that 
\begin{equation}\label{gogo1}\begin{split}
&\frac{d}{dt}\int |\overline{W}^{k+1}|^2+\alpha |\varphi^{k+1}\nabla\overline{w}^{k+1} |^2_2\\
\leq&C(\nu^{-1}+ |\varphi^k\nabla^4 w^k|^2_2)|\overline{W}^{k+1}|^2_2
+C|\overline{\varphi}^{k+1}|^2_{2}+\nu (|\varphi^k \nabla \overline{w}^k|^2_2+|\overline{\varphi}^k|_2^2+|\overline{W}^k|_2^2).
\end{split}
\end{equation}

Finally, define $$
\Gamma^{k+1}(t)=\sup_{s\in [0,t]}|\overline{W}^{k+1}(s)|^2_{2}+\sup_{s\in [0,t]}|\overline{\varphi}^{k+1}(s)|^2_{2}.
$$
It follows from  (\ref{varcha}) and (\ref{gogo1}) that 
\begin{equation*}\begin{split}
&\frac{d}{dt}\int \Big( |\overline{W}^{k+1}|^2+|\overline{\varphi}^{k+1}(t)|^2_{2}\Big)+\alpha |\varphi^{k+1}\nabla\overline{w}^{k+1} |^2_2\\
\leq& E^k_\nu (|\overline{W}^{k+1}|^2_{2}+|\overline{\varphi}^{k+1}|^2_{2})+ \nu (|\varphi^k \nabla \overline{w}^k|^2_2
+|\overline{\varphi}^k|_2^2+| \overline{W}^{k}|^2_2),
\end{split}
\end{equation*}
for some $E^k_\nu$ such that  
$$\int_{0}^{t}E^k_\nu(s)\text{d}s\leq C+C_\nu t\quad \text{for} \quad   0\leq t\leq T^{**}.
$$

Then  Gronwall's inequality implies that 
\begin{equation*}\begin{split}
&\Gamma^{k+1}+\int_{0}^{t}\alpha|\varphi^{k+1}\nabla\overline{w}^{k+1} |^2_2\text{d}s\\
\leq & \Big( \nu\int_{0}^{t}   |\varphi^k \nabla \overline{w}^k|^2_2\text{d}s+t\nu \sup_{s\in [0,t]} \big(| \overline{W}^{k}|^2_2
+| \overline{\varphi}^{k}|^2_2\big)\Big)\exp{(C+C_\nu t)}.
\end{split}
\end{equation*}
Choose $\nu_0>0$ and $T_*\in (0,\min (1,T^{**}))$ small enough such that
\begin{equation*}\begin{split}
\nu_0 \exp{C}&=\frac{1}{8} \min \big\{1, \alpha\big \}, \quad \exp (C_{\nu_0} T_*) \leq 2,
\end{split}\end{equation*}
which implies that
\begin{equation*}\begin{split}
\sum_{k=1}^{\infty}\Big( \Gamma^{k+1}(T_*)+\int_{0}^{T_*} \alpha |\varphi^{k+1}\nabla\overline{w}^{k+1} |^2_2\text{d}t\Big)\leq C<+\infty.
\end{split}
\end{equation*}
Thus, by the above estimate for $\Gamma^{k+1}(T_*)$ and (\ref{jkk}),   the whole sequence $(\varphi^k,W^k)$ converges to a limit $(\varphi, W)=(\varphi,\phi,w)$ in the following strong sense:
\begin{equation}\label{str}
\begin{split}
&(\varphi^k,W^k)\rightarrow (\varphi,W)\ \text{in}\ L^\infty([0,T_*];H^2(\mathbb{R}^3)).
\end{split}
\end{equation}
Due to  (\ref{jkk}) and the lower-continuity of norm for weak convergence,  $(\varphi,  W)$ still satisfies  (\ref{jkk}).
Now (\ref{str}) implies that $(\varphi,  W)$ satisfies
problem (\ref{li47-1}) in  the sense of distribution.
So the existence of  a classical solution is proved.

\textbf{Step 2:} Uniqueness.  Similarly to the proof of Lemma \ref{lem1q}, for any $t\in [0,T_*]$, the solution $(\varphi, \phi, w)$ is still compactly supported with respect to the space variable $x$, and  the size of their supports  is uniformly bounded by the constant $C_0(R,T_*)$.

 Let $(\varphi_1,W_1)$ and $(\varphi_2,W_2)$ be two  solutions  to  (\ref{li47-1}) satisfying the uniform a priori estimate (\ref{jkk}). Set
\begin{equation*}\begin{split}
&\overline{\varphi}=\varphi_1-\varphi_2,\quad \overline{W}=(\overline{\phi},\overline{w})=(\phi_1-\phi_2,w_1-w_2).
\end{split}
\end{equation*}
Then according to (\ref{eq:1.2w}), $(\overline{\varphi},\overline{W})$ solves the following system
 \begin{equation}
\label{zhuzhu}
\begin{cases}
\displaystyle
\ \ \overline{\varphi}_t+(w_1+\widehat{u})\cdot \nabla\overline{\varphi} +\overline{w}\cdot\nabla\varphi_{2}+\frac{\delta-1}{2}(\overline{\varphi} \text{div}(w_1+\widehat{u})+\varphi_{2}\text{div}\overline{w})=0,\\[10pt]
\ \ \overline{W}_t+\sum\limits_{j=1}^3A^*_j(W_1,\widehat{u})\partial_{j} \overline{W}+\varphi^2_1\mathbb{L}(\overline{w})
=-\sum\limits_{j=1}^3A_j(\overline{W})\partial_{j} W_{2} \\[10pt]
\ \ \ \ -\overline{\varphi}(\varphi_1+\varphi_2)\mathbb{L}(w_2) +\big(\mathbb{H}(\varphi_1)-\mathbb{H}(\varphi_2)\big)\cdot \mathbb{Q}(W_2)\\[10pt]
\ \ \ \ +\mathbb{H} (\varphi_1)\cdot \mathbb{Q}(\overline{W})-B(\nabla \widehat{u},\overline{w})-D(\overline{\varphi}(\varphi_1+\varphi_2),\nabla^2 \widehat{u}).
\end{cases}
\end{equation}

Using the same arguments as in the derivation of (\ref{varcha})-(\ref{gogo1}), and letting
$$
\Lambda(t)=|\overline{W}(t)|^2_{2}+|\overline{\varphi}(t)|^2_{2},
$$
one can get that 
\begin{equation}\label{gonm}
\displaystyle
\frac{d}{dt}\Lambda(t)+C|\varphi_1\nabla \overline{w}(t)|^2_2\leq I(t)\Lambda(t),
\end{equation}
for some $I(t)$ such that 
$$
\int_{0}^{t}I(s)ds\leq C\quad \text{for} \quad   0\leq t\leq T_*.
$$
So the Gronwall's inequality yields $\overline{\varphi}=\overline{\phi}=\overline{w}=0$.
Then the uniqueness is obtained.

\textbf{Step 3}. The time-continuity  can be obtained via the  same arguments as in Lemma \ref{lem1q}.

\section{Global-in-time well-posedness with compactly supported initial density under $(P_0)$ or $(P_1)$}

In this section,   we will consider the global-in-time well-posedness of classical solutions to the Cauchy problem  (\ref{li47-1}) with compactly supported $(\varphi_0,\phi_0)$.
To this end, we first rewrite
  (\ref{li47-1})   as
\begin{equation}\label{li47-2}
\begin{cases}
\displaystyle
\varphi_t+w\cdot\nabla\varphi+\frac{\delta-1}{2}\varphi\text{div} w=- \widehat{u} \cdot\nabla\varphi-\frac{\delta-1}{2}\varphi\text{div}  \widehat{u},\\[10pt]
\displaystyle
W_t+\sum_{j=1}^3A_j(W) \partial_j W+\varphi^2\mathbb{{L}}(w)=\mathbb{{H}}(\varphi)  \cdot \mathbb{{Q}}(w+ \widehat{u})+G^*(W, \varphi,  \widehat{u}),\\[10pt]
(\varphi, \phi, w)(t=0,x)=(\varphi_0,W_0)=(\varphi_0, \phi_0, 0),\quad x\in \mathbb{R}^3,\\[10pt]
(\varphi, \phi, w)\rightarrow (0,0,0) \quad\quad   \text{as}\quad \quad  |x|\rightarrow \infty \quad \text{for} \quad  t\geq 0,
 \end{cases}
\end{equation}
where
\begin{equation} \label{li47-2A}
\begin{split}
G^*(W, \varphi,  \widehat{u})=&-B(\nabla  \widehat{u},W)-\sum_{j=1}^3  \widehat{u}^{(j)}\partial_j W-D(\varphi^2,\nabla^2 \widehat{u}).
\end{split}
\end{equation}

Then the main result in this section can be stated as:

%

\begin{theorem}\label{ths1} Let  (\ref{canshu})
and any one of conditions $(P_0)$-$(P_3)$ hold.
 If the initial data $( \varphi_0, \phi_0,u_0)$ satisfies $(A_1)$-$(A_3)$ in Theorem \ref{newjie} and 
$$
 \|\phi_0\|_3+ \|\varphi_0\|_3\leq D_0(\alpha,\beta, \delta,  A,\gamma, \kappa, \|u_0\|_{\Xi}),$$
where $D_0>0$ is some constant depending on $(\alpha,\beta, \delta,  A,\gamma, \kappa, \|u_0\|_{\Xi})$, then for any  $T>0$,  there exists a unique global classical solution $(\varphi,  \phi, u)$ in $[0,T] \times \mathbb{R}^3$  to (\ref{li47-1}) satisfying
\begin{equation}\label{regghNN}\begin{split}
& (\varphi,\phi) \in C([0,T]; H^3), \    w\in C([0,T]; H^{s'}) \cap L^\infty([0,T]; H^3),\  \varphi\nabla^4w \in L^2([0,T]; L^2),
\end{split}
\end{equation}
for  any constant $s'\in[2,3)$.   Moreover, when $(P_2)$ holds, the smallness assumption on $\varphi_0$ could be removed.
\end{theorem}

We will  prove Theorem \ref{ths1} in the following  five Subsections $5.1$-$5.5$ via establishing a uniform-in-time weighted estimates under the condition $(P_0)$ or $(P_1)$ (see Theorem \ref{thglobal} and Remark \ref{P0}). The proof for the cases $(P2)$-$(P3)$  will be given in Section 6. Introduce a time weighed norm $Z(t)$ for  classical solutions as
\begin{equation}\label{fg}
\begin{split}
\displaystyle Y_k(t)=&|\nabla^k W(t)|_2,\qquad
 \displaystyle Y^2(t)=\sum^3_{k=0}(1+t)^{2\gamma_k} Y^2_k(t),\\
\displaystyle U_k(t)=&|\nabla^k \varphi(t)|_2,\qquad
 \displaystyle U^2(t)=\sum^3_{k=0}(1+t)^{2\delta_k} U^2_k(t),\\
 Z^2(t)=&Y^2(t)+U^2(t), \quad \gamma_k=k-n,\quad \delta_k=k-m,
\end{split}
\end{equation}
with  $(n, m)$ to be determined in Subsections 5.3-5.4. Denote also
$$
Z(0)=Z_0, \quad Y(0)=Y_0,\quad U(0)=U_0.
$$

Hereinafter,  $C\geq 1$ will denote  a generic constant depending only on fixed constants $(\alpha, \beta,\gamma, A, \delta, \kappa)$,  but independent of $(\varphi_0,W_0)$, which may be different from line to line, and  $C_0>0$  denotes a generic constant depending on $(C, \varphi_0,W_0)$. Specially, $C(l)$ (or $C_0(l)$)  denotes a generic positive constant depending on $(C,l)$ (or $(C, \varphi_0,W_0,l)$).

It then  follows from  Lemma \ref{lem2-2} that:
\begin{lemma}\label{lemK}
\begin{equation*}
\begin{split}
|W(t)|_\infty\leq& C(1+t)^{\frac{2n-3}{2}}Y(t),\quad  \text{and} \quad  |\nabla W(t)|_\infty\leq C(1+t)^{\frac{2n-5}{2}}Y(t),\\[2pt]
|\varphi(t)|_\infty\leq& C(1+t)^{\frac{2m-3}{2}} U(t),\quad \text{and} \quad  |\nabla \varphi(t)|_\infty\leq C(1+t)^{\frac{2m-5}{2}}  U(t).
\end{split}
\end{equation*}
\end{lemma}

\subsection{Energy estimates on $W$} First, applying  $\nabla^k$ to
$\eqref{li47-2}_2$, multiplying by $\nabla^k W $ and
integrating over $\mathbb{R}^3$, one can get
\begin{equation}\label{eq:2.3}
\begin{split}
&\frac{1}{2}\frac{d}{dt}|\nabla^k W|^2_2+\int \Big( \alpha \varphi^2 | \nabla^{k+1} w|^2+(\alpha+\beta) \varphi^2 |\text{div} \nabla^k w|^2\Big)\\
 =&R_k(W)+S_k(W,  \widehat{u})+ L_k(W,\varphi,  \widehat{u})+ Q_k(W,\varphi,  \widehat{u}),
\end{split}
\end{equation}
where
\begin{equation}\label{eq:2.7}
\begin{split}
R_k(W)=&-\int \nabla^k W\cdot \Big( \nabla^k\Big(\sum_{j=1}^3 A^j(W)\partial_j W\Big)-\sum_{j=1}^3 A^j(W)\partial_j \nabla^k W\Big)\\
&+\frac{1}{2}\int\sum_{j=1}^3\nabla^k W\cdot \partial_j A^j(W) \nabla^k W,\\
 S_k(W,  \widehat{u})=&-\int \nabla^k W\cdot  \nabla^kB(\nabla \widehat{u},W)+\frac{1}{2}\int \sum_{j=1}^3 \partial_j \widehat{u}^{(j)} \nabla^k W\cdot \nabla^k W\\
&-\int\nabla^k W\cdot \Big( \nabla^k\Big(\sum_{j=1}^3  \widehat{u}^{(j)}\partial_j W\Big)-\sum_{j=1}^3  \widehat{u}^{(j)}\partial_j \nabla^kW\Big),\\
L_k(W,\varphi,\widehat{u})=&-\int \Big(\nabla \varphi^2\cdot \mathbb{{S}}(\nabla^kw)- \big(\nabla^k(\varphi^2Lw)-\varphi^2L\nabla^kw\big)\Big)\cdot \nabla^k w\\
&+\int \nabla^k(\varphi^2L  \widehat{u})\cdot \nabla^k w \equiv: L^1_k+L^2_k+L^3_k,
\end{split}
\end{equation}
and
\begin{equation}\label{eq:2.7duan}
\begin{split}
Q_k(W, \varphi, \widehat{u})=&\int \Big( \nabla \varphi^2 \cdot Q(\nabla^kw)+\big(\nabla^k(\nabla \varphi^2\cdot Q(w))-\nabla \varphi^2\cdot Q(\nabla^kw)\big)\Big)\cdot \nabla^k w\\
&+\int \nabla^k\big(\nabla \varphi^2 \cdot Q( \widehat{u})\big)\cdot \nabla^k w
\equiv: Q^1_k+Q^2_k+Q^3_k.
\end{split}
\end{equation}

The right hand side of (\ref{eq:2.3}) will be estimated in the next lemmas.
\begin{lemma}[\textbf{Estimates on  $R_k$ and $S_k$}]\label{l1}
\begin{align}
\Big|R_k(W)(t, \cdot) \Big|\leq& C|\nabla W|_{\infty}Y^2_k, \label{eq:2.4}\\[8pt]
\frac{k+r}{1+t}Y^2_k+ S_k(W,  \widehat{u})(t, \cdot)\leq& C_0 Y_k Z (1+t)^{-\gamma_k-2},\label{eq:2.5}
\end{align}
for $k=0,1,2,3$, where the constant $r$ is given by:
 \begin{equation}\label{eq:2.6}
 r=-\frac{1}{2} \ \ \text{if} \ \ \gamma\geq \frac{5}{3},\ \
\text{or}\ \ \frac{3}{2}\gamma-3 \ \  \text{if} \ \  1< \gamma<\frac{5}{3}.
 \end{equation}
\end{lemma}

\begin{proof}\textbf{Step 1:} Estimates on  $R_k$.  Noticing that  $R_k$ is  a sum of terms as
$$\nabla^k W\cdot \nabla^{l}W\cdot \nabla^{k+1-l}W\quad \text{for}\quad 1\leq l\leq k,$$
then  (\ref{eq:2.4}) is obvious when $k=0$, or $1$.

For $k\neq 0$, $1$, one can  apply Lemma \ref{lem2-1} to $\nabla W$ to get  that
\begin{equation}\label{chafen}
|\nabla^jW|_{p_j}\leq C|\nabla  W|^{1-2/p_j}_\infty|\nabla^k W|^{2/p_j}_2,\quad \text{for} \quad p_j=2\frac{k-1}{j-1}.
\end{equation}
If $l\neq k$ and $l\neq 1$, since $1/p_l+1/p_{k-l+1}=\frac{1}{2}$, then H$\ddot{\text{o}}$lder's inequality implies
\begin{equation*}\begin{split}
\int |\nabla^k W \nabla^{l}W \nabla^{k+1-l}W|
&\leq  |\nabla^k W |_2|\nabla^{l}W |_{p_l}|\nabla^{k+1-l}W|_{p_{k-l+1}}
\leq C|\nabla W |_\infty|\nabla^{k}W |^2_2.
\end{split}
\end{equation*}
The other cases could be  handled similarly. Thus  (\ref{eq:2.4}) is proved.

\textbf{Step 2:} Estimates on  $S_k$. The intergrand  of $S_k(W,  \widehat{u})$ in $\eqref{eq:2.7}_2$ can be rewritten as
 \begin{equation}\label{eq:2.8}
 \begin{split}
 s_k(W,  \widehat{u})=&-\nabla^k W \cdot B(\nabla \widehat{u}, \nabla^k W)+\frac{1}{2}\sum_{j=1}^3 \partial_j  \widehat{u}^{(j)}  \nabla^k W \cdot \nabla^k W\\
 &-\nabla^k W\cdot \Big(\nabla^k(B(\nabla \widehat{u},W))-B(\nabla \widehat{u}, \nabla^k W)\Big)\\
&-\nabla^k W\cdot \Big( \nabla^k\Big(\sum_{j=1}^3  \widehat{u}^{(j)}\partial_j W\Big)-\sum_{j=1}^3  \widehat{u}^{(j)}\partial_j \nabla^kW\Big)\equiv: s^1_k+s^2_k,
 \end{split}
 \end{equation}
where $s^1_k$ is a sum of terms with a derivative of order one of $\widehat{u}$, and $s^2_k$ is a sum of terms with a derivative of order  at least two for $\widehat{u}$.

\textbf{Step 2.1:} Estimates on  $S^1_k=\int s^1_k$.
Let $\nabla^k=\partial_{\beta_1\beta_2...\beta_i...\beta_k}$ with $\beta_i=1,2,3$. Decompose $S^1_k$ as:
 \begin{equation}\label{a1}
 \begin{split}
 S^1_k=&\int \Big(-\nabla^k W \cdot B(\nabla \widehat{u}, \nabla^k W)+\frac{1}{2}\sum_{j=1}^3 \partial_j  \widehat{u}^{(j)}  \nabla^k W \cdot \nabla^k W\Big)\\
&-\int \partial_{\beta_1...\beta_k} W \cdot \sum_{i=1}^k \sum_{j=1}^3 \partial_{ \beta_i}\widehat{u}^{(j)} \partial_j\partial_{\beta_1...\beta_{i-1}\beta_{i+1}...\beta_k}W
=I_1+I_2+I_3,
 \end{split}
 \end{equation}
where,  from Proposition \ref{p1},  $I_1$-$I_3$ are given by
 \begin{equation}\label{ka1}
 \begin{split}
I_1=&-\frac{3(\gamma-1)}{2(1+t)}\int \nabla^k\phi \cdot \nabla^k \phi-\frac{1}{1+t}\int \nabla^k w\cdot \nabla^k w+G_1,\\
I_2=&\frac{3}{2}\frac{1}{1+t}Y^2_k+G_2,\quad
I_3= -\frac{k}{1+t}Y^2_k+G_3, \quad |G_j|\leq \frac{C_0}{(1+t)^2}Y^2_k,\quad j=1,2,3.
 \end{split}
 \end{equation}

 Therefore,
 \begin{equation}\label{eq:2.11}
 \begin{split}
S^1_k(W,  \widehat{u})(t, \cdot)
 \leq & \frac{C_0}{(1+t)^2}Y^2_k-\frac{A_k}{1+t}\int \nabla^k w\cdot \nabla^k w-\frac{B_k }{1+t}\int \nabla^k \phi\cdot \nabla^k \phi \\
\leq &C_0Y_k Z(1+t)^{-\gamma_k-2}-\frac{k+r}{1+t}Y^2_k,
 \end{split}
 \end{equation}
 where
 \begin{align}\label{ABK}
 A_k=k-\frac{1}{2},\qquad B_k=\frac{3}{2}\gamma-3+k, \quad r=\min\big(A_k,B_k\big)-k.
 \end{align}
 Then \eqref{eq:2.5} for $S^1_k$ follows from  \eqref{eq:2.11}-(\ref{ABK}).

\textbf{Step 2.2:} Estimates on  $S^2_k=\int s^2_k$.
$s^2_k$ is a sum of the terms  as 
\begin{equation*}\begin{split}
E_1(W)=\nabla^k W \cdot \nabla^l \widehat{u} \cdot \nabla^{k+1-l} W\quad \text{for}\quad  2\leq k\leq 3 \quad  \text{and}\quad  2\leq l\leq& k;\\
E_2(W)=\nabla^k W \cdot  \nabla^{l+1} \widehat{u}  \cdot \nabla^{k-l} W\quad \text{for}\quad  1\leq k\leq 3 \quad  \text{and}\quad  1\leq l\leq& k.
\end{split}
\end{equation*}
Hence,
\begin{equation}\label{zhon5r}
\begin{split}
 S^2_1(W)\leq & C| W|_\infty|\nabla^2\widehat{u}|_2|\nabla W|_2\leq C_0Y_1 Z (1+t)^{-2-\gamma_1},\\
S^2_2(W)\leq &C\big(|\nabla^2 \widehat{u}|_\infty |\nabla W|_2+|\nabla^3 \widehat{u}|_2|W|_\infty\big)|\nabla^2 W|_2
\leq C_0Y_2 Z (1+t)^{-2-\gamma_2},\\
S^2_3(W)\leq &C\big(|\nabla^2 \widehat{u}|_\infty|\nabla^{2}W|_2+|\nabla^3 \widehat{u}|_2|\nabla W|_\infty+|\nabla^4 \widehat{u}|_2|W|_\infty\big) |\nabla^3 W|_2\\
\leq & C_0Y_3 Z (1+t)^{-2-\gamma_3},
\end{split}
\end{equation}
which  imply that
$$
S^2_k(W)\leq
 C_0Y_k Z (1+t)^{-2-\gamma_k}, \quad \text{for} \quad  k=1,2,3.
$$

This, together with (\ref{eq:2.11}), yields \eqref{eq:2.5}-\eqref{eq:2.6}.
\end{proof}

 \begin{lemma}[\textbf{Estimates on  $L_k$}]\label{l2}
For any suitably small constant $\eta>0$, there are two constants  $C(\eta)$ and  $C_0(\eta)$ such that
\begin{equation}\label{dal}\begin{split}
 L_k(W) \leq& \eta |\varphi \nabla^{k+1} w|^2_2\delta_{3k}+C(\eta)(1+t)^{2m-5-\gamma_k}Z^{3}  Y_k\\
&+C_0(1+t)^{2m-n-4.5-\gamma_k}Z^{2}Y_k+C_0(\eta)(1+t)^{2m-10} Z^2.
\end{split}
\end{equation}
\end{lemma}
\begin{proof}
\textbf{Step 1:} Estimates on  $L^1_k$.  It is easy to check that,
\begin{equation}\label{Alk1}
\begin{split}
 L^1_k \leq &C |\varphi|_\infty|\nabla \varphi|_\infty  |\nabla^{k+1} w|_2  |\nabla^{k} w|_2 \leq C(1+t)^{2m-5-\gamma_k} Z^{3}Y_k, \ \text{for} \  k\leq 2,\\
 L^1_3 \leq & C |\varphi  \nabla^{4} w|_2 |\nabla \varphi|_\infty |\nabla^{3} w|_2 
\leq  \eta |\varphi  \nabla^{4} w|^2_2+C(\eta)(1+t)^{2m-5-\gamma_3} Z^{3}Y_3.
\end{split}
\end{equation}

\textbf{Step 2:} Estimates on  $L^3_k$. If $k=0$ or $1$, one has
\begin{equation}\label{lk1}
\begin{split}
L^3_0\leq & C|\varphi|_\infty|\nabla^2 \widehat{u}|_\infty|\varphi|_2|w|_2
\leq C_0(1+t)^{2m-n-4.5-\gamma_0}Z^{2}Y_0,\\
 L^3_1 \leq & C\big(  |\nabla\varphi|_2 |\nabla^2 \widehat{u}|_\infty+|\varphi|^{}_\infty  |\nabla^3 \widehat{u}|_2\big) |\varphi|^{}_\infty |\nabla w|_2\\
\leq&  C_0(1+t)^{2m-n-4.5-\gamma_1}Z^{2}Y_1.
\end{split}
\end{equation}
For the case  $k=2$,  decompose 
$
L^3_2\triangleq L^3_2(0,2)+L^3_2(1,1)+L^3_2(2,0)
$.
One can get
\begin{equation}\label{lk5}
\begin{split}
 L^3_2(0,2) \triangleq &\int \varphi^2  \nabla^2 L  \widehat{u}\cdot \nabla^2 w
\leq   C | \varphi|^2_\infty |\nabla^2 L\widehat{u}|_2|\nabla^2 w|_2\\
\leq &  C_0(1+t)^{2m-n-4.5-\gamma_2}Z^{2}  Y_2,\\
 L^3_2(1,1) \triangleq &\int \nabla\varphi^2 \cdot  \nabla L  \widehat{u}\cdot \nabla^2 w
\leq  C| \varphi|_{\infty}|\nabla \varphi|_\infty |\nabla L\widehat{u}|_2|\nabla^2 w|_2\\
\leq&  C_0(1+t)^{2m-n-4.5-\gamma_2}Z^{2}  Y_2,\\
 L^3_2(2,0) \triangleq &\int  \nabla^2\varphi^2 \cdot  L  \widehat{u}\cdot \nabla^2 w
\leq   C (|\nabla \varphi|^2_\infty|L\widehat{u}|_2+| \varphi|_{\infty} |\nabla^2 \varphi|_2|L\widehat{u}|_\infty) |\nabla^2 w|_2\\
\leq &  C_0(1+t)^{2m-n-4.5-\gamma_2}Z^{2}  Y_2.
\end{split}
\end{equation}

For the case $k=3$, decompose
$
L^3_3 \triangleq L^3_3(0,3)+L^3_3(1,2)+L^3_3(2,1)+L^3_3(3,0)
$.
Then,  by integration by parts, one can obtain

\begin{equation}\label{Alk4}
\begin{split}
 L^3_3(0,3)\triangleq &\int \varphi^2 \cdot  \nabla^3 L  \widehat{u}\cdot \nabla^3 w
\leq   C| \varphi|_{\infty} |\nabla^{2} L\widehat{u}|_2\big(|\varphi\nabla^{4} w|_2+|\nabla \varphi|_\infty |\nabla^{3} w|_2\big)\\
\leq  & \eta |\varphi  \nabla^{4} W|^2_2+C_0(1+t)^{2m-n-4.5-\gamma_3}Z^2  Y_3+C_0(\eta)(1+t)^{2m-10}Z^{2},\\
 L^3_3(1,2)\triangleq &\int  \nabla\varphi^2 \cdot  \nabla^{2}L  \widehat{u}\cdot \nabla^3 w
\leq   C| \varphi|_{\infty}|\nabla \varphi|_\infty |\nabla^{2} L\widehat{u}|_2|\nabla^{3} w|_2\\
\leq& C_0(1+t)^{2m-n-4.5-\gamma_3}Z^2  Y_3,\\
 L^3_3(2,1)\triangleq &\int \nabla^{2}\varphi^2 \cdot \nabla L  \widehat{u}\cdot \nabla^3 w
\leq   C\big(|\nabla\varphi|^{2}_\infty|\nabla^{3}\widehat{u}|_2+|\varphi|_\infty|\nabla^{2}\varphi|_{6}|\nabla^{3}\widehat{u}|_3\big)|\nabla^3 w|_2\\
\leq &  C_0(1+t)^{2m-n-4.5-\gamma_3}Z^{2}  Y_3,\\
 L^3_3(3,0)\triangleq &\int  \nabla^3\varphi^2 \cdot  L  \widehat{u}\cdot \nabla^3 w
\leq   C\big(| \varphi|^{}_{\infty} |\nabla^3 \varphi|_2+ |\nabla \varphi|_\infty|\nabla^2 \varphi|_2\big) |L\widehat{u}|_\infty |\nabla^3 w|_2\\
\leq &  C_0(1+t)^{2m-n-4.5-\gamma_3}Z^{2}  Y_3,
\end{split}
\end{equation}
with $\eta>0$ being any  sufficiently small constant. It follows from (\ref{lk1})-(\ref{Alk4}) that 
\begin{equation}\begin{split}\label{Azongjie1}
 L^3_k\leq& \eta |\varphi \nabla^{k+1} w|^2_2\delta_{3k}+C_0(1+t)^{2m-n-4.5-\gamma_k}Z^2 Y_k+C_0(\eta)(1+t)^{2m-10}Z^{2}.
\end{split}
\end{equation}

\textbf{Step 3:}
Estimates on $L^2_k$. If  $k= 1$, one gets
\begin{equation}\label{Alk12}
\begin{split}
 L^2_1 = &\int \nabla\varphi^2 \cdot Lw\cdot \nabla w
\leq  C|\varphi |_\infty|\nabla\varphi|_\infty|\nabla^2 w|_2| \nabla w|_2
\leq C(1+t)^{2m-5-\gamma_1} Z^3Y_1.
\end{split}
\end{equation}
Next  for $k= 2$, decompose $L^2_2\triangleq L^2_2(1,1)+L^2_2(2,0)$.    In a similar way for  $L^1_2$, one has
\begin{equation}\label{Alk13}
\begin{split}
L^2_2(1,1)\triangleq &\int  \nabla\varphi^2 \cdot \nabla Lw\cdot \nabla^2 w
\leq C(1+t)^{2m-5-\gamma_2}Z^3Y_2,\\
 L^2_2(2,0)\triangleq &\int  \nabla^2\varphi^2 \cdot Lw\cdot \nabla^2 w
\leq  C(|\nabla\varphi|^2_\infty|\nabla^2 w|_2+|\varphi|_6|\nabla^2 w|_6|\nabla^{2}\varphi|_6)|\nabla^{2} w|_2\\
\leq & C(1+t)^{2m-5-\gamma_2} Z^{3}Y_2.
\end{split}
\end{equation}
At last for $k=3$, decompose $L^2_3 \triangleq L^2_3(1,2)+L^2_3(2,1)+L^2_3(3,0)$. In a similar way for $L^1_3$, one can get
\begin{equation}\label{Alk14}
\begin{split}
 L^2_3(1,2)\triangleq &\int \nabla\varphi^2 \cdot \nabla^2 Lw\cdot \nabla^3 w
\leq  \eta |\varphi  \nabla^{4} w|^2_2+C(\eta)(1+t)^{2m-5-\gamma_3}Z^3Y_3,\\
L^2_3(2,1)\triangleq &\int  \nabla^2\varphi^2 \cdot \nabla Lw\cdot \nabla^3 w
\leq  C( |\nabla\varphi|^2_\infty | \nabla^3 w|_2+| \nabla^2\varphi|_3|\varphi \nabla Lw|_6) |\nabla^3 w|_2\\
\leq&  \eta |\varphi  \nabla^{4} w|^2_2+C(\eta)(1+t)^{2m-5-\gamma_3}Z^3 Y_3,\\
 L^2_3(3,0)\triangleq &\int  \nabla^3 \varphi^2 \cdot  Lw\cdot \nabla^{3} w\\
\leq & C|\nabla\varphi|_\infty|\nabla^2 \varphi|_3|\nabla^3 w|_2|\nabla^2 w|_6
+C|\varphi\nabla^3 w|_6|\nabla^{2}w|_3|\nabla^{3} \varphi|_2\\
\leq &\eta |\varphi \nabla^{4} w|^2_2+C(\eta)(1+t)^{2m-5-\gamma_3} Z^3Y_3.
\end{split}
\end{equation}

Then combining the estimates (\ref{Alk12})-(\ref{Alk14}) yields
\begin{equation}\begin{split}\label{zongjie2}
 L^2_k\leq& \eta |\varphi  \nabla^{k+1} w|^2_2\delta_{3k}+C(\eta)(1+t)^{2m-5-\gamma_k}Z^3Y_k.
\end{split}
\end{equation}

Thus \eqref{dal} follows from above three steps.

\end{proof}

\begin{lemma}[\textbf{Estimates on  $Q_k$}]\label{Al2-2}
For any suitably small constant $\eta>0$, there are two constants  $C(\eta)$ and  $C_0(\eta)$  such that
\begin{equation}\label{forq}\begin{split}
Q_k(W,  \widehat{u}) \leq& \eta |\varphi \nabla^{k+1} w|^2_2\delta_{3k}+C(\eta)(1+t)^{2m-5-\gamma_k}Z^{3}  Y_k\\
&+C_0(1+t)^{2m-n-3.5-\gamma_k}Z^{2}Y_k+C_0(\eta)(1+t)^{2m-10} Z^2\\
&+\frac{(4\alpha+6\beta)\delta}{(\delta-1)(1+t)}|\nabla^3\varphi|_2| \varphi\nabla^3 \text{div}w|_2.
\end{split}
\end{equation}
\end{lemma}

\begin{proof}
\textbf{Step 1:}
Estimates on  $Q^1_k$.  In a similar way for   $L^1_k$,   it is easy to get
\begin{equation*}
\begin{split}
 Q^1_k \leq&  \eta |\varphi \nabla^{k+1} w|^2_2\delta_{3,k}+ C(1+t)^{2m-5-\gamma_k} Z^{3}Y_k.
\end{split}
\end{equation*}

\textbf{Step 2:}  Estimates on  $Q^k_3$. If $k=0$, one has
\begin{equation}\label{lk1q}
\begin{split}
 Q^3_0\leq & C|\varphi|_\infty|\nabla \widehat{u}|_6|\nabla \varphi|_3|w|_2
\leq C_0(1+t)^{2m-n-4.5-\gamma_0}Z^{2}Y_0.
\end{split}
\end{equation}

For  $k=1$,  in a similar way for   $L^3_1$, one obtains
\begin{equation}\label{lk2q}
\begin{split}
 Q^3_1 \leq &C\Big(|\varphi|_\infty |\nabla^2\varphi|_3 |\nabla \widehat{u}|_6+ |\nabla\varphi|^2_6 |\nabla \widehat{u}|_6+ |\varphi|_\infty |\nabla\varphi|_2 |\nabla^2 \widehat{u}|_\infty\Big)|\nabla w|_2\\
\leq & C_0(1+t)^{2m-n-4.5-\gamma_1}Z^{2}Y_1.
\end{split}
\end{equation}

Next for  $k=2$, decompose 
$
Q^3_2\triangleq Q^3_2(0,2)+Q^3_2(1,1)+Q^3_2(2,0)
$.
Then   as  for  $L^3_2(1,1)$ and $L^3_2(2,0)$, one has
\begin{equation}\label{lk5q}
\begin{split}
Q^3_2(0,2)\triangleq &C\int \nabla \varphi^2 \cdot  \nabla^3   \widehat{u}\cdot \nabla^2 w
\leq  C_0(1+t)^{2m-n-4.5-\gamma_2}Z^{2}  Y_2,\\
Q^3_2(1,1)\triangleq&C\int \nabla^2\varphi^2  \cdot \nabla^2  \widehat{u}\cdot \nabla^2 w
\leq   C_0(1+t)^{2m-n-4.5-\gamma_2}Z^{2}  Y_2,\\
Q^3_2(2,0)\triangleq &C\int  \nabla^3\varphi^2 \cdot \nabla  \widehat{u}\cdot \nabla^2 w
\leq   C\big(  |\nabla^3 \varphi|_2  |\varphi|_\infty + |\nabla^2 \varphi|_6|\nabla \varphi|_3\big)|\nabla \widehat{u}|_\infty |\nabla^2 w|_2\\
\leq &
C_0(1+t)^{2m-n-3.5-\gamma_2}Z^{2}Y_2.
\end{split}
\end{equation}

For  $k=3$, let
$
Q^3_3\triangleq Q^3_3(0,3)+Q^3_3(1,2)+Q^3_3(2,1)+Q^3_4(3,0)
$.
Then as for   $L^3_3(1,2)$, $L^3_3(2,1)$ and  $L^3_3(3,0)$, one can get
\begin{equation}\label{lk4q}
\begin{split}
\sum_{i=0}^2Q^3_3(i,3-i)  \triangleq &C\int \nabla \varphi^2\cdot   \nabla^4   \widehat{u}\cdot \nabla^3 w
+C\int \nabla^{2}\varphi \cdot  \nabla^3  \widehat{u}\cdot \nabla^3 w
\\
&+C\int  \nabla^3\varphi^2  \cdot \nabla^2  \widehat{u}\cdot \nabla^3 w
\leq   C_0(1+t)^{2m-n-4.5-\gamma_3}Z^2  Y_3.
\end{split}
\end{equation}

Finally,  it remains to handle the term $Q^3_3(3,0)$ defined below, for which some additional information is needed. Using integration by parts and Proposition \ref{p1}, one can get

\begin{equation*}
\begin{split}
Q^3_3(3,0)\triangleq &\frac{\delta}{\delta-1}\sum_{i,j=1}^3\int \Big( \alpha \nabla^{3}\partial_j\varphi^2 ( \partial_i \widehat{u}^{(j)}+\partial_j \widehat{u}^{(i)})+\beta \nabla^3\partial_i\varphi^2\text{div}\widehat{u}\Big)\cdot \nabla^3 w^{(i)}\\
=& Q^3_3(A)-\frac{\delta}{\delta-1}\sum_{i,j=1}^3\int  \alpha \nabla^{3}\varphi^2 ( \partial_i \widehat{u}^{(j)}+\partial_j \widehat{u}^{(i)})\cdot\nabla^3 \partial_jw^{(i)}\\
&-\frac{\delta}{\delta-1}\int \Big( \beta \nabla^3\varphi^2\text{div}\widehat{u}\Big)\cdot \nabla^3 \text{div}w\\
=& Q^3_3(A)+Q^3_3(B)+Q^3_3(D)-\frac{\delta}{(\delta-1)(1+t)}\int (4\alpha+6\beta) \varphi \nabla^3\varphi\cdot \nabla^3 \text{div}w\\
\leq & Q^3_3(A)+Q^3_3(B)+Q^3_3(D)+\frac{(4\alpha+6\beta)\delta}{(\delta-1)(1+t)}|\nabla^3\varphi|_2| \varphi\nabla^3 \text{div}w|_2,
\end{split}
\end{equation*}
where
\begin{equation}\label{lk8eeq2}
\begin{split}
Q^3_3(A)=& C \int  \nabla^{3}\varphi^2\cdot   \nabla^2  \widehat{u}\cdot \nabla^3 w
= C\int \Big(\varphi \nabla^3 \varphi +\nabla \varphi \cdot \nabla^2 \varphi
\Big)\cdot   \nabla^2  \widehat{u}\cdot \nabla^3 w\\
\leq& C\big(| \varphi|_\infty|\nabla^3 \varphi|_2+|\nabla \varphi|_\infty|\nabla^2 \varphi|_2 \big)|\nabla^2 \widehat{u}|_\infty |\nabla^3 w|_2\\
\leq &C_0(1+t)^{2m-n-4.5-\gamma_3}Z^2  Y_3,\\
\end{split}
\end{equation}
and
\begin{equation}\label{lk8eeq1}
\begin{split}
Q^3_3(B)
=& C\int \nabla \varphi \cdot \nabla^2 \varphi \cdot \nabla  \widehat{u}\cdot \nabla^{4} w\\
=&C\int \Big( \big(\nabla \varphi \cdot \nabla^3\varphi + \nabla^2\varphi \cdot \nabla^2 \varphi  \big)\cdot \nabla  \widehat{u}+\nabla \varphi \cdot \nabla^2 \varphi\cdot \nabla^2  \widehat{u}\Big)\cdot \nabla^{3} w\\
\leq &C\big(|\nabla \varphi|_\infty |\nabla^3\varphi|_2  |\nabla  \widehat{u}|_\infty+ |\nabla^2\varphi|_6\big(  |\nabla^2 \varphi  |_6 |\nabla  \widehat{u}|_6+|\nabla \varphi|_3|\nabla^2 \widehat{u}|_\infty\big)\big)|\nabla^{3}w|_2\\
\leq & C_0(1+t)^{2m-n-3.5-\gamma_3}Z^2  Y_3,\\
Q^3_3(D)= &-\frac{\delta}{(\delta-1)(1+t)^2}\sum_{i,j=1}^3\int  2\alpha \varphi \nabla^{3}\varphi ( K_{ij}+K_{ji})\cdot\nabla^3 \partial_jw^{(i)}\\
&-\frac{\delta}{(\delta-1)(1+t)^2}\sum_{i=1}^3\int 2 \beta \varphi \nabla^3\varphi K_{ii}\cdot \nabla^3 \text{div}w\\
\leq & \eta |\varphi  \nabla^{4} w|^2_2+C_0(\eta)(1+t)^{2m-10} Z^{2}.
\end{split}
\end{equation}
Then combining the estimates (\ref{lk1q})-(\ref{lk8eeq1}) yields
\begin{equation}\begin{split}\label{zongjie3}
 Q^3_k\leq& \eta |\varphi \nabla^{k+1} w|^2_2\delta_{3k}+C_0(1+t)^{2m-n-3.5-\gamma_k}Z^{2}  Y_k\\
&+C_0(\eta)(1+t)^{2m-10} Z^{2}+\frac{(4\alpha+6\beta)\delta}{(\delta-1)(1+t)}|\nabla^3\varphi|_2| \varphi\nabla^3 \text{div}w|_2.
\end{split}
\end{equation}

\textbf{Step 3:}
Estimates on $Q^2_k$. For $k=1$,  direct estimates give
\begin{equation}\label{Alk15e}
\begin{split}
Q^2_1
=&C\int \big(\varphi \nabla^2 \varphi+\nabla \varphi \cdot \nabla \varphi \big)\cdot \nabla w\cdot \nabla w\\
\leq & C(|\varphi|_3|\nabla^2 \varphi|_6 | \nabla w|_\infty+|\nabla \varphi|^2_\infty |\nabla w|_2\big)|\nabla w|_2
\leq  C(1+t)^{2m-5-\gamma_1} Z^3Y_1.
\end{split}
\end{equation}
For $k=2$, in a similar way for   $L^2_2(2,0)$, one gets easily
\begin{equation}\label{lk15ee}
\begin{split}
 Q^2_2
=&C\int \big(\nabla^3 \varphi^2\cdot \nabla w+\nabla^2 \varphi^2\cdot \nabla^2 w\big)\cdot \nabla^2 w\\
\leq & C\big(|\varphi|_\infty|\nabla^3 \varphi|_2|\nabla w|_\infty+|\nabla w|_\infty |\nabla \varphi|_6 |\nabla^2 \varphi|_3 \big)|\nabla^2 w|_2\\
&+C(1+t)^{2m-5-\gamma_2} Z^{3}Y_2
\leq C(1+t)^{2m-5-\gamma_2} Z^{3}Y_2.
\end{split}
\end{equation}
For $k=3$, as for  $L^2_3(3,0)$ and $L^2_3(2,1)$, it follows from  integration by parts that 
\begin{equation}\label{lk15eee}
\begin{split}
Q^2_3
=&C\int \big(\nabla^4 \varphi^2\cdot \nabla w+\nabla^3 \varphi^2\cdot \nabla^2 w+\nabla^2 \varphi^2\cdot \nabla^3 w\big)\cdot \nabla^3 w\\
\leq & \eta |\varphi  \nabla^{4} w|^2_2+C(\eta)(1+t)^{2m-5-\gamma_3}Z^3 Y_3+Q^2_3(A),
\end{split}
\end{equation}
with
\begin{equation}\label{lk45}
\begin{split}
Q^2_3(A)\triangleq & C\int  \nabla^{3}\varphi^2\cdot   \nabla w\cdot \nabla^{4} w
=C \int \Big(\varphi \nabla^3 \varphi +\nabla \varphi \cdot  \nabla^2\varphi \Big)\cdot \nabla w\cdot \nabla^{4} w\\
\leq &C |\nabla^3 \varphi|_2|\nabla w|_\infty|\varphi \nabla^4w|_2+Q^2_3(B)\\
\leq & \eta |\varphi  \nabla^{4} w|^2_2+C(\eta)(1+t)^{2m-5-\gamma_3}Z^3 Y_3+Q^2_3(B),
\end{split}
\end{equation}
where the term $Q^2_3(B)$ can be estimated by   using  integration by parts again, 
\begin{equation}\label{Alk45ee}
\begin{split}
Q^2_3(B)
\triangleq & C\int \nabla \varphi \cdot \nabla^2\varphi \cdot \nabla w  \cdot \nabla^{4} w\\
=&C\int  \Big(\nabla^2 \varphi \cdot \nabla^2\varphi \cdot \nabla w+\nabla \varphi \cdot \nabla^3\varphi \cdot \nabla w+\nabla \varphi \cdot \nabla^2\varphi \cdot \nabla^2 w\Big)\cdot \nabla^{3} w\\
\leq &C\big(|\nabla^2 \varphi|^2_6 |\nabla w|_6 +|\nabla^3 \varphi|_2|\nabla \varphi|_\infty |\nabla w|_\infty+ |\nabla\varphi|_6 |\nabla^2 \varphi|_6 |\nabla^2 w|_6\big)|\nabla^3 w|_2\\
\leq &  C(1+t)^{2m-5-\gamma_3} Z^3Y_3.
\end{split}
\end{equation}
Then collecting the estimates (\ref{Alk15e})-(\ref{Alk45ee}) shows that 
\begin{equation}\begin{split}\label{zongjie4}
Q^2_k\leq& \eta |\varphi \nabla^{k+1} w|^2_2\delta_{3k}+C(\eta)(1+t)^{2m-5-\gamma_k}Z^{3}  Y_k.
\end{split}
\end{equation}

Then  \eqref{forq} follows directly from above estimates.

\end{proof}

It follows from (\ref{eq:2.3}) and  Lemmas \ref{l1}-\ref{Al2-2} that 
\begin{lemma}\label{Azongjie1}
\begin{equation}\label{Aeq:2.13}
\begin{split}
&\frac{1}{2}\frac{d}{dt} Y^2_k+\frac{k+r}{1+t} Y^2_k+\alpha\int   | \varphi\nabla^{k+1} w|^2+(\alpha+\beta) \int   | \varphi\nabla^{k}\text{div} w|^2\\
 \leq& C(\eta)(1+t)^{2m-5-\gamma_k}Z^{3}  Y_k+C_0\Big((1+t)^{2m-n-3.5-\gamma_k}+(1+t)^{n-2.5-\gamma_k}\Big)Z^{2}Y_k\\
&+C_0(\eta)\Big((1+t)^{2m-10} Z^2+(1+t)^{-\gamma_k-2}ZY_k\Big)\\
&+\eta |\varphi \nabla^{k+1} w|^2_2\delta_{3k}+\frac{(4\alpha+6\beta)\delta}{(\delta-1)(1+t)}|\nabla^3\varphi|_2| \varphi\nabla^k \text{div}w|_2\delta_{3k}.
\end{split}
\end{equation}

\end{lemma}

\subsection{ Estimates on $\varphi$}
$\varphi$ is estimated in the following lemma.

\begin{lemma}\label{Azongjiephi}
\begin{equation}\label{Aeq:2.14qq}
\begin{split}
&\frac{1}{2}\frac{d}{dt}|\nabla^k \varphi |^2_2+\frac{\frac{3}{2}\delta-3+k}{1+t}|\nabla^k \varphi |^2_2\\
\leq& C(1+t)^{n-2.5-\delta_k}Z^2U_k+C_0 (1+t)^{-2-\delta_k}ZU_k+\frac{\delta-1}{2}|\varphi \nabla^3 \text{div}w|_2 |\nabla^3 \varphi|_2\delta_{3k}.
\end{split}
\end{equation}
\end{lemma}

\begin{proof}

\textbf{Step 1.}
Applying $\nabla^k$ to $\eqref{li47-2}_1$, multiplying by $\nabla^k \varphi $ and
integrating over $\mathbb{R}^3$, one gets
\begin{equation}\label{Aeq:var1}
\begin{split}
\frac{1}{2}\frac{d}{dt}|\nabla^k \varphi|^2_2=\int  \Big(w\cdot  \nabla^{k+1}  {\varphi}+\frac{\delta-1}{2}\nabla^k  \varphi \text{div}w\Big)\cdot \nabla^k  {\varphi}+S^*_k(\varphi,  \widehat{u})-\Lambda^k_1-\Lambda^k_2,
\end{split}
\end{equation}
where $S^*_k(\varphi,  \widehat{u})$ is given in Step 2 below, and
\begin{equation*}
\begin{split}
\Lambda^k_1=&\Big(\nabla^k (w\cdot \nabla {\varphi})-w\cdot  \nabla^{k+1}  {\varphi}\Big)\cdot \nabla^k \varphi,\quad \Lambda^k_2= \frac{\delta-1}{2}\Big(\nabla^k (\varphi \text{div}w)-\nabla^k  \varphi \text{div}w\Big)\cdot \nabla^k \varphi.
\end{split}
\end{equation*}

Integration by parts yields immediately that 
 \begin{equation}\label{eq:2.8fr}
 \begin{split}
\int  \Big(w\cdot  \nabla^{k+1}  {\varphi}+\frac{\delta-1}{2}\nabla^k  \varphi \text{div}w\Big)\cdot \nabla^k  {\varphi}\leq  C|\nabla w|_\infty|\nabla^k \varphi(t)|^2_2.
 \end{split}
 \end{equation}

\textbf{Step 2:} Estimates on  $S^*_k=\int s^*_k$ with the  integrand defined as 
 \begin{equation}\label{eq:2.8ss}
 \begin{split}
 s^*_k(\varphi,  \widehat{u})=&-\frac{\delta-1}{2}\nabla^k \varphi \cdot \nabla^k \varphi \text{div}\widehat{u}+\frac{1}{2}\text{div}\widehat{u}  \nabla^k \varphi  \cdot \nabla^k \varphi \\
 &-\frac{\delta-1}{2}\nabla^k \varphi \cdot \Big(\nabla^k(\varphi \text{div}\widehat{u})-\text{div}\widehat{u}\nabla^k \varphi \Big)\\
&-\nabla^k \varphi \cdot \Big( \nabla^k\big( \widehat{u} \cdot \nabla \varphi\big)- \widehat{u} \cdot \nabla^{k+1} \varphi \Big)\equiv: s^{*1}_k+s^{*2}_k,
 \end{split}
 \end{equation}
where $s^{*1}_k$ is a sum of terms with a derivative of order one of $\widehat{u}$, and $s^{*2}_k$ is a sum of terms with a derivative of order  at least two of $\widehat{u}$.

\textbf{Step 2.1}: Estimates on  $S^{*1}_k=\int s^{*1}_k$.
Let $\nabla^k=\partial_{\beta_1...\beta_k}$ with $0\leq \beta_i=1,2,3$. 
 \begin{equation}\label{a1ss}
 \begin{split}
S^{*1}_k=&\int \Big(-\frac{\delta-1}{2}\nabla^k \varphi \cdot \nabla^k \varphi \text{div}\widehat{u}+\frac{1}{2}\text{div}\widehat{u}  \nabla^k \varphi  \cdot \nabla^k \varphi\Big)\\
&-\int \partial_{\beta_1...\beta_k} \varphi \cdot \sum_{i=1}^k \sum_{j=1}^k \partial_{ \beta_i}\widehat{u}^{(j)} \partial_j\partial_{\beta_1...\beta_{i-1}\beta_{i+1}...\beta_k}\varphi =J_1+J_2+J_3,
 \end{split}
 \end{equation}
where, by Proposition \ref{p1}, $J_1$-$J_3$ are estimated  by
 \begin{equation}\label{ka1ss}
 \begin{split}
J_1=&-\frac{3(\delta-1)}{2(1+t)}U^2_k+N_1,\quad
J_2=\frac{3}{2}\frac{1}{1+t}U^2_k+N_2,\\
J_3=& -\frac{k}{1+t}U^2_k+N_3, \quad |N_j|\leq \frac{C_0}{(1+t)^2}U^2_k,\quad j=1,2,3.
 \end{split}
 \end{equation}

 Therefore, it holds that 
 \begin{equation}\label{eq:2.11ss}
 \begin{split}
 & S^{*1}_k(W,  \widehat{u})(t, \cdot)\text{d} x
 \leq  \frac{C_0}{(1+t)^2}U^2_k-\frac{\frac{3}{2}\delta-3+k}{1+t}U^2_k.
 \end{split}
 \end{equation}

\textbf{Step 2.2:} Estimates on  $S^{*2}_k=\int s^{*2}_k$.
$s^{*2}_k$ is a sum of the terms defined as
$$
E_1(\varphi)=\nabla^k \varphi \cdot \nabla^l \widehat{u} \cdot  \nabla^{k+1-l} \varphi\quad \text{for}\quad  2\leq k\leq 3 \quad  \text{and}\quad  2\leq l\leq k;
$$
$$
E_2(\varphi)=\nabla^k \varphi \cdot \nabla^{l+1} \widehat{u} \cdot \nabla^{k-l} \varphi\quad \text{for}\quad  1\leq k\leq 3 \quad  \text{and}\quad  1\leq l\leq k.
$$
Then it holds that 
\begin{equation}\label{zhon5r}
\begin{split}
S^{*2}_1(\varphi)\leq & C| \varphi|_6|\nabla^2\widehat{u}|_2|\nabla \varphi|_3
\leq C_0 (1+t)^{-2-\delta_1}ZU_1,\\
 S^{*2}_2(\varphi)\leq &C(|\nabla^2 \widehat{u}|_\infty |\nabla \varphi|_2+|\nabla^3 \widehat{u}|_2|\varphi|_\infty) |\nabla^2 \varphi|_2
\leq  C_0 (1+t)^{-2-\delta_2}ZU_2,\\
 S^{*2}_3(\varphi)\leq & C\big(|\nabla^2 \widehat{u}|_\infty|\nabla^{2}\varphi|_2+|\nabla^3 \widehat{u}|_6|\nabla \varphi|_3+|\nabla^4 \widehat{u}|_2|\varphi|_\infty\big) |\nabla^3 \varphi|_2\\
 \leq&  C_0 (1+t)^{-2-\delta_3}ZU_3,
\end{split}
\end{equation}
which  implies immediately  that
$$
S^{*2}_k(\varphi)\leq  C_0 (1+t)^{-2-\delta_k}ZU_k, \quad \text{for} \quad  k=1,2,3.
$$

\textbf{Step 3:} Estimates on  $ \Lambda^k_1+\Lambda^k_2$.
It follows from Lemma  \ref{l1}  and  H\"older's inequality that 
\begin{equation}\label{Aqe1a}
\begin{split}
 \Lambda^1_1+&\Lambda^1_2
\leq   C|\nabla w|_\infty|\nabla \varphi(t)|^2_2
\leq C(1+t)^{n-2.5-\delta_1} Z^2U_1,\\
 \Lambda^2_1+&\Lambda^2_2
\leq  C(|\nabla \varphi|_\infty|\nabla^2w(t)|_2+|\nabla w|_\infty|\nabla^2\varphi(t)|_2) |\nabla^2\varphi(t)|_2\\
\leq & C(1+t)^{n-2.5-\delta_2} Z^2U_2,\\
 \Lambda^3_1
\leq &  C\big(|\nabla \varphi|_\infty|\nabla^3w(t)|_2+|\nabla^2 \varphi|_6|\nabla^2w(t)|_3+|\nabla w|_\infty|\nabla^3\varphi(t)|_2\big) |\nabla^3\varphi(t)|_2\\
\leq&C(1+t)^{n-2.5-\delta_3} Z^2U_3, \\
\Lambda^3_2\leq &  C\Big(|\nabla \varphi|_\infty|\nabla^3w(t)|_2+|\nabla^2 \varphi|_6|\nabla^2w(t)|_3+\frac{\delta-1}{2}|\varphi \nabla^3 \text{div}w|_2\Big) |\nabla^3 \varphi|_2\\
\leq& C(1+t)^{n-2.5-\delta_3} Z^2U_3+\frac{\delta-1}{2}|\varphi \nabla^3 \text{div}w|_2 |\nabla^3 \varphi|_2.
\end{split}
\end{equation}

Then  (\ref{Aeq:var1})-(\ref{Aqe1a}) yield the desired (\ref{Aeq:2.14qq}).

\end{proof}

Finally, set
\begin{equation}\label{gab}\begin{split}
H(A^*, A,B)&=\alpha (A^*)^2+(\alpha+\beta)A^2+\frac{\frac{3}{2}\delta-3+m}{1+t}B^2\\
&-\Big(\frac{\delta-1}{2}(1+t)^{n-m}+\frac{2\delta}{\delta-1}(2\alpha+3\beta)(1+t)^{m-n-1}\Big)AB,\\
A^*=&(1+t)^{\gamma_3} |\varphi \nabla^{4}w|_2,\quad A=(1+t)^{\gamma_3} |\varphi \nabla^{3}\text{div}w|_2,\quad   B=(1+t)^{\delta_3} | \nabla^{3}\varphi|_2,
\end{split}
\end{equation}
then the following lemma holds:
\begin{lemma}\label{l2}
There exist some positive constants $b_m$,  such that
\begin{equation}\label{Leq:2.14}
\begin{split}
&\frac{1}{2}\frac{d}{dt} Z^2+\frac{b_m}{1+t} (Z^2-B^2)+H(A^*, A,B)\\
&+\alpha \sum_{k=0}^2 (1+t)^{2\gamma_k} |\varphi \nabla^{k+1}w(t)|^2_2+(\alpha+\beta) \sum_{k=0}^2 (1+t)^{2\gamma_k} |\varphi \nabla^{k}\text{div}w(t)|^2_2\\
 \leq&\eta (1+t)^{2\gamma_3} |\varphi \nabla^{4} w|^2_2+ C(\eta)(1+t)^{2m-5}Z^{4}+C_0(1+t)^{2m-n-3.5}Z^3\\
&+C_0(1+t)^{n-2.5}Z^{3}+C_0(\eta)\Big((1+t)^{2m-2n-4} +(1+t)^{-2}\Big)Z^2,
\end{split}
\end{equation}
with  the constant
 $b_m$   given by
 \begin{equation}\label{Jeq:2.6q}
b_m=
 \begin{cases}
\min\Big\{n-0.5, \frac{3}{2}\delta-3+m\Big\} \;\qquad\qquad \quad \   \ \quad   \text{if} \ \ \gamma\geq \frac{5}{3},\\[6pt]
\min\Big\{\frac{3\gamma}{2}-3+n, \frac{3}{2}\delta-3+m\Big\}  \qquad \   \ \ \ \quad   \quad  \text{if} \ \  1< \gamma<\frac{5}{3},
 \end{cases}
 \end{equation}
 where $\eta>0$ is   any suitably small constant.
 \end{lemma}
\begin{proof}

(\ref{Leq:2.14}) can be  obtained by multiplying \eqref{Aeq:2.13} and \eqref{Aeq:2.14qq} by $(1+t)^{2\gamma_k}$ and  $(1+t)^{2\delta_k}$ respectively  and summing the resulting inequalities together.

\end{proof}

\subsection{Proof of Theorem \ref{ths1} under the condition $(P_0)$}
Set
\begin{equation}\label{Agab}\begin{split}
M_1=&\frac{2\alpha+3\beta}{2\alpha+\beta},\quad
M_2=-3\delta+1+\frac{1}{2}M_3,\\
M_3=&\frac{(\delta-1)^2}{4(2\alpha+\beta)}+\frac{4\delta^2(2\alpha+\beta)}{(\delta-1)^2}M^2_1+2M_1\delta,\\
M_4=&\frac{1}{2}\min\Big\{ \frac{3\gamma-3}{2},\frac{-M_2-1}{2},1\Big\}+M_2,\\
\Phi(A,B)=&(2\alpha+\beta)A^2+\frac{\frac{3}{2}\delta-3+m}{1+t}B^2\\
&-\Big(\frac{\delta-1}{2}(1+t)^{n-m}+\frac{2\delta}{\delta-1}(2\alpha+3\beta)(1+t)^{m-n-1}\Big)AB.
\end{split}
\end{equation}

\begin{lemma}\label{ll2} Let  condition $(P_0)$ and  $M_1>0$ hold.  Then for $m=n+0.5=3$,
there exist some positive constants $\nu_*$, $b_*$ and $\epsilon^*$,  such that
\begin{equation}\label{Aeq:2.14}
\begin{split}
&\frac{1}{2}\frac{d}{dt} Z^2+\frac{(1-\nu_*)b_*}{1+t} Z^2
 \leq C(1+t)^{1+\epsilon^*}Z^{4}+C_0(1+t)^{-1-\epsilon^*}Z^2,
\end{split}
\end{equation}
where
 \begin{equation}\label{Aeq:2.6q}
\begin{split}
\epsilon^*=&\frac{1}{2}\min\Big\{ \frac{3\gamma-3}{2},\frac{-M_2-1}{2},1\Big\}>0,\\
 \nu_*=&\min\Big\{\frac{3\gamma-3}{4(3\gamma-1)}, \frac{-M_4-1}{6\delta-M_3}, \frac{1}{10}\Big\}>0,\\
b_*=&
 \begin{cases}
\min\Big\{2, \frac{3}{2}\delta-\frac{1}{4}M_3\Big\}>1 \;\qquad\qquad \quad \ \ \ \quad  \ \   \text{if} \ \ \gamma\geq \frac{5}{3},\\[6pt]
\min\Big\{\frac{3\gamma}{2}-0.5, \frac{3}{2}\delta-\frac{1}{4}M_3\Big\}>1  \qquad \ \  \  \ \ \  \quad  \text{if} \ \  1< \gamma<\frac{5}{3}.
 \end{cases}
\end{split}
 \end{equation}
Moreover, there exists a  constant $\Lambda(C_0)$ such that  $Z(t)$ is globally well-defined in $[0,+\infty)$ if $Z_0\leq \Lambda(C_0) $.
\end{lemma}

\begin{proof}\textbf{Step 1}.
First, note that 
\begin{equation}\label{GAB}
\begin{split}
\Phi(A,B)=& aA^2+bB^2-cAB
=a\big(A-\frac{c}{2a}B\big)^2+d B^2,
\end{split}
\end{equation}
with
\begin{equation*}\begin{split}
a=&(2\alpha+\beta),\quad  b=\frac{\frac{3}{2}\delta-3+m}{1+t},\\
c=&\Big(\frac{\delta-1}{2}(1+t)^{n-m}+\frac{2\delta}{\delta-1}(2\alpha+3\beta)(1+t)^{m-n-1}\Big),\\
d=&(1+t)^{-1}\Big(\frac{3}{2}\delta-3+m\Big)-\frac{1}{4}\Big(\frac{(\delta-1)^2}{4(2\alpha+\beta)}(1+t)^{2n-2m}\\
&+\frac{4\delta^2(2\alpha+\beta)}{(\delta-1)^2}M^2_1 (1+t)^{2m-2n-2}+2M_1\delta (1+t)^{-1}\Big).
\end{split}
\end{equation*}

Then according to Lemma \ref{changeA1} and  (\ref{Leq:2.14}),  for any $\nu\in (0,1)$, one has
\begin{equation}\label{Beq:2.14}
\begin{split}
&\frac{1}{2}\frac{d}{dt} Z^2+\frac{b_m}{1+t} (Z^2-B^2)+\nu \alpha (A^*)^2+(1-\nu)d B^2+\nu\Big(b-\frac{c^2}{4(\alpha+\beta)}\Big)B^2\\
&+\alpha \sum_{k=0}^2 (1+t)^{2\gamma_k} |\varphi \nabla^{k+1}w(t)|^2_2+(\alpha+\beta) \sum_{k=0}^2 (1+t)^{2\gamma_k} |\varphi \nabla^{k}\text{div}w(t)|^2_2\\
 \leq&\eta (A^*)^2+ C(\eta)((1+t)^{2m-5}+(1+t)^{2m-5-2\gamma_3})Z^{4}\\
 &+C_0\big((1+t)^{2m-n-3.5}+(1+t)^{n-2.5}\big)Z^{3}\\
 &+C_0(\eta)\Big((1+t)^{2m-2n-4} +(1+t)^{-2}\Big)Z^2.
\end{split}
\end{equation}

By Proposition \ref{ode1}, in order to obtain the uniform estimates on $Z$, one needs
\begin{equation}\label{Ceq:2.14}
\begin{split}
2m-2n-2=2n-2m=-1,
\end{split}
\end{equation}
which implies that 
\begin{equation}\label{Ceq:2.14}
\begin{split}
d=(1+t)^{-1}d^*(m)=(1+t)^{-1}\Big(\frac{3}{2}\delta-3+m-\frac{1}{4}M_3\Big).
\end{split}
\end{equation}
Choose $\epsilon^*=\frac{1}{2}\text{min}\{ \frac{3\gamma-3}{2},\frac{-M_2-1}{2},\frac{1}{10}\}>0$. Then it follows from (\ref{Beq:2.14}) that 
\begin{equation}\label{Deq:2.14}
\begin{split}
&\frac{1}{2}\frac{d}{dt} Z^2+\frac{b_m}{1+t} (Z^2-B^2)+\nu \alpha (A^*)^2+(1-\nu)\frac{d^*(m)}{1+t}B^2\\
&+\frac{\nu}{1+t}\Big(d^*-\frac{1}{4}M_3\frac{\alpha}{\alpha+\beta}\Big)B^2\\
 \leq&\eta (A^*)^2+ C(\eta)((1+t)^{2m-5+\epsilon^*}+(1+t)^{2m-5-2\gamma_3})Z^{4}+C_0(\eta)(1+t)^{-1-\epsilon^*} Z^2.
\end{split}
\end{equation}

\textbf{Step 2}.
On one  hand,  for fixed $\delta>1$,  a necessary condition to gurantee  that   the set
$$
\Pi=\{(\alpha,\beta)| \alpha>0,\ 2\alpha+3\beta>0,\ \text{and} \ M_2<-1\}
$$
is not empty is that $M_1<\frac{3}{2}-\frac{1}{\delta}$.


On the other hand,  for fixed $\alpha,\beta,\gamma,\delta$, there always exists a  sufficiently large number $m$ such that
$d^*(m)>0$. Indeed, one needs that 
$$
d^*(m)=\frac{3}{2}\delta-3+m-\frac{1}{4}M_3=-\frac{1}{2}M_2+m-\frac{5}{2}>0.
$$
Here for the choice $m=n+0.5=3$, one has therefore $d^*=\frac{3}{2}\delta-\frac{1}{4}M_3$.

\textbf{Step 3}.
Due to
$$
d^*=\frac{3}{2}\delta-\frac{1}{4}M_3>1 \quad \text{and} \quad M_4=\epsilon^*+M_2<-1,
$$
so for $\nu_*=\min\Big\{\frac{3\gamma-3}{4(3\gamma-1)}, \frac{-M_4-1}{6\delta-M_3}, \frac{1}{20}\Big\}$ and $\nu=\min\Big\{\frac{1}{200}, \frac{4d^*v_*(\alpha+\beta)}{M_3\alpha}\Big\}$, it hold that
\begin{equation}\label{keyinformation}\begin{split}
1+\epsilon^*-2(1-\nu_*)d^*=M_4+2\nu_* d^*<&-1,\\
 (1-\nu)d^*+\nu\Big(d^*-\frac{1}{4}M_3\frac{\alpha}{\alpha+\beta}\Big)\geq & (1-\nu_*)d^*.
\end{split}
\end{equation}
Based on this observation, we choose $\eta=\frac{\nu \alpha}{100}$, which, together with $(\ref{Leq:2.14})$, $m=n+0.5=3$ and  (\ref{GAB}), implies that
\begin{equation}\label{Feq:2.14}
\begin{split}
&\frac{1}{2}\frac{d}{dt} Z^2+\frac{(1-\nu_*)b_*}{1+t} Z^2+\frac{\nu \alpha}{2}(1+t)^{2\gamma_3} |\varphi \nabla^{4}w(t)|^2_2\\
&+\alpha \sum_{k=0}^2 (1+t)^{2\gamma_k} |\varphi \nabla^{k+1}w(t)|^2_2+(\alpha+\beta) \sum_{k=0}^2 (1+t)^{2\gamma_k} |\varphi \nabla^{k}\text{div}w(t)|^2_2\\
 \leq&C(1+t)^{1+\epsilon^*}Z^{4}+C_0(1+t)^{-1-\epsilon^*}Z^2.
\end{split}
\end{equation}
Therefore, 
\begin{equation}\label{Geq:2.14}
\begin{split}
&\frac{d}{dt} Z+\frac{(1-\nu_*)b_*}{1+t} Z
 \leq C(1+t)^{1+\epsilon^*}Z^{3}+C_0(1+t)^{-1-\epsilon^*}Z.
\end{split}
\end{equation}

According to Proposition \ref{ode1}, and
$$
\epsilon^*>0,\quad 1+\epsilon^*-2 (1-\nu_*)b_*<-1,
$$
then  $Z(t)$ satisfies
\begin{equation}
\displaystyle
Z(t)\leq \frac{(1+t)^{-(1-\nu_*)b_*}\exp{\Big(\frac{C_0}{-\epsilon^*}}\Big((1+t)^{-\epsilon^*}-1\Big)\Big)}{\Big(Z^{-2}_{0}-2C\int^t_0 (1+s)^{M}\exp{\Big(\frac{2C_0}{-\epsilon^*}}\Big((1+t)^{-\epsilon^*}-1\Big)\Big)\text{d}s\Big)^{\frac{1}{2}}},
\end{equation}
where $M=1+\epsilon^*-2 (1-\nu_*)b_*<-1$.

 Moreover,  $Z(t)$  is globally well-defined  for $t\geq 0$ if and  only if
\begin{equation}
0<Z_{0}<\frac{1}{\Big(2C\int^t_0 (1+s)^{M}\exp{\Big(\frac{2C_0}{-\epsilon^*}}\Big((1+t)^{-\epsilon^*}-1\Big)\Big)\text{d}s\Big)^{\frac{1}{2}}}.
\end{equation}
  Moreover,
 \begin{equation}\label{Heq:3.30}
  Z(t)\leq C_0(1+t)^{-(1-\nu_*)b_*} \qquad \text{ for all }\ t\geq 0,
 \end{equation}
which implies that
\begin{equation} \label{shangjieee}
Y_k(t)\leq C_0(1+t)^{-\gamma_k-(1-\nu_*)b_*},  \ \  \text{and} \ \ U_k(t)\leq C_0(1+t)^{-\delta_k-(1-\nu_*)b_*} \quad \text{ for }\ \forall \  t\geq 0.
\end{equation}

\end{proof}

Finally, it follows from (\ref{Feq:2.14}) and  (\ref{Heq:3.30}) that  for $k=0,1,2,3$:
\begin{equation}\label{Heq:2.14}
\begin{split}
&  \sum_{k=1}^3\int_0^t (1+t)^{2\gamma_k}\int   | \varphi\nabla^{k+1} w|^2 \text{d}s\leq C_0.
\end{split}
\end{equation}

\subsection{Proof of Theorem \ref{ths1} under the condition $(P_1)$}

For this case, we first choose $2n-2m=-(1+\epsilon)<-1$ for sufficiently small constant $0<\epsilon<\frac{1}{2}\min\big\{1,3\gamma-3\big\}$. Then it follows from Lemma \ref{l2} that
\begin{lemma}\label{Ul2}
There exist some positive constants $b_m$, $C$ and $C_0$  such that
\begin{equation}\label{ULeq:2.14}
\begin{split}
&\frac{1}{2}\frac{d}{dt} Z^2+\frac{b_m}{1+t} Z^2+\frac{\alpha}{2} \sum_{k=0}^3 (1+t)^{2\gamma_k} |\varphi \nabla^{k+1}w(t)|^2_2\\
 \leq& C(1+t)^{2n-4+\epsilon}Z^{4}+C_0(1+t)^{-1-\epsilon}Z^2,
\end{split}
\end{equation}
with  the constant
 $b_m$   given by (\ref{Jeq:2.6q}).
\end{lemma}
The proof is routine and  thus omitted. The rest of the proof for  Theorem \ref{ths1} under the condition $(P_1)$ is similar to the one in Subsection 5.3 for the condition $(P_0)$.

\section{Global-in-time well-posedness with compactly supported density under $(P_2)$ or $(P_3)$}

 Denote
$
\iota=(\delta-1)/(\gamma-1)
$. It is always assumed that   $\iota \geq 2$ or $=1$. Based on the analysis of the previous section, it remains to estimate the following two terms:
\begin{equation*}
\begin{split}
L_k(W,\widehat{u})=&-\int \Big(\nabla \phi^{2\iota}\cdot \mathbb{{S}}(\nabla^kw)-\big(\nabla^k(\phi^{2\iota}Lw)-\phi^{2\iota}L\nabla^kw\big)\Big)\cdot \nabla^k w\\
&+\int\nabla^k(\phi^{2\iota}L  \widehat{u})\cdot \nabla^k w \equiv: L^1_k+L^2_k+L^3_k,\\[6pt]
Q_k(W, \widehat{u})=&\int \Big(\nabla \phi^{2\iota} \cdot Q(\nabla^kw)+\big(\nabla^k(\nabla \phi^{2\iota}\cdot Q(w))-\nabla \phi^{2\iota}\cdot Q(\nabla^kw)\big)\Big)\cdot \nabla^k w\\
&+\int\nabla^k\big(\nabla \phi^{2\iota} \cdot Q( \widehat{u})\big)\cdot \nabla^k w
\equiv: Q^1_k+Q^2_k+Q^3_k.
\end{split}
\end{equation*}

\subsection{Estimates on  $L_k$ and $Q_k$}  \begin{lemma}\label{DI2}
For any suitably small constant $\eta>0$, there are two constants  $C(\eta)$  and  $C_0(\eta)$  such that
\begin{equation}\begin{split}
 L_k(W) \leq&  \eta |\phi^\iota  \nabla^{k+1} w|^2_2\delta_{3,k}+C(\eta)(1+t)^{2\iota n-3\iota-2-\gamma_k}Y^{2\iota+1}Y_k\\
&+C_0(1+t)^{2\iota n-3\iota-n-1.5-\gamma_k}Y^{2\iota}Y_k+C_0(\eta)(1+t)^{2\iota n-3\iota-7}Y^{2\iota},\\
Q_k(W, \bar{u}) \leq& \eta |\phi^\iota  \nabla^{k+1} w|^2_2\delta_{3,k}+C(\eta)(1+t)^{2\iota n-3\iota-2-\gamma_k}Y^{2\iota+1}Y_k\\
&+C_0(1+t)^{2\iota n-3\iota-n-0.5-\gamma_k}Y^{2\iota}Y_k+C_0(\eta)(1+t)^{2\iota n-3\iota-5}Y^{2\iota}.
\end{split}
\end{equation}
\end{lemma}

\begin{proof} Now  we  give only the corresponding estimates on $Q_k$. The estimates on $L_k$ can be obtained  similarly.

\textbf{Step 1:}
Estimates on  $Q^k_1$.  Direct estimates yield that 
\begin{equation}\label{lk1-2}
\begin{split}
Q^1_k \leq& C |\phi|^{2\iota-1}_\infty |\nabla \phi|^{}_\infty |\nabla^{k+1} w|_2|\nabla^{k} w|_2
\leq  C(1+t)^{2\iota n-3\iota-2-\gamma_k}Y^{2\iota+1}Y_k,\quad \text{for} \quad k\leq 2,\\
Q^1_3\leq & C |\phi^\iota  \nabla^{4} w|_2|\phi|^{\iota-1}_\infty|\nabla \phi|_{\infty}Y_3
\leq  \eta |\phi^\iota  \nabla^{4} w|^2_2+C(\eta)(1+t)^{2\iota n-3\iota-2-\gamma_3}Y^{2\iota+1}Y_3.
\end{split}
\end{equation}

\textbf{Step 2:}  Estimates on  $Q^k_3$. If $k=0$ or $1$, one has
\begin{equation}\label{lk1qs}
\begin{split}
Q^3_0\leq & C_0|\phi|^{2\iota-2}_\infty|\nabla \phi|_\infty|\nabla \widehat{u}|_\infty|\phi|_2|w|_2
\leq C(1+t)^{2\iota n-3\iota-n-0.5-\gamma_0}Y^{2\iota}Y_0,\\
Q^3_1 \leq & C|\phi|^{2\iota-1}_\infty |\nabla\phi|_2 |\nabla^2 \widehat{u}|_\infty|\nabla w|_2
\leq C_0(1+t)^{2\iota n-3\iota-n-0.5-\gamma_1}Y^{2\iota}Y_1.
\end{split}
\end{equation}

For $k=2$, decompose
$
Q^3_2\triangleq Q^3_2(L)+Q^3_2(2,0)
$,
which can be estimated as
\begin{equation}\label{lk5q}
\begin{split}
Q^3_2(L)\triangleq &C\int (\nabla \phi^{2\iota}  \cdot \nabla^3 \widehat{u}+\nabla^2\phi^{2\iota} \cdot  \nabla^2 \widehat{u})\cdot \nabla^2 w\\
\leq & C\big(| \phi|^{}_{\infty}|\nabla \phi|_\infty |\nabla^3 \widehat{u}|_2+ |\nabla \phi|^2_\infty|\nabla^2\widehat{u}|_2+| \phi|^{}_{\infty} |\nabla^2 \phi|_2|\nabla^2\widehat{u}|_\infty \big) | \phi|^{2\iota-2}_{\infty}|\nabla^2 w|_2\\
\leq  &    C_0(1+t)^{2\iota n-3\iota-n-1.5-\gamma_2}Y^{2\iota}  Y_2,
\end{split}
\end{equation}
and 
\begin{equation}\label{lk5qsd}
\begin{split}
Q^3_2(2,0)\triangleq &C\int  \nabla^3\phi^{2\iota}  \cdot  \nabla \widehat{u}\cdot \nabla^2 w
\leq   C\big(|\nabla \phi|_\infty  |\nabla \phi|^2_6  |\nabla \widehat{u}|_6\\
&+| \phi|^{}_{\infty} |\nabla^2 \phi|_6|\nabla \phi|_6|\nabla \widehat{u}|_6 
+| \phi|^{2}_{\infty} |\nabla^3 \phi|_2|\nabla\widehat{u}|_\infty \big)| \phi|^{2\iota-3}_{\infty}|\nabla^2 w|_2\\
\leq & C_0(1+t)^{2\iota n-3\iota-n-0.5-\gamma_2}Y^{2\iota}  Y_2.
\end{split}
\end{equation}

For $k=3$, set
$
Q^3_3\triangleq Q^3_3(L)+Q^3_3(3,0)
$, where 
\begin{equation}\label{lk4q}
\begin{split}
Q^3_3(L)=& C\int \big(\nabla \phi^{2\iota} \cdot   \nabla^4 \widehat{u}+\nabla^{2}\phi^{2\iota}  \cdot  \nabla^3 \widehat{u}+\nabla^3\phi^{2\iota} \cdot   \nabla^2 \widehat{u}\big)\cdot \nabla^3w \\
\leq &  C\Big(|\phi|^{2}_\infty| \nabla\phi|_{\infty}|\nabla^4 \widehat{u}|_2+|\phi|^{}_\infty \nabla\phi|^2_{\infty}|\nabla^3 \widehat{u}|_2+|\phi|^{2}_\infty|\nabla^{2}\phi|_{6}|\nabla^3 \widehat{u}|_3\\
&+\big( |\nabla \phi|^2_\infty|\nabla \phi|_2+| \phi|_{\infty} |\nabla^2 \phi|_2|\nabla \phi|_\infty+| \phi|^{2}_{\infty} |\nabla^3 \phi|_2\big)|\nabla^2\widehat{u}|_\infty\Big)| \phi|^{2\iota-3}_{\infty} |\nabla^3 w|_2\\
\leq& C_0(1+t)^{2\iota n-3\iota-n-1.5-\gamma_3}Y^{2\iota}Y_3,\\
Q^3_3(3,0)= &C\int  \nabla^{4}\phi^{2\iota} \cdot  \nabla \widehat{u}\cdot \nabla^3 w
= C\int  \nabla^{3}\phi^{2\iota}  \cdot  \big(\nabla^2 \widehat{u}\cdot \nabla^3 w+\nabla \widehat{u}\cdot \nabla^{4} w)\\
\leq &   C_0(1+t)^{2\iota n-3\iota-n-1.5-\gamma_3}Y^{2\iota}  Y_3+Q^3_3(A),
\end{split}
\end{equation}
with
\begin{equation}\label{lk8eeq}
\begin{split}
 Q^3_3(A)\triangleq &C \int  \nabla^{3}\phi^{2\iota} \cdot \nabla \widehat{u}\cdot \nabla^{4} w\\
\leq &C\big( |\nabla^3 \phi|_2 |\phi|^{\iota-1}_\infty+|\nabla \phi|_3 |\nabla^2 \phi|_6 |\phi|^{\iota-2}_\infty\big)|\nabla \widehat{u}|_\infty |\phi^\iota \nabla^4w|_2+Q^3_3(B)\\
\leq & \eta |\phi^\iota \nabla^{4} w|^2_2+C_0(\eta)(1+t)^{2\iota n-3\iota-5}Y^{2\iota}+Q^3_3(B),
\end{split}
\end{equation}
where  the term $Q^3_3(B)$ can be estimated  by  using  integration by parts again, 
\begin{equation}\label{lk9eeqs}
\begin{split}
 Q^3_3(B)
\triangleq &C \int \phi^{2\iota-3} \nabla \phi\cdot \nabla \phi\cdot \nabla\phi \cdot \nabla \widehat{u}\cdot \nabla^{4} w\\
\leq &C\Big(|\nabla \phi|^3_\infty |\phi|^{2\iota-4}_\infty|\nabla \phi|_2+|\nabla^2 \phi|_6 |\nabla \phi|^2_6 |\phi|^{2\iota-3}_\infty\Big) |\nabla \widehat{u}|_\infty|\nabla^{3}w|_2\\
&+C(1+t)^{2\iota n-3\iota-n-1.5-\gamma_3}Y^{2\iota}  Y_3
\leq   C_0(1+t)^{2\iota n-3\iota-n-0.5-\gamma_3}Y^{2\iota}  Y_3.
\end{split}
\end{equation}

Then  (\ref{lk1qs})-(\ref{lk9eeqs}) give that 
\begin{equation}\begin{split}\label{zongjie3}
Q^3_k\leq& \eta |\phi^\iota \nabla^{k+1} w|^2_2+C_0(1+t)^{2\iota n-3\iota-n-0.5-\gamma_k}Y^{2\iota}  Y_k+C_0(\eta)(1+t)^{2\iota n-3\iota-5}Y^{2\iota}.
\end{split}
\end{equation}

\textbf{Step 3:}
Estimates on  $Q^2_k$. For $k=1$, one has
\begin{equation}\label{lk15e}
\begin{split}
Q^2_1
=&C\int \big(\phi^{2\iota-1}\nabla^2 \phi+\phi^{2\iota-2}\nabla \phi \cdot \nabla \phi \big)\cdot  \nabla w\cdot \nabla w\\
\leq & C|\phi|^{2\iota-2}_\infty\big(|\phi|^{}_\infty|\nabla^2 \phi|_6 | \nabla w|_3+|\nabla \phi|^2_\infty |\nabla w|_2\big)|\nabla w|_2\\
\leq & C(1+t)^{2\iota n-3\iota-2-\gamma_1} Y^{2\iota+1}Y_1.
\end{split}
\end{equation}
For $k=2$, one has
\begin{equation}\label{lk15ee}
\begin{split}
Q^2_2
=&C\int \big(\nabla^3 \phi^{2\iota}\cdot \nabla w+\nabla^2 \phi^{2\iota}\cdot \nabla^2 w\big)\cdot \nabla^2 w\\
\leq & C\big(|\phi|^{}_\infty|\nabla \phi|_\infty|\nabla w|_\infty|\nabla^2 \phi|_2+|\phi|^{2}_\infty|\nabla w|_\infty |\nabla^3 \phi|_2 \\
&+|\nabla \phi|^2_\infty|\nabla \phi|_3|\nabla w|_6+|\phi|_\infty |\nabla\phi|^2_\infty|\nabla^2 w|_2\\
&+|\nabla^2 w|_6|\nabla^{2}\phi|_3|\phi|^{2}_\infty\big)|\phi|^{2\iota-3}_\infty|\nabla^{2} w|_2
\leq  C(1+t)^{2\iota n-3\iota-2-\gamma_2}Y^{2\iota+1}Y_2.
\end{split}
\end{equation}

For $k=3$,  using  integration by parts, one can get 
\begin{equation}\label{lk15eee}
\begin{split}
Q^2_3
=&C\int \big(\nabla^4 \phi^{2\iota}\cdot \nabla w+\nabla^3 \phi^{2\iota}\cdot \nabla^2 w+\nabla^2 \phi^{2\iota}\cdot \nabla^3 w\big)\cdot \nabla^3 w\\
\leq & C|\phi|^{2\iota-3}_\infty |\nabla\phi|_6\big( |\nabla\phi|_\infty|\nabla\phi|_6+|\phi|^{}_\infty|\nabla^2 \phi|_6\big)|\nabla^2 w|_6|\nabla^3 w|_2\\
&+C|\phi^\iota\nabla^3 w|_6|\nabla^{2}w|_3|\phi|^{\iota-1}_\infty|\nabla^{3} \phi|_2\\
 &+ C|\phi|^{\iota-1}_\infty(|\phi|^{\iota-1}_\infty| \nabla\phi|^2_{\infty}|\nabla^3 w|_2+| \nabla^2\phi|_{3}|\phi^{\iota}\nabla^3 w|_6)|\nabla^3 w|_2\\
\leq & \eta |\phi^\iota  \nabla^{4} w|^2_2+C(\eta)(1+t)^{2\iota n-3\iota-2-\gamma_3}Y^{2\iota+1}Y_3+Q^2_3(A),
\end{split}
\end{equation}
with
\begin{equation}\label{lk45}
\begin{split}
Q^2_3(A)=&C \int  \nabla^{3}\phi^{2\iota} \cdot  \nabla w\cdot \nabla^{4} w\\
\leq &C \big(|\nabla^3 \phi|_2 |\phi|_\infty+|\nabla \phi|_3 |\nabla^2 \phi|_6\big)  |\phi|^{\iota-2}_\infty|\nabla w|_\infty |\phi^\iota \nabla^4w|_2+Q^2_3(B)\\
\leq & C(\eta)(1+t)^{2\iota n-3\iota-2-\gamma_3} Y^{2\iota+1}Y_3+\eta |\phi^{\iota} \nabla^{4} w|^2_2+Q^2_3(B),
\end{split}
\end{equation}
where  the term $Q^2_3(B)$ can be estimated by using  integration by parts again, 
\begin{equation}\label{lk45ee}
\begin{split}
 Q^2_3(B)
=&C \int \phi^{2\iota-3} \nabla \phi \cdot\nabla \phi \cdot \nabla \phi \cdot   \nabla w\cdot \nabla^{4} w\\
\leq &C\big(|\nabla \phi|^2_\infty |\nabla \phi|_2+|\nabla^2 \phi|_6 |\nabla \phi|_3 |\phi|^{}_\infty\big)|\phi|^{2\iota-4}_\infty|\nabla \phi|_\infty |\nabla w|_\infty|\nabla^{3}w|_2\\
&+C|\phi|^{2\iota-3}_\infty |\nabla\phi|^3_\infty|\nabla^2 w|_6|\nabla^3 w|_2\\
\leq &  C(1+t)^{2\iota n-3\iota-2-\gamma_3} Y^{2\iota+1}Y_3.
\end{split}
\end{equation}
Then the  estimates (\ref{lk15e})-(\ref{lk45}) lead to 
\begin{equation}\begin{split}\label{zongjie4}
Q^2_k\leq& \eta |\phi^\iota \nabla^{k+1} w|^2_2+C(\eta)(1+t)^{2\iota n-3\iota-2-\gamma_k}Y^{2\iota+1}  Y_k.
\end{split}
\end{equation}

\end{proof}

\subsection{Derivation of a global-in-time  well-posedness for an ordinary differential inequality} Then, based on the above  energy estimates, one can derive the following lemma.
\begin{lemma}\label{Dl2}
There exist some positive constants $a=r+n$,  $C_0$ and $C$ such that 
\begin{equation}\label{eq:2.12}
\frac{\text{d} Y(t)}{\text{d} t}+\frac{a }{1+t}Y(t)\leq  C_0(1+t)^{-1-\eta}Y(t)+C(1+t)^{2\iota n-3\iota-1+2\theta \eta-\eta} Y^{2\iota+1},
\end{equation}
holds for any sufficiently small constant $0<\eta<\frac{3\iota(\gamma-1)}{2\iota-1}$. Moreover, there exists a  constant $\Lambda(C_0)$ such that  $Y(t)$ is globally well-defined in $[0,+\infty)$ if $Y_0\leq \Lambda(C_0) $.

\end{lemma}
\begin{proof}
First, it follows from  Lemmas \ref{l1},  \ref{DI2} and \eqref{eq:2.3}  that
\begin{equation}\label{eq:2.13}
\begin{split}
&\frac{1}{2}\frac{d}{dt} Y^2_k+\frac{k+r}{1+t} Y^2_k+\frac{1}{2} \alpha\int   | \phi^\iota\nabla^{k+1} w|^2\\
 \leq& CYY_k  (1+t)^{-\gamma_k-2}+C(1+t)^{n-2.5}YY^2_k+C(1+t)^{2\iota n-3\iota-2-\gamma_k}Y^{2\iota+1}Y_k\\
&+C_0(1+t)^{2\iota n-3\iota-n-0.5-\gamma_k}Y^{2\iota}Y_k+C_0(1+t)^{2\iota n-3\iota-5}Y^{2\iota}.
\end{split}
\end{equation}

Second, multiplying \eqref{eq:2.13} by $(1+t)^{2\gamma_k}$  and simplifying by $Y$ give
\begin{equation}\label{eq:2.14}
\begin{split}
\frac{\text{d} Y(t)}{\text{d} t}+\frac{a }{1+t}Y(t)\leq &  C_0(1+t)^{-2}Y+C(1+t)^{n-2.5}Y^2+C(1+t)^{2\iota n-3\iota-2}Y^{2\iota+1}\\
&+C_0(1+t)^{2\iota n-3\iota-n-0.5}Y^{2\iota}+C_0(\eta)(1+t)^{2\iota n-3\iota-2n+1}Y^{2\iota-1}\\
\leq & C_0(1+t)^{-1-\eta}Y(t)+C(1+t)^{2\iota n-3\iota-1+2\theta \eta-\eta} Y^{2\iota+1}
\end{split}
\end{equation}
for any sufficiently small constant $0< \eta <1$. A sufficient condition for the  global existence of solutions to (\ref{eq:2.14}) is 
$$
 K_{a,\iota,\eta}=2\iota n-3\iota-2+1+2\iota \eta-\eta-2\iota a<-1.
$$
Indeed, it follows from the definition of $a$ and $\iota\geq 2$ or $\iota=1$ that 
 \begin{equation}\label{eq:zuihou}
 K_{a,\iota,\eta}=
 \begin{cases}
-2\iota-1-\eta+2\iota \eta<-1 \;\qquad\qquad  \quad \  \text{if} \ \ \gamma\geq \frac{5}{3};\\[6pt]
3\iota(1-\gamma)+\eta (2\iota-1)< -1 \; \qquad\qquad  \text{if} \ \ 1<\gamma<\frac{5}{3},
\end{cases}
 \end{equation}
for the sufficiently small constant $0<\eta<\frac{3\iota(\gamma-1)}{2\iota-1}$.

Then,  (\ref{eq:2.14})-(\ref{eq:zuihou}), Proposition \ref{ode1} and Lemma \ref{l1} give
\begin{equation}
\displaystyle
Y(t)\leq \frac{(1+t)^{-a}\exp{\Big(\frac{C_0}{\eta}}(1-(1+t)^{-\eta})\Big)}{\Big(Z^{-2\iota}_{0}-2\iota C\int^t_0 (1+s)^{K_{a,\iota,\eta}}\exp{\Big(\frac{2\iota C_0}{\eta}}\big(1-(1+s)^{-\eta}\big)\Big)\text{d}s\Big)^{\frac{1}{2\iota}}}.
\end{equation}
Thus, $Y(t)$  is defined  for $t\geq 0$ if and only if
\begin{equation}\label{eq:3.29}
0<Y_{0}<\frac{1}{\Big(2\iota C\int^t_0 (1+s)^{K_{a,\iota,\eta}}\exp{\Big(\frac{2\iota C_0}{\eta}}\big(1-(1+s)^{-\eta}\big)\Big)\text{d}s\Big)^{\frac{1}{2\iota}}}.
\end{equation}
 Moreover, it holds that 
 \begin{equation}\label{eq:3.30}
Y(t)\leq C_0(1+t)^{-r-n},\quad \text{and} \quad  Y_k(t)\leq K(1+t)^{-k-r} \qquad \text{ for all }\ t\geq 0.
 \end{equation}
\end{proof}

Now we are ready to prove Theorem \ref{ths1}. \textbf{Step 1:} Estimates on $\varphi\nabla^4 w$.
Note that (\ref{eq:2.13}) and (\ref{eq:3.30}) for $k=3$ imply that 
\begin{equation}\label{eq:2.13dd}
\begin{split}
&\frac{1}{2}\frac{d}{dt} Y^2_3+\frac{3+r}{1+t} Y^2_3+\frac{1}{2} \alpha\int   | \phi^\iota\nabla^{4} w|^2\leq C(1+t)^{-8.5-3r}.
\end{split}
\end{equation}
Multiplying \eqref{eq:2.13dd} by $(1+t)^{2e}$ on both sides gives
\begin{equation}\label{eq:2.13ee}
\begin{split}
&\frac{1}{2}\frac{d}{dt} ((1+t)^{2e}Y^2_3)+\frac{3+r-e}{1+t} ((1+t)^{2e}Y^2_3)+\frac{1}{2} \alpha (1+t)^{2e}\int   | \phi^\iota\nabla^{4} w|^2\\
\leq& C(1+t)^{-8.5-3r+2e}.
\end{split}
\end{equation}
Under the assumptions that
$$
3+r-e\geq 0,\quad -8.5-3r+2e<-1,
$$
i. e., $e\leq 3+r$,
one can obtain
\begin{equation}\label{eq:2.13ff}
\begin{split}
(1+t)^{2e}Y^2_3+\frac{1}{2} \alpha \int_0^t (1+t)^{2e}\int   | \phi^\iota\nabla^{4} w|^2\text{d}s\leq C_0.
\end{split}
\end{equation}

\textbf{Step 2:} Estimates on $\varphi=\phi^\iota$.
Here we always assume that $\iota\geq 2$ in the rest of this proof.
Due to   (\ref{eq:3.30}), one can get
\begin{equation}\label{eqai}
\begin{split}
|\nabla^k \varphi|_2\leq C(1+t)^{-1.5\iota-r\iota+1.5-k}.
\end{split}
\end{equation}

Applying $\nabla^3$ to $\eqref{li47-2}_1$,  multiplying by $\nabla^3 \varphi$ and integrating over $\mathbb{R}^3$, then similar to Lemma \ref{Azongjiephi}, one can get
\begin{equation}\label{guji23}
\begin{split}
&\frac{1}{2}\frac{d}{dt}|\nabla^3 \varphi |^2_2+\frac{3\delta}{2(1+t)}|\nabla^3 \varphi |^2_2\\
\leq&  C_0(1+t)^{-r-2.5}|\nabla^3 \varphi |^2_2+C(1+t)^{-0.75\iota-0.5 r\iota-r-3.25}|\nabla^3\varphi(t)|^{3/2}_2\\
&+C (1+t)^{-1.5\iota-r\iota-3.5}|\nabla^3 \varphi|_2+C|\varphi \nabla^4w|_2|\nabla^3\varphi(t)|_2.
\end{split}
\end{equation}

Then, multiplying \eqref{guji23} by $(1+t)^{2p}$ and simplifying by
$G^2=(1+t)^{2p}|\nabla^3 \varphi |^2_2$ lead to 
\begin{equation}\label{eq:2.14qq}
\begin{split}
&\frac{1}{2}\frac{\text{d} G^2}{\text{d} t}+\frac{\frac{3}{2}\delta-p}{1+t}G^2\\
\leq&  C_0\big((1+t)^{-r-2.5}+(1+t)^{2p-2e}\big)G^2+C(1+t)^{-0.75\iota-0.5 r\iota-r-3.25+0.5p}G^{3/2}\\
&+C (1+t)^{-1.5\iota-r\iota-3.5+p}G+C(1+t)^{2e}|\varphi \nabla^4w|^2_2.
\end{split}
\end{equation}
Choose $p$ satisfying the following conditions:
\begin{equation*}\begin{split}
 \frac{3}{2}\delta-p>0,\quad -0.75 \iota-0.5 r\iota-r-3.25+0.5p<&-1,\\
\quad 2p-2e<-1,\quad -1.5\iota-r\iota-3.5+p<&-1.
\end{split}
\end{equation*}
Then (\ref{eq:2.14qq}) yields 
$$
(1+t)^{2p}|\nabla^3 \varphi |^2_2\leq C_0.
$$

\section{Global-in-time well-posedness without compactly supported assumption}

In this section,  we will extend the global-in-time well-posedness  in Theorem \ref{ths1} to the   more general case  whose   initial mass density $\rho_0$ is still small, but is not necessary to be   compactly  supported, which  can be stated as:
\begin{theorem}\label{ths1B} Let (\ref{canshu}) and any one of the conditions $(P_0)$-$(P_3)$ hold.
 If  initial data $( \varphi_0, \phi_0,u_0)$  satisfies $(A_1)$-$(A_2)$, then for any positive time $T>0$,  there exists a unique global classical solution $(\varphi,  \phi, u)$ in $[0,T] \times \mathbb{R}^3$  to the Cauchy problem (\ref{li47-1}) satisfying
\begin{equation}\label{AregghNN}\begin{split}
& (\varphi,\phi, w) \in C([0,T]; H^{s'}_{loc})\cap L^\infty([0,T]; H^3),  \  \varphi\nabla^4w \in L^2([0,T]; L^2),
\end{split}
\end{equation}
for  any constant $s'\in[2,3)$.   Moreover, when  $(P_2)$ holds, the smallness assumption on $\varphi_0$ could be removed.
\end{theorem}

In the following subsections, we will prove Theorem \ref{ths1B} under the  condition $(P_0)$. The proof for other cases is similar, and so is omitted.

\subsection{Existence}

According to Theorem \ref{ths1},  for the initial data
$$
(\varphi^R_0,\phi^R_0, w_0)=(\varphi_0F(|x|/R), \phi_0F(|x|/R), 0),
$$
there exists the unique  global regular solution  $(\varphi^R, \phi^R,w^R)$ satisfying:
\begin{equation}\label{decay2}
\begin{split}
|\nabla^k \phi^R(t)|_2+|\nabla^kw^R(t)|_2\leq & C_0(1+t)^{-(1-\nu_*)b_*+2.5-k};\\
 \|(1+t)^{k-2.5}\varphi^R\nabla^{k+1} w^R\|_{L^2L^2_t}\leq  C_0,\
  |\nabla^k \varphi^R(t)|_2\leq & C_0(1+t)^{-(1-\nu_*)b_*+3-k},
\end{split}
\end{equation}
for any positive time $t>0$, with    $C_0$  independent of $R$.

Due to  (\ref{decay2}) and the  following relations:
\begin{equation}\label{shijianfangxiang}
\begin{cases}
\displaystyle
\varphi^R_t=-(w^R+ \widehat{u} )\cdot \nabla\varphi^R-\frac{\delta-1}{2}\varphi\text{div} (w^R+ \widehat{u} ),\\[8pt]
\displaystyle
\phi^R_t=-(w^R+ \widehat{u} )\cdot \nabla \phi^R-\frac{\gamma-1}{2}\phi^R\text{div} (w^R+ \widehat{u} ),\\[8pt]
\displaystyle
w^R_t=-w^R\cdot \nabla w^R-\frac{\gamma-1}{2}\phi^R\nabla \phi^R-(\varphi^R)^2 Lw^R\\[8pt]
\quad +\nabla (\varphi^R)^2\cdot Q(w^R+ \widehat{u} )- \widehat{u} \cdot \nabla w^R-w^R\cdot \nabla  \widehat{u} -(\varphi^R)^2 L\widehat{u},
 \end{cases}
\end{equation}
 it holds that  for any finite constant  $R_0>0$ and finite time $T>0$, 
\begin{equation}\label{uniformshijianN}
\begin{split}
\|\varphi^R_t\|_{H^2(B_{R_0})}+\|\phi^R_t\|_{H^2(B_{R_0})}+\|w^R_t\|_{H^1(B_{R_0})}+\int_0^t \|\nabla^2 w^R_t\|^2_{L^2(B_{R_0})}\text{d}s\leq C_0(R_0,T),
\end{split}
\end{equation}
for $0\leq t \leq T$, where the constant $C_0(R_0,T)>0$ depending  on $C_0$, $R_0$ and $T$.

Since  (\ref{decay2}) and (\ref{uniformshijianN}) are  independent of  $R$,   there exists a subsequence of solutions (still denoted by) $(\varphi^R, \phi^R, w^R)$   converging to  a limit  $(\varphi,\phi,w)$  in the  sense:
\begin{equation}\label{strongjixian}
\begin{split}
&(\varphi^R, \phi^R,w^R)\rightarrow   (\varphi, \phi,w) \quad \text{strongly \ in } \ C([0,T];H^2(B_{R_0})).
\end{split}
\end{equation}

For $k=0,1,2,3$, denote
$$
a_k=-(1-\nu_*)b_*+2.5-k,\quad b_k=-(1-\nu_*)b_*+3-k,\quad c_k=k-2.5.
$$
Again due to  (\ref{decay2}), there exists a subsequence of solutions (still denoted by) $(\varphi^R, \phi^R, w^R)$ converging  to the same limit $(\varphi,\phi,w)$ as above in the following   weak* sense (for $k=0,1,2,3$):
\begin{equation}\label{ruojixian}
\begin{split}
(1+t)^{b_k}\varphi^R \rightharpoonup  (1+t)^{b_k}\varphi \quad \text{weakly* \ in } \ & L^\infty([0,T];H^3(\mathbb{R}^3)),\\
(1+t)^{a_k}\big( \phi^R, w^R\big) \rightharpoonup  (1+t)^{a_k}\big(\phi, w) \quad \text{weakly* \ in } \ & L^\infty([0,T];H^3(\mathbb{R}^3)).
\end{split}
\end{equation}

Combining  the strong convergence in (\ref{strongjixian}) and the weak convergence in (\ref{ruojixian}) shows that  $(\varphi, \phi,w) $ also satisfies the corresponding  estimates (\ref{decay2}) and (for $k=0,1,2,3$):
\begin{equation}\label{ruojixian1}
\begin{split}
(1+t)^{c_k}\varphi^R \nabla^{k+1} w^R \rightharpoonup (1+t)^{c_k} \varphi\nabla^{k+1} w \quad &\text{weakly \ in } \ L^2([0,T]\times \mathbb{R}^3).
\end{split}
\end{equation}


It is then obvious that $(\varphi, W) $ is a weak solution to problem (\ref{li47-1}) in the sense of distribution.

\subsection{Uniqueness}

Let $(\varphi_1,W_1)$ and $(\varphi_2,W_2)$ be two  solutions  to  (\ref{li47-1}) satisfying the uniform a priori estimates (\ref{decay2}). Set
$F_N=F(|x|/N)$, and 
\begin{equation*}\begin{split}
&\overline{\varphi}=\varphi_1-\varphi_2,\quad \overline{W}=(\overline{\phi},\overline{w})=(\phi_1-\phi_2,w_1-w_2),\\
&\overline{\varphi}^N=\overline{\varphi} F_N,\ \  \ \  \overline{W}^N=\overline{W}F_N=(\overline{\phi}^N,\overline{w}^N),
\end{split}
\end{equation*}
then  $(\overline{\varphi}^N,\overline{W}^N)$ solves the  following problem
 \begin{equation}
\label{weiyixing}
\begin{cases}
\displaystyle
\ \  \ \overline{\varphi}^N_t+(w_1+\widehat{u})\cdot \nabla\overline{\varphi}^N +\overline{w}^N\cdot\nabla\varphi_{2}+\frac{\delta-1}{2}(\overline{\varphi}^N \text{div}(w_2+\widehat{u})+\varphi_{1}\text{div}\overline{w}^N)\\[6pt]
\displaystyle
 =\overline{\varphi}(w_1+\widehat{u})\cdot  \nabla F_N +\frac{\delta-1}{2}\varphi_{1}\overline{w}\cdot \nabla F_N, \\[6pt]
\ \  \ \overline{W}^N_t+\sum\limits_{j=1}^3A^*_j(W_1,\widehat{u})\partial_{j} \overline{W}^N+\varphi^2_1\mathbb{L}(\overline{w}^N)\\[6pt]
=-\sum\limits_{j=1}^3A_j(\overline{W}^N)\partial_{j} W_{2}  -\overline{\varphi}^N(\varphi_1+\varphi_2)\mathbb{L}(w_2) +F_N\big(\mathbb{H}(\varphi_1)-\mathbb{H}(\varphi_2)\big)\cdot \mathbb{Q}(W_2)\\[6pt]
\ \ \ \ +\mathbb{H} (\varphi_1)\cdot \mathbb{Q}(\overline{W}^N)-B(\nabla \widehat{u},\overline{w}^N)-D(\overline{\varphi}^N(\varphi_1+\varphi_2),\nabla^2 \widehat{u})\\
\ \ \ \ +\sum\limits_{j=1}^3A^*_j(W_1,\widehat{u}) \overline{W} \partial_{j} F_N +\varphi^2_1(\mathbb{L}(\overline{w}^N)-F_N\mathbb{L}(\overline{w})) -\mathbb{H} (\varphi_1)\cdot \mathbb{Q}(\overline{W})\cdot \nabla F_N,\\[6pt]
 \ \ \ \ (\overline{\varphi}^N,\overline{W}^N)|_{t=0}=(0, 0),\quad x\in \mathbb{R}^3,\\[6pt]
\ \ \ \ (\overline{\varphi}^N,\overline{W}^N)\rightarrow (0,0) \quad \text{as } \quad |x|\rightarrow +\infty,\quad t> 0.
\end{cases}
\end{equation}

It is not hard to obtain that
\begin{equation}\label{weiyi1}
\begin{split}
&\frac{d}{dt} |\overline{\varphi}^N|^2_2\leq  C(|\nabla w_1|_\infty+|\nabla w_2|_\infty+|\nabla \widehat{u}|_\infty)|\overline{\varphi}^N|^2_2+C|\nabla \varphi_2|_\infty|\overline{w}^N|_2|\overline{\varphi}^N|_2\\
&\qquad \qquad +C|\varphi_{1}\text{div}\overline{w}^N|_2|\overline{\varphi}^N|_2+I^1_N,\\
&\frac{d}{dt}|\overline{W}^N|^2_2+\frac{1}{2}\alpha\int  \varphi^2_1 |\nabla \overline{w}^N |^2
\leq  C\Big(1+\|W_1\|^2_{3}+\| W_2\|^2_{3}+\| \nabla \widehat{u}\|^2_{3}\Big)|\overline{W}^N|^2_2\\
&\quad \  +C\big(|\varphi_2\nabla^2 w_2|_\infty|\overline{w}^N|_2+|\nabla^2 w_2|_3 \big(|\varphi_1\nabla \overline{w}^N|_2 +|\nabla \varphi_1|_\infty|\overline{w}^N|_2\big)\big)|\overline{\varphi}^N|_2+I^2_N,
\end{split}
\end{equation}
where the error terms $I^1_N$-$I^2_N$ are given and estimated  by
\begin{equation*}
\begin{split}
I^1_N=&\int_{N\leq |x|\leq 2N} \frac{1}{N}\Big((|w_1|+|\widehat{u}|)|\overline{\varphi}|^2+|\overline{\varphi}||\overline{w}||\varphi_1|\Big) \text{d}x\\
\leq & C\Big((|w_1|_\infty+|\nabla \widehat{u}|_\infty)\|\overline{\varphi}\|_{L^2(\mathbb{R}^3/B_N)}+|\varphi_1|_\infty\|\overline{w}\|_{L^2(\mathbb{R}^3/B_N)}\Big)\|\overline{\varphi}\|_{L^2(\mathbb{R}^3/B_N)},\\
I^2_N=&\int_{N\leq |x|\leq 2N}\frac{1}{N} \Big((|W_1|+|\widehat{u}|)|\overline{W}|^2+|\varphi_1|^2 |\overline{w}|^2+| \varphi_1|^2|\overline{w}||\nabla \overline{w}|\Big) \text{d}x\\
&+\int_{N\leq |x|\leq 2N} \frac{1}{N}\Big(|\varphi_1+\varphi_2| |\overline{\varphi}||\nabla w_2|+|\nabla \varphi_1||\varphi_1\nabla \overline{w}|\Big)|\overline{w}| \text{d}x\\
\leq & C(|W_1|_\infty+|\nabla \widehat{u}|_\infty)\|\overline{W}\|^2_{L^2(\mathbb{R}^3/B_N)}+C|\varphi_1|^2_\infty\|\overline{w}\|^2_{L^2(\mathbb{R}^3/B_N)}\\
&+C(|\varphi_1|_\infty+|\nabla \varphi_1|_\infty)\|\overline{w}\|_{L^2(\mathbb{R}^3/B_N)}|\varphi_1\nabla \overline{w}|_2\\
&+C|\varphi_1+\varphi_2|_\infty|w_2|_\infty\|\overline{w}\|_{L^2(\mathbb{R}^3/B_N)}\|\overline{\varphi}\|_{L^2(\mathbb{R}^3/B_N)},
\end{split}
\end{equation*}
for $0\leq t \leq T$.

Using the same arguments as in the derivation of (\ref{varcha})-(\ref{gogo1}), and letting

$$
\Lambda^N(t)=|\overline{W}^N(t)|^2_{2}+\frac{\alpha}{2C}|\overline{\varphi}^N(t)|^2_{2},
$$
one can  have
\begin{equation}\label{jiewei}
\begin{cases}
\displaystyle
\frac{d}{dt}\Lambda^N(t)+\frac{\alpha}{2}|\varphi_1\nabla \overline{w}^N(t)|^2_2\leq J(t)\Lambda^N(t)+I^1_N(t)+I^2_N(t),\\[6pt]
\displaystyle
\int_{0}^{t}(J(s)+I_N(s))ds\leq C_0\quad \quad  \quad  \quad  \quad   \quad  \text{for} \quad 0\leq t\leq T,
\end{cases}
\end{equation}
where the constant   $C_0>0 $ is  independent of $N$.

It follows from 
$$
\int_0^t \int (I^1_N+I^2_N)\text{d}t\rightarrow 0, \quad \text{as} \quad N\rightarrow +\infty,
$$
and Gronwall's inequality that $\overline{\varphi}=\overline{\phi}=\overline{w}=0$.
Then the uniqueness is obtained.

\subsection{Time continuity} This follows from the uniform estimates (\ref{decay2}) and equations in (\ref{li47-1}).

\section{Proof of  Theorem \ref{thglobal}}
Based on Theorem \ref{ths1B}, now we are ready to give the global well-posedness of the regular solution to the original problem (\ref{eq:1.1})-(\ref{far}), i.e., the proof of Theorem \ref{thglobal}. Moreover, we will show that this regular solution  satisfies  (\ref{eq:1.1}) classically in positive time $(0,T_*]$ when $1<\min(\gamma,\delta)\leq 3$.

\begin{proof}

First,  Theorem \ref{ths1B} shows  that
there exists   a unique global classical solution $(\varphi,\phi,w)$ satisfying (\ref{decay2}) and
\begin{equation}\label{reg2}
(\rho^{\frac{\delta-1}{2}},\rho^{\frac{\gamma-1}{2}} )=(\varphi,\phi)\in C^1((0,T)\times \mathbb{R}^3),\quad \text{and} \quad (u,\nabla u)\in C((0,T)\times \mathbb{R}^3).
\end{equation}

Second, in terms of  $(\varphi, \phi,u)$, one has 
\begin{equation}
\begin{cases}
\label{fanzheng}
\displaystyle
\phi_t+u\cdot \nabla \phi+\frac{\gamma-1}{2}\phi \text{div} u=0,\\[10pt]
\displaystyle
u_t+u\cdot\nabla u +\frac{\gamma-1}{2}\phi\nabla \phi+\varphi^2 Lu= \nabla \varphi^2 \cdot Q(u).
 \end{cases}
\end{equation}

\textbf{Case 1:} \text{$1<\min\{\gamma,\delta\}\leq 3$}. We assume that $1<\gamma\leq 3$, and the other cases can be dealt with similarly.
Since $\rho=\phi^{\frac{2}{\gamma-1}}$ and
$\frac{2}{\gamma-1}\geq 1$, so 
$$\rho \in C^1((0,T)\times\mathbb{R}^3).$$

Multiplying $(\ref{fanzheng})_1$ by
$
\frac{\partial \rho}{\partial \phi}(t,x)=\frac{2}{\gamma-1}\phi^{\frac{3-\gamma}{\gamma-1}}(t,x)\in C((0,T)\times \mathbb{R}^3)
$ on both sides yields  the continuity equation, $(\ref{eq:1.1})_1$.

While multiplying $(\ref{fanzheng})_2$ by
$
\phi^{\frac{2}{\gamma-1}}=\rho(t,x)\in C^1((0,T)\times \mathbb{R}^3)
$ on both sides gives  the momentum equations, $(\ref{eq:1.1})_2$.

Thus, $(\rho,u)$  solves  the Cauchy problem (\ref{eq:1.1})-(\ref{far}) in the classical sense.

\textbf{Case 2:} \text{$\min\{\gamma,\delta\}> 3$}.  For definiteness, we assume that $\gamma>3$, and the other cases could be dealt with similarly. Since $\phi=\rho^{\frac{\gamma-1}{2}}$ and
$\frac{2}{\gamma-1}>0$, so
$\rho \in C((0,T)\times\mathbb{R}^3)$.

It follows from $(\ref{fanzheng})_2$ and $\frac{\gamma-1}{2}>1$ that 
\begin{equation}\label{tuidao}\rho_t+\text{div}(\rho u)=0,\quad \text{when} \quad \rho(t,x)>0.
\end{equation}
Now  the whole space could be divide into two domains: the vacuum domain $V_0$ and its complement  $V_p(t)$.
Then it holds that for any smooth $f$,
\begin{equation}\label{tuidao1}
\int_0^t\int_{V_p(s)} \big( \rho  f_t+ \rho u\cdot \nabla f\big)\text{d}x\text{d}s=\int_{V_p(0)} \rho_0 f(0,x) \text{d}x,
\end{equation}
which means that
\begin{equation}\label{tuidao2}
\int_0^t \int_{\mathbb{R}^3} \big( \rho  f_t+ \rho u\cdot \nabla f\big)\text{d}x\text{d}t=\int_{\mathbb{R}^3}  \rho_0 f(0,x) \text{d}x.
\end{equation}

Multiplying $(\ref{fanzheng})_2$ by
$
\phi^{\frac{2}{\gamma-1}}=\rho(t,x)
$ gives  the momentum equations, $(\ref{eq:1.1})_2$.

Thus, $(\rho,u)$  satisfies the Cauchy problem (\ref{eq:1.1})-(\ref{far}) in the sense of distributions.

Finally, it is easy to show that 
$$
\rho(t,x)\geq 0, \ \forall (t,x)\in [0,T]\times \mathbb{R}^3.
$$
That is to say, $(\rho,u)$ satisfies the Cauchy problem (\ref{eq:1.1})-(\ref{far}) in the sense of distributions   and has the regularities shown in Definition \ref{d1}, which means that  the Cauchy problem(\ref{eq:1.1})-(\ref{far}) has a unique regular solution $(\rho,u)$.

\end{proof}

\section{Proof of Theorem \ref{thglobal3}}

In this section, we prove Theorem \ref{thglobal3} by modifying the proof of Theorem  \ref{thglobal}.
For the Cauchy problem  (\ref{eq:1.1})-(\ref{eq:1.2}) with (\ref{initial})-(\ref{far}), if
$
\text{div}\mathbb{T}=\alpha \rho^\delta\triangle u
$,
 it can be rewritten into

\begin{equation}\label{li47-2*}
\begin{cases}
\displaystyle
\varphi_t+w\cdot\nabla\varphi+\frac{\delta-1}{2}\varphi\text{div} w=- \widehat{u} \cdot\nabla\varphi-\frac{\delta-1}{2}\varphi\text{div}  \widehat{u},\\[5pt]
\displaystyle
W_t+\sum_{j=1}^3A_j(W) \partial_j W-\varphi^2\triangle w=G_*(W, \varphi,  \widehat{u}),\\[5pt]
(\varphi, \phi, w)(t=0,x)=(\varphi_0, \phi_0, 0),\quad x\in \mathbb{R}^3,\\[5pt]
(\varphi, \phi, w)\rightarrow (0,0,0) \quad\quad   \text{as}\quad \quad  |x|\rightarrow \infty \quad \text{for} \quad  t\geq 0,
 \end{cases}
\end{equation}
where
\begin{equation} \label{li47-2GH}
\begin{split}
G_*(W, \varphi,  \widehat{u})=&-B(\nabla  \widehat{u},W)-\sum_{j=1}^3  \widehat{u}^{(j)}\partial_j W-\left(\begin{array}{c}0\\[8pt]
 \varphi^2 \triangle  \widehat{u}
\end{array}\right).
\end{split}
\end{equation}

Let $\gamma_k=k-\frac{5}{2}$ and $\delta_k=k-m$, where $m$ will be determined in the end of this section, and $Z(t)$ be defined in (\ref{fg}). Then one has
\begin{lemma}\label{Cl2}
There exist some positive constants $b_m$,  such that 
\begin{equation}\label{Ceq:2.14}
\frac{\text{d} Z}{\text{d} t}(t)+\frac{b_m }{1+t}Z(t)\leq
 \begin{cases}
C(1+t)^{2m-5}Z^3(t)+C_0(1+t)^{2m-9}Z \quad \ \  \text{if}  \ \ m\in [3.5,4),\\[8pt]
C(1+t)^{2m-5}Z^3(t)+C_0(1+t)^{5-2m}Z   \quad  \ \  \text{if} \ \   m\in (3,3.5),
 \end{cases}
\end{equation}
where $b_m$ is given by (\ref{Jeq:2.6q}).
\end{lemma}
The proof is similar to that of Lemma \ref{l2} with $Q_k=0$. So the  details are omitted here.

\subsection{} We treat the case $m\in [3.5,4)$ first.
\begin{lemma}\label{yushi1}
Let  $\delta>1$, $\gamma>4/3$ and $m\in [3.5,4)$. For the following problem
\begin{equation}\label{eq:2.16A1a}
\begin{cases}
\begin{split}
&\frac{\text{d} Z(t)}{\text{d} t}+\frac{b_m }{1+t}Z(t)=C_1(1+t)^{2m-5}Z^3(t)+C_2(1+t)^{2m-9}Z,\\[6pt]
&Z(x, 0)=Z_{0},
\end{split}
\end{cases}
\end{equation}
there exists a constant $\Lambda$ such that  $Z(t)$ is globally well-defined in $[0,+\infty)$ if $Z_0\leq \Lambda$.
\end{lemma}

\begin{proof}

Set $M=2m-5-2b_m$. Due to Proposition \ref{ode1} and its proof,
a  sufficient condition for the global existence of solution is 
$$
M<-1, \quad \text{and} \quad 2m-8<0,
$$
which requires that 
 \begin{equation}\label{xishu2}
M<
-1\quad   \text{if}  \ \ \gamma\geq \frac{5}{3},\quad  \text{or} \ \   \frac{4}{3}\leq  \frac{2}{3}m-1<\gamma <\frac{5}{3}.
 \end{equation}

Therefore, by choosing $Z_{0}$ small enough, one  can obtain the global existence of the solution to  (\ref{eq:2.16A1a}).  Moreover,
 \begin{equation}\label{Req:3.30}
  Z(t)\leq C_0(1+t)^{-b_m} \qquad \text{ for all }\ t\geq 0.
 \end{equation}
\end{proof}

Next we prove Theorem \ref{thglobal3} for the cases $\gamma> 4/3$. First, according to Lemma \ref{Cl2}, when $m\in [3.5,4)$,  the following inequality holds
\begin{equation}\label{eq:2.141}
\begin{split}
\frac{\text{d} Z}{\text{d} t}(t)+\frac{b_m }{1+t}Z(t)
\leq & C(1+t)^{2m-5}Z^3(t)+C_0(1+t)^{2m-9}Z.
\end{split}
\end{equation}
Then it follows from Lemma \ref{yushi1}  that there exists a constant $\Lambda$ such that  $Z(t)$ is globally well-defined in $[0,+\infty)$ if $Z_0\leq \Lambda$, and
(\ref{Req:3.30})  holds.

 According to  Lemma \ref{Azongjie1} and its proof with $Q_k=0$, it is easy to see that 
\begin{equation}\label{higho}
 \alpha \sum_{k=1}^3\int_0^t (1+t)^{2\gamma_k}\int   | \varphi\nabla^{k+1} w|^2 \text{d}s\leq C_0, \quad \text{for} \quad t\geq 0.
\end{equation}

The rest of the proof  is similar to that for Theorem \ref{thglobal}, and so is omitted.


\subsection{}We now  treat the case $m\in (3,3.5)$.

\begin{lemma}\label{yushi2} Let $\delta>1$, $\gamma>1$ and
$m\in (3,3.5)$.
For the following problem
\begin{equation}\label{eq:2.16A2}
\begin{cases}
\begin{split}
&\frac{\text{d} Z}{\text{d} t}(t)+\frac{b_m}{1+t}Z(t)=C(1+t)^{2m-5}Z^3(t)+C_0(1+t)^{5-2m}Z,\\[6pt]
&Z(x, 0)=Z_{0},
\end{split}
\end{cases}
\end{equation}
there exists a constant $\Lambda$ such that  $Z(t)$ is globally well-defined in $[0,+\infty)$ if $Z_0\leq \Lambda$.
\end{lemma}

\begin{proof}
Set $M=2m-5-2b_m$. Due to Proposition \ref{ode1} and its proof,  a sufficient condition for a global solution with small data is 
$$
M<-1, \quad \text{and} \quad 6-2m<0,
$$
which means that $m$ must belong to $(3,3.5)$, and requires that 
 \begin{equation}\label{xishu22}
M<
-1\quad \text{if}  \ \ \gamma\geq \frac{5}{3},\quad   \text{or} \ \  1< \frac{2m}{3}-1 <\gamma <\frac{5}{3}.
 \end{equation}

Therefore, by choosing $Z_{0}$ small enough, one  can obtain the global existence of the solution to  (\ref{eq:2.16A2}).  Moreover, (\ref{Req:3.30}) still holds.
\end{proof}

Next we prove  Theorem \ref{thglobal3} for the cases $\gamma >1$. First, according to Lemma \ref{Cl2}, when $m\in (3,3.5)$, it holds that 
\begin{equation}\label{eq:2.142}
\begin{split}
\frac{\text{d} Z}{\text{d} t}(t)+\frac{b_m }{1+t}Z(t)
\leq & C(1+t)^{2m-5}Z^3(t)
+C_0(1+t)^{5-2m}Z.
\end{split}
\end{equation}
Then from Lemma \ref{yushi2},  there exists a constant $\Lambda$ such that  $Z(t)$ is globally well-defined in $[0,+\infty)$ if $Z_0\leq \Lambda$, and (\ref{Req:3.30})  holds.

 According to  Lemma \ref{Azongjie1} and its proof with $Q_k=0$, it is easy to see  that (\ref{higho}) holds.

The rest of the proof  is similar to that for Theorem \ref{thglobal},  and so is omitted.

\section{Appendix}
In the first part of this appendix, we list some  lemmas  which were used frequently in the previous sections.  The rest of the appendix will be devoted to show the  proofs of Propositions \ref{p1}-\ref{ode1} and some special Sobolev inequalities.

\subsection{Some lemmas}
The first one is the  well-known Gagliardo-Nirenberg inequality.
\begin{lemma}\cite{oar}\label{lem2as}\
For $p\in [2,6]$, $q\in (1,\infty)$, and $r\in (3,\infty)$, there exists some generic constant $C> 0$ that may depend on $q$ and $r$ such that for
$$f\in H^1(\mathbb{R}^3),\quad \text{and} \quad  g\in L^q(\mathbb{R}^3)\cap D^{1,r}(\mathbb{R}^3),$$
it holds that 
\begin{equation}\label{33}
\begin{split}
&|f|^p_p \leq C |f|^{(6-p)/2}_2 |\nabla f|^{(3p-6)/2}_2,\quad  |g|_\infty\leq C |g|^{q(r-3)/(3r+q(r-3))}_q |\nabla g|^{3r/(3r+q(r-3))}_r.
\end{split}
\end{equation}
\end{lemma}
Some common versions of this inequality can be written as
\begin{equation}\label{ine}\begin{split}
|u|_6\leq C|u|_{D^1},\quad |u|_{\infty}\leq C\|\nabla u\|_{1}, \quad |u|_{\infty}\leq C\|u\|_{W^{1,r}},\quad \text{for} \quad r>3.
\end{split}
\end{equation}

The Sobolev imbedding theorem yields
\begin{lemma}\cite{oar}\label{lem2-2}\
Let $p>3/2 $ and $f\in H^p(\mathbb{R}^3)$. Then
\begin{equation}\label{33-2}
\begin{split}
&|f|_{\infty} \leq C |f|^{1-\frac{3}{2p}}_2 |\nabla^p f|^{\frac{3}{2p}}_2,
\end{split}
\end{equation}
where $C$ is a positive constant that may depend on $p$.
\end{lemma}

The next one  can be found in Majda \cite{amj}.  
\begin{lemma}\cite{amj}\label{zhen1}
Let  $r$, $a$ and $b$ be constants such that 
$$\frac{1}{r}=\frac{1}{a}+\frac{1}{b},\quad \text{and} \quad 1\leq a,\ b, \ r\leq \infty.$$  $ \forall s\geq 1$, if $f, g \in W^{s,a} \cap  W^{s,b}(\mathbb{R}^3)$, then it holds that
\begin{equation}\begin{split}\label{ku11}
&|\nabla^s(fg)-f \nabla^s g|_r\leq C_s\big(|\nabla f|_a |\nabla^{s-1}g|_b+|\nabla^s f|_b|g|_a\big),
\end{split}
\end{equation}
\begin{equation}\begin{split}\label{ku22}
&|\nabla^s(fg)-f \nabla^s g|_r\leq C_s\big(|\nabla f|_a |\nabla^{s-1}g|_b+|\nabla^s f|_a|g|_b\big),
\end{split}
\end{equation}
where $C_s> 0$ is a constant depending only on $s$,  and $\nabla^s f$ ($s>1$) is the set of  all $\nabla^\zeta_x f$  with $|\zeta|=s$. Here $\zeta=(\zeta_1,\zeta_2,\zeta_3)\in \mathbb{R}^3$ is a multi-index.
\end{lemma}

Next,  the interpolation estimate, product estimate,  composite function estimate and so on are given  in the following four lemmas.
\begin{lemma}\cite{amj}\label{gag111}
Let  $u\in H^s$, then for any $s'\in[0,s]$,  there exists  a constant $C_s$  depending only on $s$ such that
$$
\|u\|_{s'} \leq C_s \|u\|^{1-\frac{s'}{s}}_0 \|u\|^{\frac{s'}{s}}_s.
$$
\end{lemma}

\begin{lemma}\cite{oar}\label{lem2-1}\
Let $r\geq 0$, $i\in [0,r]$,  and $f\in L^\infty\cap H^r$. Then $\nabla^i f \in L^{2r/i}$, and there  some generic constants $C_{i,r}> 0$ such that
\begin{equation}\label{33-1}
\begin{split}
&|\nabla^i f|_{2r/i} \leq C_{i,r} |f|^{1-i/r}_\infty |\nabla^r f|^{i/r}_2.
\end{split}
\end{equation}
\end{lemma}

\begin{lemma}\cite{amj}\label{gag113}
Let  functions $u,\ v \in H^s$ and $s>\frac{3}{2}$, then  $u\cdot v \in H^s$,  and  there exists  a constant $C_s$  depending only on $s$ such that
$$
\|uv\|_{s} \leq C_s \|u\|_s \|v\|_s.
$$
\end{lemma}

\begin{lemma}\cite{amj}
\begin{enumerate}
\item For functions $f,\ g \in H^s \cap L^\infty$ and $|\nu|\leq s$,   there exists  a constant $C_s$  depending only  on $s$ such that
\begin{equation}\label{liu01}
\begin{split}
\|\nabla^\nu (fg)\|_s\leq C_s(|f|_{\infty}|\nabla^s g|_2+|g|_{\infty}|\nabla^s f|_{2}).
\end{split}
\end{equation}

\item Assume that  $g(u)$ is a smooth vector-valued function on $\Omega$, $u(x)$ is a continuous function with $u\in H^s \cap L^\infty$, $u(x)\in \Omega_1$ and $\overline{\Omega}_1 \subseteq \Omega$. Then for $s\geq 1$,   there exists  a constant $C_s$  depending only on $s$ such that
\begin{equation}\label{liu02}
\begin{split}
|\nabla^s g(u)|_2\leq C_s\Big \|\frac{\partial g}{\partial u}\Big\|_{s-1, \overline{\Omega}_1}|u|^{s-1}_{\infty}|\nabla^s u|_{2}.
\end{split}
\end{equation}
\end{enumerate}
\end{lemma}

As a consequence of the  Aubin-Lions Lemma, one has (c.f. \cite{jm}),
\begin{lemma}\cite{jm}\label{aubin} Let $X_0$, $X$ and $X_1$ be three Banach spaces satisfying $X_0\subset X\subset X_1$. Suppose that $X_0$ is compactly embedded in $X$ and that $X$ is continuously embedded in $X_1$.
\begin{enumerate}
\item Let $G$ be bounded in $L^p(0,T;X_0)$ with $1\leq p < +\infty$, and $\frac{\partial G}{\partial t}$ be bounded in $L^1(0,T;X_1)$. Then $G$ is relatively compact in $L^p(0,T;X)$.\\[2pt]

\item Let $F$ be bounded in $L^\infty(0,T;X_0)$  and $\frac{\partial F}{\partial t}$ be bounded in $L^q(0,T;X_1)$ with $q>1$. Then $F$ is relatively compact in $C(0,T;X)$.
\end{enumerate}
\end{lemma}

The following lemma is  useful to improve weak convergence to strong convergence.
\begin{lemma}\cite{amj}\label{zheng5}
If the function sequence $\{w_n\}^\infty_{n=1}$ converges weakly  to $w$ in a Hilbert space $X$, then it converges strongly to $w$ in $X$ if and only if
$$
\|w\|_X \geq \lim \text{sup}_{n \rightarrow \infty} \|w_n\|_X.
$$
\end{lemma}

The last lemma will be used in the proof shown in Section 5.
\begin{lemma}\label{changeA1} Let $(\varphi,\phi,w)$ be the  regular the solution  to the Cauchy problem (\ref{li47-1}), and  $Z$ be the time weighted energy defined in (\ref{fg}). One has
\begin{equation}\begin{split}\label{changea22}
|\varphi\nabla^3\text{div}w|^2_2\leq |\varphi\nabla^4 w|^2_2+J^*,
\end{split}
\end{equation}
where the term $J^*$ can be controlled by 
$$
J^*\leq \eta |\varphi \nabla^4 w|^2_2+C(\eta)(1+t)^{2m-5-2\gamma_3}Z^4,$$
for any constant $\eta>0$ small enough, and the constant $C(\eta)>0$.
\end{lemma}
\begin{proof}
According to the definition of $\text{div}$, one directly has 
\begin{equation}\label{changeA2}
\begin{split}
|\varphi\nabla^3\text{div}w|^2_2=\sum_{i=1}^3 |\varphi\nabla^3\partial_iw^{(i)}|^2_2+2\sum_{i,j=1,\ i<j}^3 \int \varphi^2 \nabla^3\partial_iw^{(i)}\cdot  \nabla^3\partial_jw^{(j)}.
\end{split}
\end{equation}
Via the integration by parts, one can obtain that 
\begin{equation}\label{changeA3}
\begin{split}
J_{ij} =&\int \varphi^2 \nabla^3\partial_iw^{(i)}\cdot  \nabla^3\partial_jw^{(j)}
= \int \varphi^2 \nabla^3\partial_iw^{(j)}\cdot  \nabla^3\partial_jw^{(i)}\\
&\quad +\int \big(\partial_i \varphi^2 \nabla^3 \partial_j w^{(i)} \nabla^3  w^{(j)}-\partial_j \varphi^2 \nabla^3 \partial_i w^{(i)} \nabla^3  w^{(j)} \big).
\end{split}
\end{equation}
It is easy to see that 
\begin{equation}\label{changeA4}
\begin{split}
J^*_{ij} =& \int \big(\partial_i \varphi^2 \nabla^3 \partial_j w^{(i)} \nabla^3  w^{(j)}-\partial_j \varphi^2 \nabla^3 \partial_i w^{(i)} \nabla^3  w^{(j)} \big)\\
\leq & C|\nabla \varphi|_\infty|\varphi \nabla^4 w|_2|\nabla^3 w|_2\\
\leq & \eta |\varphi \nabla^4 w|^2_2+C(\eta)(1+t)^{2m-5-2\gamma_3}Z^4,
\end{split}
\end{equation}
for any constant $\eta>0$ small enough, and the constant $C(\eta)>0$,  which, along with (\ref{changeA2})-(\ref{changeA3}), quickly implies (\ref{changea22}).

\end{proof}

\subsection{Proof of Proposition \ref{p1}} Here in this subsection, set
$$
\|f\|_{\Xi}=|\nabla f|_\infty+\|\nabla^2 f\|_{m-2} \ \  \text{for \ any} \ m\geq 2.
$$
\begin{proof} \textbf{(I)} Let $G(t,x)=\nabla \widehat{u}(t,x)$ and $G_0(x_0)=\nabla u_0(x_0)$.
Then
$$
G(t,X(t;x_0))=\big(\mathbb{I}_d+tG_0(x_0)\big)^{-1}G_0(x_0).
$$
Note  that $
\big(\mathbb{I}_d+tG_0(x_0)\big)^{-1}=\big(\text{det}(\mathbb{I}_d+tG_0)\big)^{-1}(\text{adj} (\mathbb{I}_d+tG_0))^\top$,
where $\text{adj} (\mathbb{I}_d+tG_0)$ stands for the adjugate  of $(\mathbb{I}_d+tG_0)$. Then it holds that 
\begin{equation}\label{wuqiong}
|\nabla_x \widehat{u}|_\infty\leq \frac{(1+t|\nabla u_0|_\infty)^{d-1}}{(1+t\kappa)^d}|\nabla_{x_0} u_0|_\infty.
\end{equation}
Rewrite
$$
G(t,X(t;x_0))= \frac{1}{1+t}\mathbb{I}_d+\frac{1}{(1+t)^2}K(t,x_0),
$$
where $K(t,x_0)=(1+t)^2(\mathbb{I}_d+tG_0)^{-1} G_0-(1+t)\mathbb{I}_d$.

Since $
G^{-1}_0=\big(\text{det}(G_0)\big)^{-1}(\text{adj} G_0)^\top
$,  so
$$
|G^{-1}_0|_\infty\leq C\kappa^{-d}|G_0|^{d-1}_\infty.
$$

Then for $t$ large enough, one has $|t^{-1}G^{-1} _0(x_0)|<1$ for all $x_0$, and
\begin{equation*}\begin{split}
K(t,x_0)=& \frac{(t+1)^2}{t} (\mathbb{I}_d+t^{-1}G^{-1}_0)^{-1}-(1+t)\mathbb{I}_d\\
=& \frac{(t+1)^2}{t} \Big(\mathbb{I}_d-\frac{G^{-1}_0}{t}+O\Big(\frac{1}{t^2}\Big)\Big)-(1+t)\mathbb{I}_d
= \frac{(t+1)}{t} \mathbb{I}_d- \frac{(t+1)^2}{t^2} G^{-1} _0+O\Big(\frac{1}{t}\Big).
\end{split}
\end{equation*}

\begin{remark}
Note that
\begin{equation*}\begin{split}
G=& \frac{1}{t} \Big( \mathbb{I}_d- \big(\mathbb{I}_d+tG_0(x_0)\big)^{-1}\Big).
\end{split}
\end{equation*}
Then  any eigenvalue $\lambda_{G}$ of $G$  has the following form
$$
\lambda_{G}=\frac{1}{t} \Big( 1- \big(1+t\lambda\big)^{-1}\Big)=\frac{\lambda}{1+t\lambda}>0,
$$
where $\lambda$ is an eigenvalue of $G_0$.

\end{remark}

\textbf{(II)}
 Let $H(t,x_0)=G(t,X(t;x_0))$.  By induction, it is easy to  show that, for $k\geq 1$:
$$
\nabla^k_{x_0} H=(\mathbb{I}_d+tG_0)^{-1}\Lambda_k (\mathbb{I}_d+tG_0)^{-1},
$$
where $\Lambda_k$ is a sum of products of $t(\mathbb{I}_d+tG_0)^{-1}$ and $\nabla^j G_0, \ j\in \{0,1,...,k\}$, appearing $\beta_j$ times with $\sum_j j\beta_j=k$.

On the one hand, by induction, it is also  easy to  show that, for $k\geq 1$:
$$
\nabla^k_{x_0} H(t,x_0) =\sum_{j=1}^k \nabla^j_{x} G(t,X(t;x_0)) \Big(\sum_{1\leq k_i\leq k} \nabla^{k_1}_{x_0}X\otimes ...\otimes \nabla^{k_j}_{x_0}X\Big)
$$
with $\sum_{i=1}^j k_i=k$, and
\begin{equation*}\begin{split}
\nabla_{x_0} X=& \mathbb{I}_d+tG_0(x_0),  \quad \nabla^{l}_{x_0}X=t\nabla^{l-1} G_0(x_0),\quad \text{for} \quad l\geq 2.
\end{split}
\end{equation*}

On the other hand, one can also  show that for all $j\geq 1$ by induction:   \textbf{IH(j)}.
$\nabla^j_x G$ is a sum of terms which are products in a certain order of : $(\mathbb{I}_d+tG_0)^{-1}$, $t\mathbb{I}_d$,  $\mathbb{I}_d+tG_0$ or  $\nabla^l G_0$ appearing $\beta_l$ times, with $\sum_l l\beta_l=j$. Moreover, the $L^\infty$-norm of the terms with $t$ is bounded by a constant times $(1+t)^{-(j+2)}$, and 
$$
|\nabla^j_x G(X(t;\cdot),t)|_2\leq C_j(1+ t)^{-(j+2)},\quad \text{with} \quad C_j=C(\kappa, j, \|u_0\|_{\Xi}).
$$

Suppose \textbf{IH(k-1)}, then
\begin{equation*}\begin{split}
\nabla^k_{x} G(t,X(t;x_0)) =&\Big(\nabla^k_{x_0}H(t,x_0)-\sum_{j=1}^{k-1} \nabla^j_{x} G(t,X(t;x_0)) \Big(\sum_{1\leq k_i\leq k-1} \nabla^{k_1}_{x_0}X\otimes ...\otimes \nabla^{k_j}_{x_0}X\Big)\Big)\\
& \odot \Big((\mathbb{I}_d+tG_0)^{-1}\Big)^{\otimes k}.
\end{split}
\end{equation*}

In the right-hand side term, one has the norm of the following terms to estimate:

\begin{itemize}
\item a)$J_1=(\mathbb{I}_d+tG_0)^{-1}\Lambda_k (\mathbb{I}_d+tG_0)^{-k-1}$;
\item  b)$J_2=\nabla^j_{x} G(\mathbb{I}_d+tG_0)^{j-s} \amalg_{k_i \neq 1}t\nabla^{k_i-1}G_0(\mathbb{I}_d+tG_0)^{-k}$,\ \  $\sum_{k_i\neq 1}(k_i-1)=k-j$,
\end{itemize}
where $s$ is the number of $k_j\neq 1$.

For each of these terms, it needs to  apply first the induction hypothesis and  consider the $L^\infty$-norm in space for the terms with $t$. It needs to  use \textbf{IH(j)} for $j\leq k-1$ to show that $b)$ is a product of
$(\nabla^j G_0)^{\beta_j}$
with $\sum_j j\beta_j=k$ and of $t \mathbb{I}_d$, $\mathbb{I}_d+tG_0$,  $(\mathbb{I}_d+tG_0)^{-1}$, such that the $L^\infty$-norm of the terms with $t$ is bounded by a constant times $(1+ t)^{-(k+2)}$. Then one can  find an upper bound in $L^2$-norm for
$
\amalg_{1\leq j\leq k} \big(\nabla^j G_0\big)^{\beta_j}
$
with $\sum_j j\beta_j=k$, by using Gagliardo-Nirenberg inequality. Then one gets for $\sum_j j\beta_j=k$:
\begin{equation*}\begin{split}
|J_1|_2\leq  C(1+ t)^{-k-2}|\Lambda_k|_2\leq C(1+t\kappa)^{-k-2}|\amalg \big(\nabla^j G_0\big)^{\beta_j}|_2
\leq& C(1+t)^{-k-2}c^{k+1}_0,\\
|J_2|_2\leq  C(1+t)^{-k+j}|\nabla^j_{x} G \amalg_{k_i \neq 1}\nabla^{k_i-1}G_0|_2
\leq & C(1+ t)^{-k-2}c^{k+1}_0.
\end{split}
\end{equation*}
To conclude, one obtains the upper bound in \textbf{IH(k)} which depends on $\kappa$, $|G_0|_\infty$,  and $|\nabla^k G_0|_2$ for $1\leq k \leq m$, that is to say on $\kappa$ and $\|u_0\|_\Xi$.
Finally, it needs to make a change of variables to obtain:
$$
|\nabla^j_x G(t,\cdot)|_2 \leq (1+ t)^{\frac{d}{2}}|\nabla^j_x G(t,X(t;\cdot))|_2.
$$
 Therefore (ii) is true for $l\in \mathbb{N},\ l\geq 2$ since $\nabla_x G=\nabla^2 \widehat{u}$.  Then by interpolation one obtains the result for all $l\in \mathbb{R},\ l \geq 2$.

\textbf{(III)} Since $m-1>\frac{d}{2}$, one gets that $\nabla^2 u_0\in L^\infty$. Thus $\nabla^2 \widehat{u} \in L^\infty$, and 
\begin{equation}\label{erjieQ}
\nabla_x G(t,X(t; x_0))=-(\mathbb{I}_d+tG_0)^{-1}\nabla G_0(x_0) (\mathbb{I}_d+tG_0)^{-2}.
\end{equation}

Using the estimates obtained in \textbf{(I)}, we obtain:  $|\nabla_x G(t,X(t; x_0))|_\infty=O((1+ t)^{-3})$.

%

\textbf{(IV)}
If $u_0(0)=0$, then  it is obvious that $
 \widehat{u}(t,0)=0 $ for any $t\geq 0$, and
$$
 \widehat{u}(t,x)= \widehat{u}(t,0)+\nabla\widehat{u}(t,ax)\cdot x=\nabla\widehat{u}(t,ax)\cdot x,\ \  \text{for\ any} \ \  (t,x)\in [0,T]\times \mathbb{R}^3,
$$
where $a\in [0,1]$ is some constant.  Thus  (\ref{linear relation}) is proved.

\end{proof}

\subsection{Proof of Proposition \ref{ode1}}
\begin{proof}
The solution $Z(t)$ of the Cauchy problem (\ref{Beq:2.16A1}) can be solved as:
\begin{equation}
\displaystyle
Z(t)=\frac{(1+t)^{-b}\exp{\Big(\frac{C_2}{D_2+1}}\Big((1+t)^{D_2+1}-1\Big)\Big)}{\Big(Z^{-(a-1)}_{0}-(a-1)C_1\int^t_0 (1+s)^{M}\exp{\Big(\frac{(a-1)C_2}{D_2+1}}\Big((1+t)^{D_2+1}-1\Big)\Big)\text{d}s\Big)^{\frac{1}{a-1}}},
\end{equation}
where $M=D_1-(a-1)b<-1$.

Thus, according to (\ref{zhibiao1}),   $Z(t)$  is globally well-defined for $t\geq 0$ if and  only if
\begin{equation}\label{eq:3.29}
0<Z_{0}<\frac{1}{\Big((a-1)C_1\int^t_0 (1+s)^{M}\exp{\Big(\frac{(a-1)C_2}{D_2+1}}\Big((1+t)^{D_2+1}-1\Big)\Big)\text{d}s\Big)^{\frac{1}{a-1}}}.
\end{equation}

Therefore, by choosing $Z_{0}$ small enough, one can obtain the global existence of our Cauchy problem (\ref{Beq:2.16A1}).  
\end{proof}

\subsection{Some Sobolev inequalities}
It follows from Lemma \ref{lem2as} that 
\begin{equation}\label{qianru1}
\begin{split}
|\varphi\nabla w|_6\leq&  C|\varphi\nabla w|_{D^1}
\leq C(|\varphi\nabla^2 w|_2+|\nabla \varphi|_\infty\nabla w|_2),\\
|\varphi\nabla^2 w|_\infty\leq&  C|\varphi\nabla^2 w|^{\frac{1}{2}}_6|\nabla (\varphi\nabla^2 w)|^{\frac{1}{2}}_6\leq C|\varphi\nabla^2 w|^{\frac{1}{2}}_{D^1}|\nabla (\varphi\nabla^2 w)|^{\frac{1}{2}}_{D^1}\\
\leq& C\|\nabla (\varphi\nabla^2 w)\|_1
\leq C(\|\nabla\varphi\|_2\|\nabla^2w\|_1+|\varphi \nabla^4 w|_2).
\end{split}
\end{equation}

\bigskip

{\bf Acknowledgement:}
This research is partially supported by Zheng Ge Ru Foundation, Hong Kong RGC Earmarked Research Grants  CUHK 14300917, CUHK-14305315 and CUHK 4048/13P, NSFC/RGC Joint Research Scheme Grant N-CUHK 443-14, and a Focus Area Grant from the Chinese University of Hong Kong. Zhu's research is also  supported in part
by National Natural Science Foundation of China under grant 11231006,  Natural Science Foundation of Shanghai under grant 14ZR1423100,  Australian Research Council grant DP170100630 and Newton International Fellowships NF170015.

\end{document}